\documentclass[12pt]{amsart}

\usepackage[hmargin=0.8in,height=8.6in]{geometry}
\usepackage{amssymb,amsthm, times}
\usepackage{delarray,verbatim}
\usepackage{ifpdf}
\ifpdf
\usepackage[pdftex]{graphicx}
 \else
\usepackage[dvips]{graphicx}
 \fi

\usepackage{bm}
\usepackage{multicol}

\usepackage{ifpdf}
\usepackage{color}
\definecolor{webgreen}{rgb}{0,.5,0}
\definecolor{webbrown}{rgb}{.8,0,0}
\definecolor{emphcolor}{rgb}{0.95,0.95,0.95}

\usepackage{hyperref}
\hypersetup{%
          colorlinks=true,
          linkcolor=webbrown,
          filecolor=webbrown,
          citecolor=webgreen,
          breaklinks=true}
\ifpdf \hypersetup{pdftex,
            pdfstartview=FitH, 
            bookmarksopen=true,
            bookmarksnumbered=true
} \else \hypersetup{dvips} \fi

\newcommand{\lapinv}{\Phi(q)}

\newcommand {\B}{\mathcal{B}}

\numberwithin{equation}{section}

\newtheorem{theorem}{Theorem}[section]
\newtheorem{proposition}{Proposition}[section]

\newtheorem{remark}{Remark}[section]
\newtheorem{lemma}{Lemma}[section]

\newtheorem{assump}{Assumption}[section]

\numberwithin{remark}{section} \numberwithin{proposition}{section}
\numberwithin{corollary}{section}
\newcommand {\R}{\mathbb{R}}

\newcommand {\N}{\mathbb{N}}
\newcommand {\p}{\mathbb{P}}

\newcommand {\E}{\mathbb{E}}

\newcommand{\diff}{{\rm d}}

\newcommand{\lev}{L\'{e}vy }

\title[Games of singular control and stopping]{Games of singular control and stopping driven by spectrally one-sided L\'{E}VY processes}
\thanks{This version: \today. }
\thanks{$*$\, Centro de Investigaci\'on en Matem\'aticas, Apartado Postal 402, Guanajuato GTO 36000, Mexico. Email: dher@cimat.mx}
\thanks{$\dagger$\, (corresponding author) Department of Mathematics,
Faculty of Engineering Science, Kansai University, 3-3-35 Yamate-cho, Suita-shi, Osaka 564-8680, Japan. Email: \mbox{{\em
kyamazak@kansai-u.ac.jp}}.  Tel: +81-6-6368-1527. }

\author[D.\ Hern\'andez-Hern\'andez]{Daniel Hern\'andez-Hern\'andez$^*$}
\author[K. Yamazaki]{Kazutoshi Yamazaki$^\dagger$}
\date{}
\begin{document}

\begin{abstract}
We study a zero-sum game where the evolution of a spectrally one-sided \lev process  is modified by a singular controller and is terminated by the stopper.  The singular controller minimizes the expected values of running, controlling and terminal costs while the stopper maximizes them.   Using fluctuation theory and scale functions, we derive a saddle point and the value function of the game.  Numerical examples under phase-type \lev processes are also given.
\end{abstract}

\maketitle
{\noindent \small{\textbf{Keywords:}\,  controller-and-stopper games;  \lev processes; scale functions; continuous and smooth fit}}

\section{Introduction}
We study the \lev model of the zero-sum game between a singular controller and a stopper in the framework of Hern\'andez-Hern\'andez et al.\ \cite{Hernandez_Simon_Zervos_2012}.  Based on the information provided by a given \lev process $X$, the \emph{controller} modifies the evolution of $X$ while the \emph{stopper} decides to  terminate the game.  More precisely,  the former chooses a nondecreasing control process $\xi$ so as to minimize the expected costs that depend on the path of the controlled process $U^\xi := X + \xi$ while the latter chooses a stopping time $\tau$ to maximize them.  The costs consist of the running cost $\int_0^\tau e^{-qt} h (U_t^\xi) \diff t$, the controlling cost $\int_{[0,\tau]} e^{-qt} \diff \xi_t$  as well as the terminal cost $e^{-q\tau}g(U_\tau^\xi)$.  The objective is to identify a pair of strategies $(\xi^*, \tau^*)$ that constitute a {\it saddle point} (\emph{Nash equilibrium}). 

Games of singular control and stopping arise in problems where both players are interested  in maintaining the state of the system in some region. Applications of the problem studied here include the following example adapted from the monotone follower problems (see, e.g., \cite{Karatzas_Shreve_1984}): if the process $-X$ represents a random demand of some specific good, the controller wants to prevent the lack of the product, while the stopper wants to prevent excess of inventory.  Here, the process $\xi$ can represent the amount of accumulated capital   needed to meet the demand.   Other applications in the field of finance and insurance can be found in, among others, \cite{Karatzas_Wong_2000} and \cite{Bayraktar_Young_2011}.

Another potential application lies in robust control, where the decision maker wants to optimize the worst case expectation when there is an ambiguity in the model specification.   One way to model this is to include a ``malevolent second agent", who acts adversely to the decision maker (see the introduction of Hansen et al.\ \cite{Hansen_2006}).  This reduces to solving the two-player zero-sum game.  In the light of our problem, it can be seen as a robust optimization of a singular controller when there is an ambiguity in the time-horizon; this can also be seen as a robust optimal stopping problem in the existence of misspecification  of the underlying process.   Regarding this matter, we refer the readers to, among others, \cite{Hansen_2001, Hansen_2006} for the robust control problems written in terms of games.

The version of the controller-stopper problem this paper is focusing on was first studied by \cite{Hernandez_Simon_Zervos_2012} where they focus on a one-dimensional diffusion model.  They derived a set of variational inequalities that fully characterize the saddle point and the associated value function. They also gave examples that admit explicit solutions. Although the variational inequalities and relevant theories of singular control and optimal stopping problems are well studied, when their games are considered,  existing results are rather limited and there are very few cases that are known to be solved explicitly. For different formulations of the game between a controller and a stopper, we refer the reader to \cite{Maitra_Sudderth_1996} and \cite{Karatzas_Sudderth_2001},   and references therein. In the former they studied controller-stopper games in discrete time, while the latter considers the problem of optimally modifying the drift and  diffusion coefficients of a Brownian motion  using absolutely continuous controls.   See also \cite{Davis_Zervos_1994} for the solution of a  singular optimal control problem with an arbitrary stopping time.

In this paper, we consider the aforementioned problem focusing on the case when $X$ is a spectrally one-sided (spectrally negative or positive) \lev process.   Without the continuity of paths, the problem naturally tends to be more challenging because one needs to take into account the overshoots when the process   up-crosses or down-crosses a given barrier. Nonetheless, by taking advantage of the recent advances in the theory of \lev processes, this can be handled for spectrally one-sided L\'evy
processes using so-called \emph{scale functions}.  We show, under appropriate conditions on the functions $h$ and $g$, that a  saddle point and the associated value function can be identified.

 In order to achieve our goals, we take the following steps:
\begin{enumerate}
\item We first conjecture that the optimal strategies are of barrier-type; namely,  the controlled process and the stopping time at  equilibrium are a \emph{reflected \lev process} and a \emph{first up/down-crossing time} of a certain level, respectively.  By focusing on the set of these strategies, the associated expected cost functional is written in terms of the scale function.
\item We then identify the reflection-trigger level $a^*$ as well as the stopping-trigger level $b^*$ by using the continuous/smooth fit principle.  Namely, the values of $(a^*, b^*)$ are chosen so that the corresponding expected cost functional is continuous (and/or smooth)  at $a^*$ and $b^*$ simultaneously.  The smoothness of the resulting value functions differ according to the path variation of $X$ and whether a spectrally negative or positive \lev process is considered.
\item  We finally verify the optimality of the candidate value function by showing that it satisfies the variational inequality and that it is indeed a sufficient condition for optimality.
\end{enumerate}

These steps in the above solution methods are motivated by recent papers on optimal stopping/control problems for spectrally one-sided \lev processes.   In these papers, the scale function is commonly used as a proxy to the underlying process so as to write efficiently the expected cost corresponding to a barrier strategy.  The optimal barrier is then determined by the continuous/smooth fit condition together with the known continuity/smoothness property of the scale function that differs according to the path variation of the process.   The verification of optimality is carried out by the martingale (harmonic) property of the scale function.   These methods avoid the use of techniques of integro partial differential equations (IPDEs) to prove existence and smoothness of the solution of the associated HJB equation; this tends to be hard to solve except for special cases (e.g.\ when the process has i.i.d.\ exponential jumps).
 For papers regarding spectrally one-sided \lev processes and their applications, we refer the reader to, among others, \cite{Alili2005, Avram_2004, Leung_Yamazaki_2010} for optimal stopping problems in finance, \cite{Avram_et_al_2007,Bayraktar_2012,Kyprianou_Palmowski_2007,Loeffen_2008} for singular control  problems in insurance, and \cite{BauKyrianouPardo2011, Baurdoux2008,  Leung_Yamazaki_2011} for optimal stopping games.  To our best knowledge, this paper is the first attempt to solve the game between a controller and a stopper for spectrally one-sided \lev processes.

We take advantage of the fluctuation theory and scale function for spectrally negative and spectrally positive \lev cases. However these two cases are analyzed separately in different sections because the behavior of these processes is indeed significantly different.  We require different assumptions and in addition the smoothness of the derived value function is different.   There also is an interesting difference that while $a^* = b^*$ can happen for the spectrally positive \lev case, it cannot happen for the spectrally negative \lev case.  This is caused due to the fact that while the lower boundary of an open set is regular (see Definition 6.4 of \cite{Kyprianou_2006}) for a spectrally negative \lev process, it is not so for the case of  a spectrally positive \lev process when it has paths of bounded variation; we discuss more in detail in Remark \ref{remark_a_equals_b} below.  It should be remarked that there also is a similarity, for example, in that for both cases the value function becomes convex.

The derived saddle point and the associated value function are expressed in terms of the scale function for both the spectrally negative and positive cases.   Hence, computing these values reduce to computing the scale function. In general, this cannot be done explicitly, but there
are numerical procedures that allow one to approximate them, for example, by Laplace inversion \cite{Kuznetsov2013, Surya_2008} or by phase-type fitting \cite{Egami_Yamazaki_2010_2}.  In this paper, we  use the latter and illustrate how the values of $(a^*, b^*)$ and the value function can be computed.  We also show the impact of the parameters that describe the functions $h$ and $g$ on the value function.

The rest of the paper is organized as follows.   In Section \ref{section_formulation}, we give a mathematical formulation of the game. In Section \ref{section_preliminary},  we review the spectrally one-sided \lev process and its scale function.  In Sections \ref{section_spectrally_negative} and \ref{section_spectrally_positive},  respectively, using the method described above, the spectrally negative and positive formulation is solved.  In Section  \ref{section_numerics}, we give numerical results for both the spectrally negative and positive cases.
Throughout the paper, superscripts $x^+ := \lim_{y \downarrow x}$ and $x^-  := \lim_{y \uparrow x}$ are used to indicate the right and left limits, respectively.  The subscripts $x_+ := \max(x, 0)$ and $x_- := \max(-x, 0)$ are used to indicate positive and negative parts.

\section{Model and problem formulation} \label{section_formulation}
Let $(\Omega, \mathcal{F}, \p)$ be a probability space hosting a \lev process $X = \left\{X_t; t \geq 0 \right\}$. Let $\mathbb{P}_x$ be the conditional probability under which $X_0 = x$ (also let $\mathbb{P} \equiv \mathbb{P}_0$), and let $\mathbb{F}^0 := \left\{ \mathcal{F}^0_t: t \geq 0 \right\}$ be the  filtration generated by $X$. We denote by $\mathbb{F}:= \left\{ \mathcal{F}_t: t \geq 0 \right\}$  the completed filtration of $\mathbb{F}^0$
with the $\p$-null sets of $\mathcal{F}$.

The games analyzed in this paper consist of  two players. The controller chooses a  process $\xi\in \Xi$, where $\Xi$ denotes the set of \emph{nondecreasing} and  \emph{right-continuous}  $\mathbb{F}$-adapted processes with $\xi_{0^-} := 0$, while the stopper chooses the time $\tau \in \Upsilon$ among the set of $\mathbb{F}$-stopping times $\Upsilon$.  The controller minimizes and the stopper maximizes the common performance  criterion:
\begin{align*}
J(x;\xi, \tau) := \E_x \left[ \int_0^\tau e^{-qt} h(U_t^\xi) \diff t + \int_{[0,\tau]} e^{-qt} \diff \xi_t + e^{-q \tau} g(U_\tau^\xi) 1_{\{ \tau < \infty \}}\right], \quad x \in \R, \; \xi \in \Xi \; \textrm{and } \tau \in \Upsilon,
\end{align*}
where  $U^\xi$ is a right-continuous controlled process defined by
\begin{align*}
U_t^\xi := X_t + \xi_t, \quad t \geq 0.
\end{align*}
The problem is to show the existence of a \emph{saddle point}, or equivalently a \emph{Nash equilibrium}  $(\xi^*, \tau^*) \in \Xi \times \Upsilon$, such that $\tau^*$ is the \emph{best response} given $\xi^*$ while at the same time $\xi^*$ is the best response given $\tau^*$.

More precisely, we shall show under a suitable condition that a saddle point is given by a pair of \emph{barrier strategies} $(\xi^a, \tau_{a,b})$ for some $a < b$, where we define
\begin{align} \label{the_strategies}
\begin{split}
\xi^{a}_t &:= \sup_{0 \leq t' \leq t} (a-X_{t'}) \vee 0, \quad t \geq 0, \\
\tau_{a,b}&:= \inf \{ t \geq 0: U_t^{\xi^a} > b \}.
\end{split}
\end{align}
In this case, the controlled process $U^{\xi^a} = X + \xi^{a}$ is a reflected \lev process, whose fluctuation theory has been studied in, for example, \cite{Avram_et_al_2007,Pistorius_2004}; see Figures \ref{sample_path_neg} and \ref{sample_path_pos} below for its sample paths when $X$ is a spectrally one-sided \lev process. The corresponding performance criterion reduces to
\begin{align} \label{J_a_b_def}
J_{a,b}(x) := J(x;\xi^a, \tau_{a,b}) = \E_x \left[ \int_0^{\tau_{a,b}} e^{-qt} h(U_t^{\xi^a}) \diff t + \int_{[0,\tau_{a,b}]} e^{-q t} \diff \xi_t^a + e^{-q \tau_{a,b}} g(U_{\tau_{a,b}}^{\xi^a}) 1_{\{ \tau_{a,b}< \infty \}} \right].
\end{align}
The objective is to show, upon appropriate choices of $a=a^*$ and $b=b^*$,
\begin{align}\label{Nash}
J(x;\xi^{a^*},\tau) \leq J_{a^*,b^*}(x)  \leq J(x;\xi,\tau_{b^*}^\xi), \quad x \in \R,
\end{align}
for any $\tau\in\Upsilon$ and  for any $\xi\in \Xi$, where we define $\tau_{b^*}^\xi:=\inf \{ t \geq 0: U_t^{\xi} > b^* \}$.


Intuitively, for such strategies to be a saddle point, it is suggested that $h$ decreases while $g$ increases, or more generally $h$ decreases faster than $g$ decreases.  See Assumptions \ref{assump_h_g} and \ref{assump_h_g_SP} below for the assumptions we make for the spectrally negative and positive case, respectively.

In addition, we need to choose the slope of $g$ carefully in view of the controlling cost derived by $\{ \diff \xi^a_t; t \geq 0\}$ as well.
Throughout the paper we assume that $g$ is affine and its slope is larger than $-1$ due to the reason described below.

\begin{assump} \label{assumption_K}
We assume that $g(x) = C + Kx$, $x \in \R$, for some $C \in \R$ and $K > -1$.
\end{assump}


Our formulation considers the case the controller has the \emph{first move advantage} in that the stopper can stop only \emph{after} the controller modifies the process at any time $t \geq 0$.  As is studied in  \cite{Hernandez_Simon_Zervos_2012}, we can also consider a version where the stopper has the first move advantage, and in general these two formulations are not equivalent.   Under Assumption \ref{assumption_K}, however, these are equivalent in our case.  Indeed,  at $x$ when
the controller modifies the process by $\Delta_x > 0$ units,
the difference of payoffs between stopping immediately after and stopping immediately before is
\begin{align*}g(x + \Delta_x) - g(x) + \Delta_x = (K+1)\Delta_x,
\end{align*}
which is greater than zero by  Assumption \ref{assumption_K}.

For the rest of the paper, we assume that $X_1$ has a finite first moment, so that the problem is well
defined and nontrivial.
\begin{assump}  \label{assump_finiteness_mu}We assume that $\mu, \hat{\mu} \in (-\infty, \infty)$, where we define
\begin{align}
\mu := \E [X_1] \quad \textrm{and} \quad \hat{\mu} := - \E [X_1]. \label{drift}
\end{align}
\end{assump}

%
%



\section{Spectrally one-sided \lev processes and scale functions}  \label{section_preliminary} In this section, we review the spectrally negative \lev process and its scale function, which will be used to express fluctuation identities for Section \ref{section_spectrally_negative}.  The spectrally positive \lev process is its dual and hence the scale function is again used analogously for our discussion in Section \ref{section_spectrally_positive}.

Let $X = \left\{X_t; t \geq 0 \right\}$ be a
\emph{spectrally negative \lev process} whose \emph{Laplace exponent} is given by
\begin{align}
\psi(s)  := \log \E \left[ e^{s X_1} \right] =  c s +\frac{1}{2}\sigma^2 s^2 + \int_{(-\infty,0)} (e^{s z}-1 - s z 1_{\{-1 < z < 0\}}) \nu (\diff z), \quad s \geq 0, \label{laplace_spectrally_positive}
\end{align}
where $\nu$ is a \lev measure with the support $(-\infty, 0)$ that satisfies the integrability condition $\int_{(-\infty,0)} (1 \wedge z^2) \nu(\diff z) < \infty$.  It has paths of bounded variation if and only if $\sigma = 0$ and $\int_{( -1,0)}|z|\, \nu(\diff z) < \infty$; in this case, we write \eqref{laplace_spectrally_positive} as
\begin{align*}
\psi(s)   =  \delta s +\int_{(-\infty,0)} (e^{s z}-1 ) \nu (\diff z), \quad s \geq 0,
\end{align*}
with $\delta := c - \int_{( -1,0)}z\, \nu(\diff z)$.  We exclude the case in which $X$ is a subordinator (i.e., $X$ has monotone paths a.s.). This assumption implies that $\delta > 0$ when $X$ is of bounded variation.  The first moment of $X_1$ in \eqref{drift} is $\mu = \psi'(0^+)$, which is assumed to be finite by Assumption \ref{assump_finiteness_mu}.

\subsection{Scale functions}
 Fix $q > 0$, as  the discount factor in the definition of $J(x;\xi,\tau)$. For any spectrally negative \lev process, there exists a function called  the  \emph{q-scale function}
\begin{align*}
W^{(q)}: \R \rightarrow [0,\infty),
\end{align*}
which is zero on $(-\infty,0)$, continuous and strictly increasing on $[0,\infty)$, and is characterized by the Laplace transform:
\begin{align*}
\int_0^\infty e^{-s x} W^{(q)}(x) \diff x = \frac 1
{\psi(s)-q}, \qquad s > \lapinv,
\end{align*}
where
\begin{equation}
\lapinv :=\sup\{\lambda \geq 0: \psi(\lambda)=q\}. \notag
\end{equation}
Here, the Laplace exponent $\psi$ in \eqref{laplace_spectrally_positive} is known to be zero at the origin and convex on $[0,\infty)$; therefore $\lapinv$ is well defined and is strictly positive as $q > 0$.   We also define, for $x \in \R$,
\begin{align*}
\overline{W}^{(q)}(x) &:=  \int_0^x W^{(q)}(y) \diff y, \\
Z^{(q)}(x) &:= 1 + q \overline{W}^{(q)}(x),  \\
\overline{Z}^{(q)}(x) &:= \int_0^x Z^{(q)} (z) \diff z = x + q \int_0^x \int_0^z W^{(q)} (w) \diff w \diff z.
\end{align*}
In particular, because $W^{(q)}$ is zero on the negative half line, we have
\begin{align}
Z^{(q)}(x) = 1  \quad \textrm{and} \quad \overline{Z}^{(q)}(x) = x, \quad x \leq 0.  \label{z_below_zero}
\end{align}

Let us define the \emph{first down-} and \emph{up-crossing times}, respectively, of $X$ by
\begin{align}
\label{first_passage_time}
T_b^- := \inf \left\{ t \geq 0: X_t < b \right\} \quad \textrm{and} \quad T_b^+ := \inf \left\{ t \geq 0: X_t >  b \right\}, \quad b \in \R.
\end{align}
Then, for any $b > 0$ and $x \leq b$,
\begin{align}
\begin{split}
\E_x \left[ e^{-q T_b^+} 1_{\left\{ T_b^+ < T_0^- \right\}}\right] &= \frac {W^{(q)}(x)}  {W^{(q)}(b)}, \\
 \E_x \left[ e^{-q T_0^-} 1_{\left\{ T_b^+ > T_0^- \right\}}\right] &= Z^{(q)}(x) -  Z^{(q)}(b) \frac {W^{(q)}(x)}  {W^{(q)}(b)}, \\
 \E_x \left[ e^{-q T_0^-} \right] &= Z^{(q)}(x) -  \frac q {\Phi(q)} W^{(q)}(x).
\end{split}
 \label{laplace_in_terms_of_z}
\end{align}

Fix $\lambda \geq 0$ and define $\psi_\lambda(\cdot)$ as the Laplace exponent of $X$ under $\p^\lambda$ with the change of measure
\begin{align*}
\left. \frac {\diff \p^\lambda} {\diff \p}\right|_{\mathcal{F}_t} = \exp(\lambda X_t - \psi(\lambda) t), \quad t \geq 0;
\end{align*}
see page 213 of \cite{Kyprianou_2006}. The process $X$ remains to be a spectrally negative \lev process under this measure.
 Suppose $W_\lambda^{(q)}$ and $Z_\lambda^{(q)}$ are the scale functions associated with $X$ under $\p^\lambda$ (or equivalently with $\psi_\lambda(\cdot)$).
Then, by Lemma 8.4 of \cite{Kyprianou_2006}, $W_\lambda^{(q-\psi(\lambda))}(x) = e^{-\lambda x} W^{(q)}(x)$, $x \in \R$,
which is well defined even for $q \leq \psi(\lambda)$ as in Lemmas 8.3 and 8.5 of \cite{Kyprianou_2006}.  In particular, we define
\begin{align*}
W_{\Phi(q)}(x) := W_{\Phi(q)}^{(0)}(x) = e^{-\Phi(q) x} W^{(q)}(x), \quad x \in \R,
\end{align*}
which is increasing as it is the scale function of $X$ under
$\p^{\Phi(q)}$.

Important properties of the functions defined above are summarized next, in terms of the path variation of the process $X$.
\begin{remark} \label{remark_smoothness_zero}
\begin{enumerate}
\item If either $X$ is of unbounded variation or the \lev measure is atomless, it is known that $W^{(q)}$ is $C^1(\R \backslash \{0\})$; see, e.g.,\ \cite{Chan_2009}.  Hence,
\begin{enumerate}
\item $Z^{(q)}$ is $C^1 (\R \backslash \{0\})$ and $C^0 (\R)$ for the bounded variation case, while it is $C^2(\R \backslash \{0\})$ and $C^1 (\R)$ for the unbounded variation case, and
\item $\overline{Z}^{(q)}$ is $C^2(\R \backslash \{0\})$ and $C^1 (\R)$ for the bounded variation case, while it is $C^3(\R \backslash \{0\})$ and $C^2 (\R)$ for the unbounded variation case.
\end{enumerate}
\item Regarding the asymptotic behavior near zero, as in Lemmas 4.3 and 4.4 of \cite{Kyprianou_Surya_2007},
\begin{align}\label{eq:Wq0}
W^{(q)} (0) &= \left\{ \begin{array}{ll} 0, & \textrm{if $X$ is of unbounded
variation,} \\ \frac 1 {\delta}, & \textrm{if $X$ is of bounded variation,}
\end{array} \right. \\
\label{eq:Wqp0}
W^{(q)'} (0^+) &=
\left\{ \begin{array}{ll}  \frac 2 {\sigma^2}, & \textrm{if }\sigma > 0, \\
\infty, & \textrm{if }\sigma = 0 \; \textrm{and} \; \nu(-\infty,0) = \infty, \\
\frac {q + \nu(-\infty,0)} {\delta^2}, & \textrm{if }\sigma = 0 \; \textrm{and} \;  \nu(-\infty,0) < \infty.
\end{array} \right.
\end{align}
On the other hand,  as in Lemma 3.3 of \cite{Kuznetsov2013},  $W_{\Phi(q)} (x) \nearrow \psi'(\Phi(q))^{-1}$ as $x \rightarrow \infty$, and consequently
\begin{align}
\lim_{x \rightarrow \infty} W^{(q)}(x) = \lim_{x \rightarrow \infty} e^{\Phi(q)x}W_{\Phi(q)}(x) = \infty. \label{W_infty}
\end{align}
\item As in (8.18) and Lemma 8.2 of \cite{Kyprianou_2006},
\begin{align*}
\frac {W^{(q)'}(y^+)} {W^{(q)}(y)} \leq \frac {W^{(q)'}(x^+)} {W^{(q)}(x)},  \quad  y > x > 0.
\end{align*}
In addition,  $W^{(q)'}(x^-) \geq W^{(q)'}(x^+)$ for all $x \in \R$, and we have the limit ${W^{(q)'}(x^+)} / {W^{(q)}(x)} \xrightarrow{x \uparrow \infty} \Phi(q)$.
\end{enumerate}
\end{remark}

\section{Spectrally negative \lev case} \label{section_spectrally_negative}

In this section, we assume that $X$ is a spectrally negative \lev process as defined in the last section, and obtain a saddle point \eqref{Nash}.  In addition to Assumptions \ref{assumption_K} and \ref{assump_finiteness_mu} above, we require additional conditions on the running cost function $h$, which are imposed solely in this section.
We define
\begin{align}
\tilde{h}(x) := h(x) + qx, \quad x \in \R, \label{h_tilde_SN}
\end{align}
where $q$ is the discount factor we use in the performance criterion $J$.
\begin{assump} \label{assump_h_g}
We assume that $h$ (and $\tilde{h}$) are continuously differentiable with absolutely continuous first derivatives and
\begin{enumerate}
\item $\tilde{h}'(x) \leq 0$ for all $x \in \R$;
\item $\tilde{h}' \leq -c_1$ on $[a_1,\infty)$ for some $a_1 \in \R$ and $c_1 > 0$;
 \item $\tilde{h}' \leq -c_2$ on $(-\infty, a_2]$ for some $a_2 \in \R$ and $c_2 > 0$.
\end{enumerate}
\end{assump}
The monotonicity of $h$ is justified in various settings.  For example, in the inventory example discussed in the introduction, the controller and the stopper try to avoid shortage and excess, respectively, and hence the running reward function $h$ needs to be decreasing.  When robust control is considered, there are many applications where the higher value of a controlled process is more desirable. Examples include, in addition to the classical inventory models, 
the equity to debt ratio the company needs to control (so as to meet the capital adequacy requirements).  Similar monotone running reward function appears in insurance dividend problem with capital injection when a random independent exponential time horizon is introduced (see \cite{Albrecher_2012}).

Notice that Assumption \ref{assump_h_g}(1)  is slightly stronger than the condition that $h$ is decreasing, but the value of $q$ is typically a small number.  This tilted version \eqref{h_tilde_SN} of the running reward function
  is commonly used in singular/impulse control problems; see, e.g., \cite{Bensoussan_2009, Bensoussan_2005, Yamazaki_2013}.   Assumption \ref{assump_h_g}(2,3), which require that it is strictly decreasing in the tail, are assumed for technical reasons where we take limits in the arguments below.

\begin{figure}[htbp]
\begin{center}
\begin{minipage}{1.0\textwidth}
\centering
\begin{tabular}{cc}
 \includegraphics[scale=0.58]{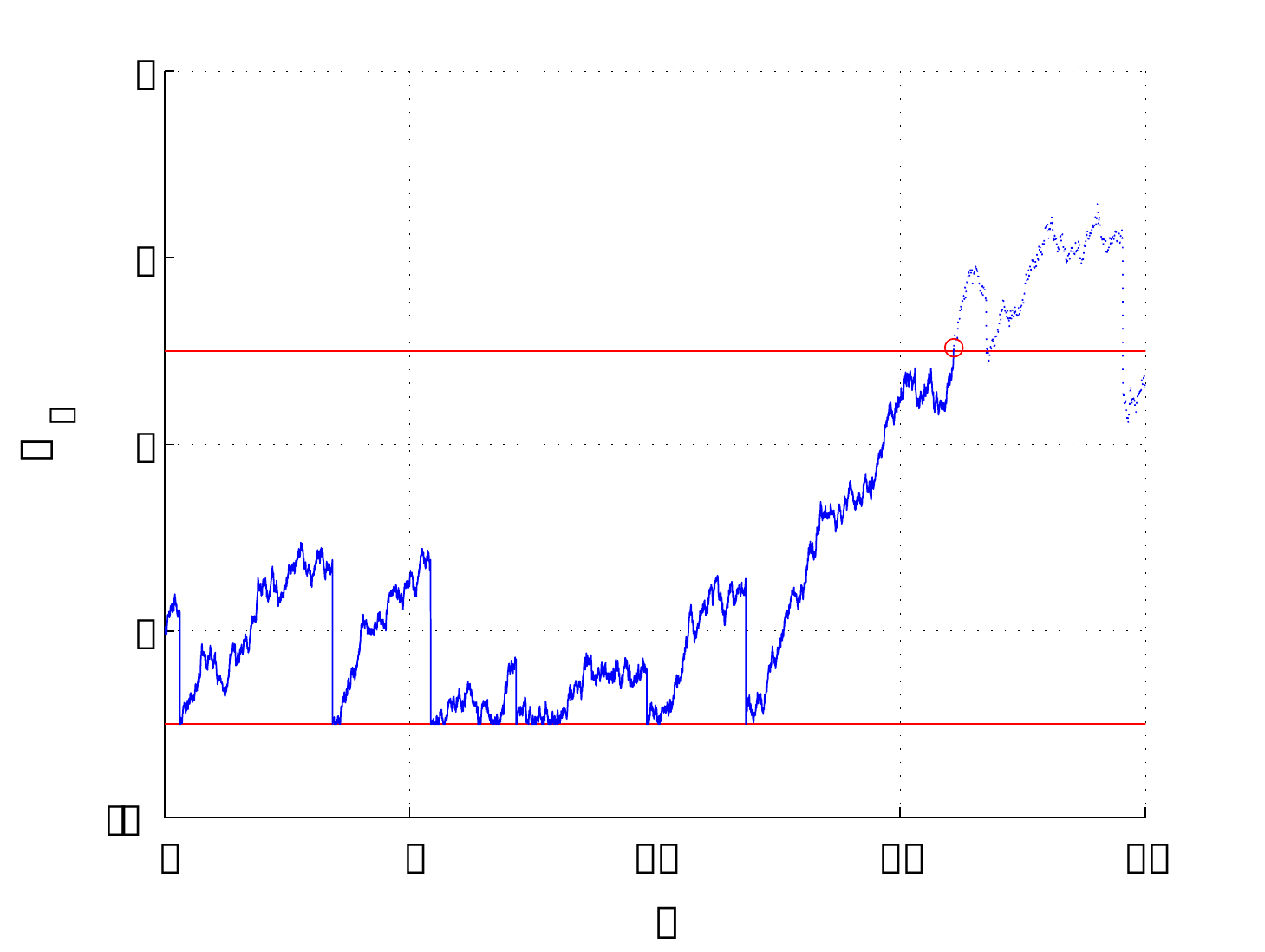} & \includegraphics[scale=0.58]{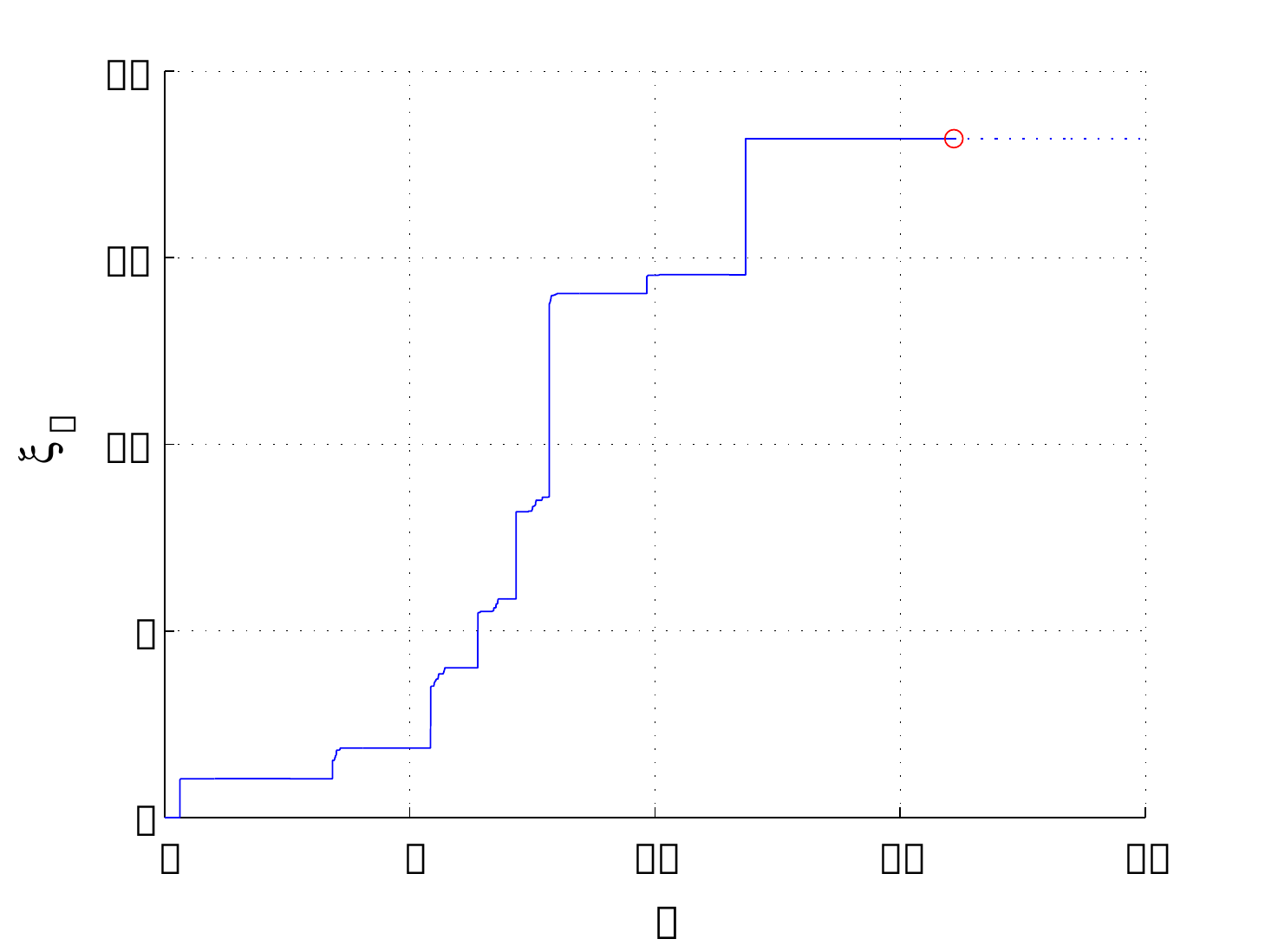}  \\
$U^{\xi^a}$ & $\xi^a$ \\
\end{tabular}
\end{minipage}
\caption{Sample realizations of $U^{\xi^a}$ and $\xi^a$ when $X$ is a spectrally negative \lev process. The circles indicate the points at the stopping time $\tau_{a,b}$.} \label{sample_path_neg}
\end{center}
\end{figure}

\subsection{Write $J_{a,b}$ using scale functions}  \label{subsection_J_scale_function_spectrally_negative}
Recall the set of strategies $(\xi^a, \tau_{a,b})$ defined as in \eqref{the_strategies}.  Our aim is to show that a saddle point is given in this form for appropriate choice of $a$ and $b$.

As an illustration, Figure \ref{sample_path_neg} shows a sample path of the controlled process $U^{\xi^a}$ and its corresponding control process $\xi^a$ when $X$ is a spectrally negative \lev process.  In the figure, the controlled process is reflected at the lower boundary $a=-1$ and is stopped at the first time it reaches the upper boundary $b=3$.  Due to the lack of positive jumps, the process necessarily creeps upward at $b$ (assuming it starts at a point lower than $b$).  The process $\xi^a$ increases whenever the process $U^{\xi^a}$ touches or goes below the level $a$.  Due to the negative jumps, it can increase both continuously and discontinuously.

Our first task is to write the performance criterion $J_{a,b}$  as in \eqref{J_a_b_def} using scale functions.  Because the controlled process $U^{\xi^a}$ is a reflected \lev process, we can use its fluctuation theories developed by Pistorius \cite{Pistorius_2004} and Avram et al.\ \cite{Avram_et_al_2007}.   Define, for any measurable function $f : \R \rightarrow \R$,
\begin{align} \label{def_psi_phi}
\varphi_a (x;f) &:= \int_{a}^x W^{(q)} (x-y) f(y) \diff y, \quad a,x \in \R.
\end{align}
Here $\varphi_a (x;f) = 0$ for any $x \leq a$ because $W^{(q)}$ is zero on $(-\infty,0)$.
We also define
\begin{align*}
l(x) := \overline{Z}^{(q)}(x) + \frac \mu q - \frac {Z^{(q)}(x)} {\Phi(q)}, \quad x \in \R.
\end{align*}

\begin{lemma}
For any $a \leq x \leq b$,
\begin{align} \label{J_a_b_initial}
\begin{split}
J_{a,b}(x)
&= \frac {Z^{(q)}(x-a)} {Z^{(q)}(b-a)} \left[ \varphi_a (b;h) + l(b-a) + g(b)\right] -  \varphi_a (x;h)  - l(x-a).
\end{split}
\end{align}
For $x > b$, we have $J_{a,b}(x) = g(x)$, while for $x<a$ the equality $J_{a,b}(x) = a-x+J_{a,b}(a)$ holds.
\end{lemma}
\begin{proof} We prove for $a \leq x \leq b$. The rest holds trivially.

First, by Theorem 1(i) of \cite{Pistorius_2004}, for any $a \leq x \leq b$,
\begin{align*}
\E_x \Big[ \int_0^{\tau_{a,b}} e^{-qt}  h (U_{t}^{\xi^a}) \diff t   \Big] &=  \int_a^\infty\Big[ Z^{(q)}(x-a)\frac { W^{(q)} (b-y)} {Z^{(q)}(b-a)} -  W^{(q)} (x-y) \Big]  h(y)  \diff y \\
&=  \frac {Z^{(q)}(x-a)} {Z^{(q)}(b-a)}\varphi_a (b;h) -  \varphi_a (x;h).
\end{align*}

Second, as in the proof of Theorem 1 of \cite{Avram_et_al_2007},
\begin{align*}
\E_x \Big[ \int_0^{\tau_{0,b}} e^{-qt} \diff \xi^0_{t} \Big] = - l(x) + Z^{(q)}(x) \frac {l(b)} {Z^{(q)}(b)}, \quad 0 \leq x \leq b.
\end{align*}
By shifting the initial position of $X$, \begin{align*}
\E_x \Big[ \int_0^{\tau_{a,b}} e^{-qt} \diff \xi^a_{t} \Big] = - l(x-a) + Z^{(q)}(x-a) \frac {l(b-a)} {Z^{(q)}(b-a)}, \quad a \leq x \leq b.
\end{align*}

Finally, by page 228 of \cite{Kyprianou_2006},
\begin{align*}
\E_x [ e^{- q \tau_{a,b}} ] = \frac {Z^{(q)}(x-a)} {Z^{(q)}(b-a)}, \quad a \leq x \leq b.
\end{align*}

Summing up these,  and because $U^{\xi^a}_{\tau_{a,b}} = b$, $\p_x$-a.s.\ (due to no positive jumps) for $x \leq b$, the proof is complete.
\end{proof}

\subsection{Smooth fit}  Now we shall choose the values of $a$ and $b$ so that the function $J_{a,b}$ is smooth. We will see that such desired pair satisfies some conditions that are characterized efficiently by the two functions:
\begin{align}\label{def_big_lambda}
\begin{split}
\Lambda(a,b) &:= {\varphi_a (b;h) + l(b-a) + g(b)}  -  \frac {h(a)} q {Z^{(q)}(b-a)} +  \frac 1 {\Phi(q)} {Z^{(q)}(b-a)} \\
&= {\varphi_a (b;h) + \overline{Z}^{(q)}(b-a) + \frac \mu q  + g(b)} -  \frac {h(a)} q {Z^{(q)}(b-a)}, \quad b \geq a,
\end{split}
\end{align}
and its derivative with respect to the second argument,
\begin{align}
 \lambda(a,b) := \frac \partial {\partial b}\Lambda(a,b) = \varphi_a (b;h')  +  Z^{(q)}(b-a) + K, \quad b > a. \label{def_small_lambda}
\end{align}
Here we used that
\begin{align}
\varphi_a' (x;h)  &=  W^{(q)}(x-a) h(a) + \varphi_a (x;h') \label{varphi_derivative}
\end{align}
which holds because $\varphi_a (x;h)  =  \overline{W}^{(q)}(x-a) h(a) + \int_a^{x} \overline{W}^{(q)}(x-y) h'(y) \diff y$ by integration by parts.

\begin{proposition} \label{proposition_smoothness}Suppose  $a < b$ are such that $\Lambda(a,b) = \lambda(a,b) = 0$.  Then
\begin{enumerate}
\item $J_{a,b}$ is differentiable (resp.\ twice-differentiable) at $a$ when $X$ is of bounded (resp.\ unbounded) variation;
\item $J_{a,b}$ is differentiable at $b$ for all cases.
\end{enumerate}
\end{proposition}

We shall prove this proposition by a straightforward differentiation and the asymptotic property of the scale function near zero as in Remark \ref{remark_smoothness_zero}(2).

First, taking derivatives in \eqref{J_a_b_initial}, for $a \leq x \leq b$,
\begin{align} \label{eq_derivatives}
\begin{split}
J_{a,b}'(x)
&= q \frac {W^{(q)}(x-a)} {Z^{(q)}(b-a)} \left[ \varphi_a (b;h) + l(b-a) + g(b)\right] -  \varphi_a' (x;h)  - l'(x-a), \\
J_{a,b}''(x)
&= q \frac {W^{(q)'}(x-a)} {Z^{(q)}(b-a)} \left[ \varphi_a (b;h) + l(b-a) + g(b)\right] -  \varphi_a'' (x;h)  - l''(x-a),
\end{split}
\end{align}
where we assume $X$ is of unbounded variation for the latter.  Here $\varphi_a' (x;h)$ and $\varphi_a'' (x;h)$ can be written, respectively, by \eqref{varphi_derivative} and
\begin{align*}
\varphi_a'' (x;h)  &=  W^{(q)'}(x-a) h(a) + W^{(q)}(x-a) h'(a) + \varphi_a (x;h'').
\end{align*}
On the other hand,
 \begin{align*}
l'(x-a) &= Z^{(q)}(x-a) - \frac{q W^{(q)}(x-a)} {\Phi(q)}, \\
l''(x-a) &= q \left[W^{(q)}(x-a) - \frac {W^{(q)'}(x-a)} {\Phi(q)} \right].
\end{align*}

Using these, \eqref{eq_derivatives} is written, for $a \leq x \leq b$,
\begin{align} \label{J_first_derivative}
\begin{split}
J_{a,b}'(x)
&= q \frac {W^{(q)}(x-a)} {Z^{(q)}(b-a)} \left[ \varphi_a (b;h) + l(b-a) + g(b)\right] \\ &-  W^{(q)}(x-a) h(a) - \varphi_a (x;h')  -  Z^{(q)}(x-a) +  \frac {q W^{(q)}(x-a)} {\Phi(q)} \\
&= q  {W^{(q)}(x-a)} \frac {\Lambda(a,b)} {Z^{(q)}(b-a)} - \varphi_a (x;h')  -  Z^{(q)}(x-a),
\end{split}
\end{align}
and, in particular when $X$ is of unbounded variation,
\begin{align}
J_{a,b}''(x)
&= q  {W^{(q)'}(x-a)}\frac {\Lambda(a,b)} {Z^{(q)}(b-a)}  - W^{(q)}(x-a) (h'(a)+q) - \varphi_a (x;h''). \label{J_second_derivative}
\end{align}

Sending $x \downarrow a$ in  \eqref{J_first_derivative},
\begin{align*}
J_{a,b}'(a^+)
&=  q{W^{(q)}(0)} \frac {\Lambda(a,b)} {Z^{(q)}(b-a)} - 1.
\end{align*}
In addition, for the unbounded variation case,  by \eqref{J_second_derivative},
\begin{align*}
J_{a,b}''(a^+)
&= q{W^{(q)'}(0^+)} \frac {\Lambda(a,b)} {Z^{(q)}(b-a)}.
\end{align*}
By Remark \ref{remark_smoothness_zero}(2), if $(a,b)$ are such that $\Lambda(a,b) = 0$, then $J_{a,b}$ is differentiable (resp.\ twice-differentiable) at $a$ when $X$ is of bounded (resp.\ unbounded) variation.

Now consider the smoothness at $b$.  By sending $x \uparrow b$ in \eqref{J_a_b_initial},
\begin{align*}
J_{a,b}(b^-)
&= \varphi_a (b;h) + l(b-a) + g(b) -  \varphi_a (b;h)  - l(b-a) = g(b) = J_{a,b} (b^+),
\end{align*}
and hence it is clearly continuous for any choice of $b(> a)$.  On the other hand, by taking $x \uparrow b$ in  \eqref{J_first_derivative} and subtracting $g'(b) = K$ on both sides,
\begin{align*}
J_{a,b}'(b^-) - g'(b)
&=  q W^{(q)}(b-a) \frac {\Lambda(a,b)} {Z^{(q)}(b-a)} - \lambda(a,b).
\end{align*}
We therefore see that if we can find $a < b$ such that $\Lambda(a,b) =  \lambda(a,b) = 0$, we can attain a smooth function that satisfies (1) and (2) of Proposition \ref{proposition_smoothness}.


\subsection{Existence of $(a^*,b^*)$}  We show the existence of $a^* < b^*$ such that $\Lambda(a^*,b^*) = \lambda(a^*,b^*) = 0$, or equivalently $(a^*,b^*)$ such that the function $ \Lambda(a^*,\cdot): [a^*,\infty) \rightarrow \R$ touches and also is tangent to the $x$-axis at $b^*$.

By \eqref{def_big_lambda}, the function $\Lambda(a,\cdot)$ starts at
\begin{align}
\Lambda(a,a)
&=  \frac {\mu} q   - \frac {h(a)} q + g(a) = \frac {\mu} q   - \frac {\tilde{h}(a)} q + C + (1+K) a. \label{Lambda_a_a}
\end{align}
Because this is monotonically increasing in $a$ from $-\infty$ to $\infty$ (by Assumptions \ref{assumption_K} and \ref{assump_h_g}),
 we can define $\bar{a}$ to be the zero of $\Lambda(a,a)$.
Then, $\Lambda(a,a) < 0$ if and only if $a < \bar{a}$.

\begin{proposition} \label{existence_a_b}
There exist $a^* < b^*$ such that $a^* < \bar{a}$ and
\begin{enumerate}
\item $\Lambda(a^*,b^*) = \lambda(a^*,b^*) = 0$,
\item  $\lambda(a^*,x) \geq 0$ for $x \in (a^*, b^*)$ and $\lambda(a^*,x) \leq 0$ for $x \in (b^*,\infty)$,
\item  $\Lambda(a^*, x) \leq 0$ for all $x \geq a^*$.
\end{enumerate}
\end{proposition}

We shall prove Proposition \ref{existence_a_b} for the rest of this subsection.
We first rewrite $\Lambda(a,b)$ and $\lambda(a,b)$ in terms of $\tilde{h}$ as in  \eqref{h_tilde_SN}. Integration by parts gives
\begin{align*}
\varphi_{a} (b;h)
&= \varphi_{a} (b;\tilde{h}) - \left[  a Z^{(q)}(b-a) + \overline{Z}^{(q)}(b-a) -b\right], \quad b \in \R.
\end{align*}
Using this and \eqref{varphi_derivative},
\begin{align} \label{lambda_rewrite_capital}
\begin{split}
\Lambda(a,b)
&= \varphi_a (b;\tilde{h})  -   a Z^{(q)}(b-a) +b + \frac \mu q   - \frac {h(a)} q Z^{(q)}(b-a)  + g(b) \\
&=  \int_{a}^b W^{(q)} (b-y) (\tilde{h}(y)- \tilde{h}(a)) \diff y    +b + \frac {\mu} q   - \frac {\tilde{h}(a)} q + g(b).
\end{split}
\end{align}
We can also write \eqref{def_small_lambda} as
\begin{align}
\lambda(a,b) = \int_{a}^b W^{(q)} (b-y) h'(y)\diff y  +  Z^{(q)}(b-a) + K
= \int_{a}^b W^{(q)} (b-y) \tilde{h}'(y)\diff y + K + 1. \label{lambda_rewrite}
\end{align}
Differentiating this further,
\begin{align}
\frac \partial {\partial b} \lambda(a,b)
&= \int_{a}^b W^{(q)'} (b-y) \tilde{h}'(y)\diff y  + \tilde{h}'(b) W^{(q)}(0) \leq 0, \label{lambda_derivative_b}
\end{align}
where the last inequality holds by Assumption \ref{assump_h_g} and because the scale function $W^{(q)}$ is increasing and  nonnegative.

We show the following lemma to show the existence of a local maximizer of $\Lambda(a, \cdot)$ over $(a,\infty)$.
\begin{lemma} \label{lemma_b_infty}For any fixed $a \in \R$, $\lambda(a,b) \xrightarrow{b \uparrow \infty} -\infty$ and $\Lambda(a,b) \xrightarrow{b \uparrow \infty} -\infty$.
\end{lemma}
\begin{proof}
Recall Assumption \ref{assump_h_g}.  For $\epsilon > 0$ and $b > a_1 \vee a + \epsilon$,
\begin{multline*}
\int_{a_1 \vee a}^b W^{(q)} (b-y) \tilde{h}'(y)\diff y = e^{\Phi(q) b}\int_{a_1 \vee a}^b e^{-\Phi(q)y}W_{\Phi(q)} (b-y) \tilde{h}'(y) \diff y \\ \leq - c_1e^{\Phi(q) b}\int_{a_1 \vee a}^b e^{-\Phi(q)y}W_{\Phi(q)} (b-y) \diff y
\leq - c_1e^{\Phi(q) b}\int_{a_1 \vee a}^{b-\epsilon} e^{-\Phi(q)y}W_{\Phi(q)} (\epsilon) \diff y \xrightarrow{b \uparrow \infty} - \infty,
\end{multline*}
where the last inequality holds because $W_{\Phi(q)}$ is increasing.
In view of  \eqref{lambda_rewrite}, this together with  $\int_a^{a_1 \vee a} W^{(q)} (b-y) \tilde{h}'(y)\diff y \leq 0$ shows the first claim. The second claim holds immediately because $\lambda(a,b)$ is the derivative of $\Lambda(a,b)$ with respect to $b$ as in \eqref{def_small_lambda}.
\end{proof}

Note that \eqref{def_small_lambda} and Assumption \ref{assumption_K} imply
\begin{align}
\lambda(a,a^+) = 1+ K  > 0. \label{lambda_small_starting}
\end{align}
By \eqref{lambda_derivative_b}, \eqref{lambda_small_starting} and Lemma \ref{lemma_b_infty}, for any choice of $a$, $\Lambda(a,\cdot)$ is a concave function that first increases and then decreases to $-\infty$ in the limit. This shows that there exists a unique global/local maximizer:
\begin{align}
\tilde{b}(a) := \arg \sup_{b' > a}\Lambda(a,b'). \label{def_b_of_a_SN}
\end{align}
Moreover, because differentiation of \eqref{lambda_rewrite} with respect to the first argument and Assumption  \ref{assump_h_g} yield
\begin{align}
\frac \partial {\partial a} \lambda(a,b)
&= -W^{(q)} (b-a) \tilde{h}'(a) \geq 0, \label{lambda_derivative_a}
\end{align}
we see that $\tilde{b}(a)$ is increasing in $a$.
We shall show that there exists $a$ such that $\Lambda(a,\tilde{b}(a)) = 0$.

We first start at $a=\bar{a}$ that makes \eqref{Lambda_a_a} vanish.  By \eqref{lambda_small_starting} and \eqref{def_b_of_a_SN}, $\Lambda(\bar{a},\tilde{b}(\bar{a})) > 0$.  We now gradually decrease its value to attain $a$ such that $\Lambda(a,\tilde{b}(a)) = 0$.  Differentiating \eqref{lambda_rewrite_capital} and Assumption \ref{assump_h_g}(1) give
\begin{align}
\frac \partial {\partial a}\Lambda(a,b)
&=  -\int_{a}^b W^{(q)} (b-y) \tilde{h}'(a) \diff y  - \frac {\tilde{h}'(a)} q \geq 0. \label{lambda_capital_increasing}
\end{align}
This suggests that $\Lambda(a,\tilde{b}(a))$ decreases monotonically as $a \downarrow -\infty$.  Indeed, for any $x < y < \tilde{b}(x)$ (here we know $\tilde{b}(x) > x$ by \eqref{lambda_small_starting}), we have $\Lambda(x, \tilde{b}(x)) \leq \Lambda(y, \tilde{b}(x)) \leq  \Lambda(y, \tilde{b}(y))$ by \eqref{lambda_capital_increasing} and the definition of $\tilde{b}$ as in \eqref{def_b_of_a_SN}.  Because $\Lambda(a, \tilde{b}(a))$ is continuous, it is now sufficient to show that for $a$ sufficiently small, $\Lambda(a, \tilde{b}(a))$ becomes negative.

%

Because $\tilde{b}(a)$ decreases as $a$ decreases as discussed above, there exists a limit $\tilde{b}(-\infty) := \lim_{a \downarrow -\infty}\tilde{b}(a)$.  We shall first show that this is $-\infty$.  By \eqref{lambda_rewrite} and Assumption \ref{assump_h_g}(3), for any $a < b \leq a_2$,
\begin{align}
\lambda(a,b) =  \int_{a}^b W^{(q)} (b-y) \tilde{h}'(y)\diff y + K + 1 \leq - c_2 \overline{W}^{(q)} (b-a) + K + 1.  \label{lambda_bound}
\end{align}
Hence if we choose $\epsilon := (\overline{W}^{(q)})^{-1} ((K+1)/c_2) \in (0, \infty)$,
then the right-hand side of \eqref{lambda_bound} vanishes for $b=a+\epsilon$.  Hence $\lambda(a, a+\epsilon)  \leq 0$ if $a+\epsilon  \leq a_2$. This and \eqref{lambda_derivative_b} show that $a \leq \tilde{b}(a) \leq a + \epsilon$ for any $a \leq a_2 - \epsilon$.  Therefore $\tilde{b}(-\infty)  = -\infty$.

Now due to the continuity of $\tilde{b}(a)$ (thanks to the continuity of $\lambda$), there exists some finite $\underline{a}$ such that $\tilde{b}(\underline{a}) = \bar{a}$.  By \eqref{lambda_capital_increasing},  $0 = \Lambda(\bar{a}, \bar{a})  =  \Lambda(\tilde{b}(\underline{a}), \tilde{b}(\underline{a}))  \geq  \Lambda(\underline{a}, \tilde{b}(\underline{a}))$, which shows the existence of $a^*$ such that $\Lambda(a^*, \tilde{b}(a^*)) = 0$, as desired. This completes the proof of Proposition  \ref{existence_a_b}.

\subsection{Value function}
Using $(a^*,b^*)$ as obtained in Proposition  \ref{existence_a_b}, equation \eqref{def_big_lambda} gives
\begin{align*}
{\varphi_{a^*} (b^*;h) + l(b^*-a^*) + g(b^*)}  = {Z^{(q)}(b^*-a^*)} \left( \frac {h(a^*)} q -  \frac 1 {\Phi(q)} \right).
\end{align*}
Substituting this in \eqref{J_a_b_initial},
\begin{align} \label{def_value_function}
\begin{split}
J_{a^*,b^*}(x)
&=  {Z^{(q)}(x-a^*)} \left( \frac {h(a^*)} q -  \frac 1 {\Phi(q)} \right) -  \varphi_{a^*} (x;h)  - l(x-a^*) \\
&=  {Z^{(q)}(x-a^*)} \left( \frac {h(a^*)} q -  \frac 1 {\Phi(q)} \right) -  \overline{Z}^{(q)}(x-a^*) - \frac \mu q + \frac {Z^{(q)}(x-a^*)} {\Phi(q)}  -  \varphi_{a^*} (x;h) \\
&=  {Z^{(q)}(x-a^*)} \frac {h(a^*)} q -  \overline{Z}^{(q)}(x-a^*) - \frac \mu q  -  \varphi_{a^*} (x;h),
\end{split}
\end{align}
which holds for any $x \leq b^*$. For $x > b^*$, we have $J_{a^*,b^*}(x)=g(x)$.  By \eqref{def_big_lambda}, we can also write
\begin{align}
J_{a^*,b^*}(x)  = g(x) - \Lambda(a^*,x), \quad a^* \leq x \leq b^*. \label{def_value_function_simple}
\end{align}

\begin{proposition} \label{convexity_SN}
The function $J_{a^*,b^*}$ is convex such that $-1 \leq J_{a^*,b^*}'(x) \leq K = g'(x)$ and $J_{a^*,b^*}(x) \geq g(x)$ for $x \in \R$.
\end{proposition}
\begin{proof}
Differentiating \eqref{def_value_function_simple},
\begin{align}
J_{a^*,b^*}'(x)
&=  K - \lambda(a^*,x),   \quad a^* < x < b^*. \label{def_value_function_simple_derivative}
\end{align}
By this and \eqref{lambda_derivative_b}, $J_{a^*,b^*}$ is convex on $(a^*, b^*)$.  Because $J_{a^*,b^*}$ is smooth at $a^*$ and $b^*$ as in Proposition \ref{proposition_smoothness}, affine on $(-\infty, a^*]$ and equals $g$  on $[b^*, \infty)$, the claim is immediate.
\end{proof}


We are now ready to state the main theorem of this section; we shall prove this in the next subsection.
\begin{theorem}\label{verification}
Let $a^*$ and $b^*$ be as in Proposition \ref{existence_a_b}, and define $v(x) := J_{a^*,b^*}(x)$ as the value of the  functional associated with the strategies
$\xi^{a^*}_t := \sup_{0 \leq t' \leq t} (a^*-X_{t'}) \vee 0$, $t \geq 0$, and $\tau_{a^*,b^*}:= \inf \{ t \geq 0: U_t^{\xi^{a^*}} > b^* \}$.
Then, we have
\begin{enumerate}
\item    $v(x)\leq J(x;\xi,\tau_{b^*}^\xi)$ for any $\xi\in \Xi$;
\item   $J(x;\xi^{a^*},\tau) \leq v(x)$ for any $\tau\in\Upsilon$.
\end{enumerate}
In other words,   the pair $(\xi^{a^*}, \tau_{a^*,b^*})$ is a \emph{saddle point} and $v(\cdot)$ is the value function of the game.
\end{theorem}

\subsection{Verification of optimality}
In order to prove Theorem \ref{verification}, we first present some preliminary results.
%

Let $\mathcal{L}$ be the infinitesimal generator associated with
the process $X$ applied to a sufficiently smooth function $f$
\begin{align*}
\mathcal{L} f(x) &:= c f'(x) + \frac 1 2 \sigma^2 f''(x) + \int_{(-\infty,0)} \left[ f(x+z) - f(x) -  f'(x) z 1_{\{-1 < z < 0\}} \right] \nu(\diff z), \quad x \in \R.
\end{align*}
By Remark \ref{remark_smoothness_zero}(1) and Proposition \ref{proposition_smoothness}, the function $J_{a^*,b^*}$ is $C^1 (\R)$ when $X$ is of bounded variation while it is  $C^1 (\R)$ and  $C^2(\R \backslash \{ b^*\}) $ when $X$ is of unbounded variation. Moreover, by Assumption \ref{assump_finiteness_mu} and Proposition \ref{convexity_SN}, the integral part of $\mathcal{L}J_{a^*,b^*}$ is finite.
Hence, $\mathcal{L} J_{a^*,b^*} (\cdot)$ makes sense at least anywhere on $\R\backslash \{ b^*\}$.

\begin{proposition} \label{verification_1} The following relations are satisfied.
\begin{enumerate}
\item $(\mathcal{L}-q) J_{a^*,b^*}(x) + h(x)= 0$ for $ x \in (a^*, b^*)$,
\item $(\mathcal{L}-q) J_{a^*,b^*}(x) + h(x) \geq 0$ for $x  \in (-\infty,a^*)$.
\end{enumerate}
\end{proposition}
\begin{proof}
(1) By Proposition 2 of \cite{Avram_et_al_2007} and as in the proof of Theorem 8.10 of \cite{Kyprianou_2006}, the processes
\begin{align*}
e^{-q (t \wedge T^-_{a^*} \wedge T^+_{b^*})} Z^{(q)}(X_{t \wedge T^-_{a^*} \wedge T^+_{b^*}} - a^*) \quad \textrm{and} \quad e^{-q (t \wedge T^-_{a^*} \wedge T^+_{b^*})}  R^{(q)} (X_{t \wedge T^-_{a^*} \wedge T^+_{b^*}}- a^*), \quad t \geq 0,
\end{align*}
with $R^{(q)}(x) := \overline{Z}^{(q)}(x) + \mu /q$, $x \in \R$, are martingales. Thanks to the smoothness of $Z^{(q)}$ and $\overline{Z}^{(q)}$ on $(0,\infty)$ as in Remark \ref{remark_smoothness_zero}(1), we obtain $(\mathcal{L}-q) R^{(q)}(y-a^*) = (\mathcal{L}-q) Z^{(q)}(y-a^*) = 0$ for any $y \in (a^*, b^*)$. This step is similar to the proof of  Theorem 2.1 in \cite{Bayraktar_2012}.

On the other hand, as in the proof of Lemma 4.5 of \cite{Egami-Yamazaki-2010-1},
\begin{align*}
(\mathcal{L}-q) \varphi_{a^*} (x;h) = h(x).
\end{align*}
Hence in view of  \eqref{def_value_function},  (1) is proved.

 (2)  For $x <a^*$, because $\mathcal{L} J_{a^*,b^*}$ is a constant (due to no positive jumps) and by \eqref{z_below_zero}, we have
 \begin{align}
\frac \partial {\partial x} \left[ (\mathcal{L}-q) J_{a^*,b^*}(x) + h(x) \right] = q + h'(x) = \tilde{h}'(x) \leq 0, \quad  x <a^*. \label{generator_less}
\end{align}
Hence $(\mathcal{L}-q) J_{a^*,b^*}(x) + h(x)$ is decreasing on $(-\infty, a^*)$.  Finally, the proof is complete by (1) and the smoothness condition at $a^*$ as in Proposition  \ref{proposition_smoothness}.
\end{proof}
On the right hand side  of the continuation region we have the next result.
\begin{lemma} \label{lemma_generator_b_plus} We have
$(\mathcal{L}-q) J_{a^*,b^*}(x) + h(x) \leq 0$ for $x \in (b^*, \infty)$.
\end{lemma}
\begin{proof}
We first prove $(\mathcal{L}-q) J_{a^*,b^*}(b^{*+}) + h(b^*)\leq 0$. By the smooth fit at $b^*$, this holds for $\sigma = 0$.  For $\sigma > 0$, we have by differentiating \eqref{def_value_function_simple_derivative} further and by \eqref{lambda_derivative_b},
\begin{align*}
J_{a^*,b^*}''(b^{*-}) =  -  \frac \partial {\partial b}\lambda(a^*,b^{*-})  \geq 0 = J_{a^*,b^*}''(b^{*+}).
\end{align*}
Hence by the smooth fit at $b^*$ and  Proposition \ref{verification_1}(1),
\begin{align*}
0 = (\mathcal{L}-q) J_{a^*,b^*}(b^{*-})  + h(b^*) \geq (\mathcal{L}-q) J_{a^*,b^*}(b^{*+}) + h(b^*) .
\end{align*}
\noindent

Now it remains to show $(\mathcal{L}-q) J_{a^*,b^*}(x) + h(x)$ is decreasing on $(b^*, \infty)$.
For any $y > x > b^*$, by Assumptions \ref{assumption_K} and \ref{assump_h_g},
\begin{align*}
&[(\mathcal{L}-q) J_{a^*,b^*}(y) + h(y)] - [(\mathcal{L}-q) J_{a^*,b^*}(x) + h(x)]  \\
&=\mathcal{L} J_{a^*,b^*}(y)  - \mathcal{L} J_{a^*,b^*}(x)  - q(g(y) - g(x)) + h(y) - h(x) \\
&\leq \mathcal{L} J_{a^*,b^*}(y)  - \mathcal{L} J_{a^*,b^*}(x)  - qK(y-x) - q(y-x) \\
&\leq \mathcal{L} J_{a^*,b^*}(y)  - \mathcal{L} J_{a^*,b^*}(x).
\end{align*}
In order to show that the right-hand side is nonpositive, let us define
\begin{align*}
\tilde{J}_{a^*,b^*} := J_{a^*,b^*} - g.
\end{align*}
Because $\mathcal{L} g$ is a constant, we have $\mathcal{L} J_{a^*,b^*}(y)  - \mathcal{L} J_{a^*,b^*}(x) = \mathcal{L} \tilde{J}_{a^*,b^*}(y)  - \mathcal{L} \tilde{J}_{a^*,b^*}(x)$.
Now because $\tilde{J}_{a^*,b^*}$ is decreasing on $(-\infty,b^*)$ by Proposition \ref{convexity_SN} and is zero on $[b^*,\infty)$, we have
\begin{align*}
\mathcal{L} \tilde{J}_{a^*,b^*}(y) - \mathcal{L} \tilde{J}_{a^*,b^*}(x) &=  \int_{(-\infty,0)} \left[  \tilde{J}_{a^*,b^*}(y+z) - \tilde{J}_{a^*,b^*}(y) - \tilde{J}_{a^*,b^*}'(y) z 1_{\{-1 < z < 0\}} \right] \nu(\diff z) \\
&- \int_{(-\infty,0)}\left[  \tilde{J}_{a^*,b^*}(x+z) -  \tilde{J}_{a^*,b^*}(x) -   \tilde{J}_{a^*,b^*}'(x) z 1_{\{-1 < z < 0\}} \right] \nu(\diff z) \\
 &=  \int_{(-\infty,b^*-y)} \tilde{J}_{a^*,b^*}(y+z)  \nu(\diff z) - \int_{(-\infty,b^*-x)}\tilde{J}_{a^*,b^*}(x+z) \nu(\diff z) \\
&= \int_{(-\infty,b^*-y)} \left[ \tilde{J}_{a^*,b^*}(y+z)  - \tilde{J}_{a^*,b^*}(x+z) \right] \nu(\diff z) - \int_{[b^*-y,b^*-x)} \tilde{J}_{a^*,b^*}(x+z) \nu(\diff z),
\end{align*}
which is nonpositive because $\tilde{J}_{a^*,b^*}$ is nonnegative and decreasing. This proves the claim.

\end{proof}

\begin{lemma}\label{transcondition} Let  $\underline{X}$ be the running minimum process of $X$.  We have $\lim_{t \rightarrow \infty} e^{-qt}\E \left[ \underline{X}_t \right] = 0$.
\end{lemma}
\begin{proof}
Under $\p = \p_0$, with $\eta(r)$ an independent exponential random variable with parameter $r > 0$, duality and the Wiener-Hopf factorization imply (see, e.g., (3.8) of \cite{Avram_et_al_2007})
\begin{align*}
X_{\eta(r)} - \underline{X}_{\eta(r)} \sim \exp(\Phi(r)).
\end{align*}
This together with $\E [X_{\eta(r)}]=\E [ \E [X_{\eta(r)} | \eta(r)]] = \E [ \eta(r) \mu ] = \mu/r$ gives
\begin{align*}
- \int_0^\infty r  e^{-r t}\E \underline{X}_t \diff t = -\E \left[ \underline{X}_{\eta(r)} \right] = \frac 1 {\Phi(r)} - \frac \mu r < \infty.
\end{align*}
On the other hand, we can also write $-\E[\underline{X}_{\eta(r)}]=\int_0^{\infty} \p \{ -\underline{X}_{\eta(r)}>y \} \diff y   =
\int_0^{\infty}  \p \{ T_{-y}^- < \eta(r) \} \diff y $, and
\begin{align*}
 \p \{ T_{-y}^- < \eta(r) \} &= \int_0^{\infty} r e^{-r s}\p\{  T_{-y}^-<s \}  \diff s\geq \int_t^{\infty} r e^{-r s}\p \{  T_{-y}^- <s \} \diff s\\
&\geq \int_t^{\infty} r e^{-r s}\p \{  T_{-y}^-<t \}  \diff s =  e^{-r t}\p \{  T_{-y}^- <t\}.
\end{align*}
Hence,
\begin{align*} 
\begin{split}
\infty &> \frac 1 {\Phi(r)} - \frac \mu r = -\E \left[ \underline{X}_{\eta(r)} \right] =\int_0^{\infty}\p \{  T_{-y}^- \leq \eta(r) \} \diff y \\
&\geq e^{- r t}\int_0^{\infty}\p \{  T_{-y}^-<t\} \diff y = e^{- r t}\E[-\underline{X}_{t}].
\end{split}
\end{align*}
This implies that $\limsup_{t\rightarrow \infty}e^{-rt}\E[-\underline{X}_{t}] < \infty$, for all $0 < r < q$, from which we conclude  the proof.
\end{proof}

We are now ready to prove Theorem \ref{verification}.

\begin{proof}[Proof of Theorem \ref{verification}]
We present the proof when $X$ is of unbounded variation. The proof for the bounded variation case is similar; while the smoothness of the function $J_{a^*,b^*}$ at $a^*$ is different, the following arguments in terms of the Meyer-It\^o formula clearly hold.

Let $\xi\in \Xi$ be arbitrary such that  $J(x;\xi,\tau_{b^*}^\xi) < \infty$
and hence
\begin{align}
\E_x [e^{-q \tau_{b^*}^\xi} (g(U_{\tau_{b^*}^\xi}^\xi))_+ 1_{\{\tau_{b^*}^\xi < \infty \}}] < \infty. \label{finiteness_policy_xi}
\end{align}
 Observe that when $x\geq b^*$ and $\Delta \xi_0 = 0$ the statement of the theorem holds, since
$v(x)=g(x)=J(x;\xi,\tau_{b^*}^\xi)$; when $x\geq b^*$ and $\Delta \xi_0 > 0$, $J(x;\xi,\tau_{b^*}^\xi) = \Delta \xi_0 + g(x+\Delta \xi_0) \geq g(x) = v(x)$ by Assumption \ref{assumption_K}.
 Hence we assume that $x< b^*$.

Recall that $v$ is  $C^1 (\R)$ and  $C^2(\R \backslash \{ b^*\})$ with bounded second derivatives near $b^*$.  Hence, using the Meyer-It\^o formula as in Theorem 4.71 of \cite{MR2273672},
 \begin{equation}\label{intbyparts}
 v(U_t^{\xi})-v(x)= \int_{[0,t]} v'(U_{s^-}^{\xi})\diff U_s^{\xi}+\frac{\sigma^2}{2}\int_0^t v'' (U_{s^-}^{\xi}) \diff s + \sum_{0\leq s\leq t}[v(U_s^{\xi})-v(U_{s^-}^{\xi})-v'(U_{s^-}^{\xi})\Delta U_s^{\xi}].
 \end{equation}
Further, if we denote $\xi^c$ as the continuous part of $\xi$, we can write
\begin{align}
 \int_{[0,t]} v'(U_{s^-}^{\xi})\diff U_s^{\xi}=\int_{[0,t]} v'(U_{s^-}^{\xi}) \diff X_s+\int_0^t v'(U_{s^-}^{\xi}) \diff\xi_s^c+\sum_{0\leq s\leq t}v'(U_{s^-}^{\xi})\Delta \xi_s. \label{decomposition_v_prime}
 \end{align}

 From the It\^o decomposition theorem (e.g., Theorem 2.1 of \cite{Kyprianou_2006}), we know that
 $$
 X_t=(\sigma B_t +c t)+  \left(\int_{[0,t]}\int_{(-\infty,-1]} x N(\diff s\times \diff x)\right)+\left(\lim_{\varepsilon\downarrow 0}\int_{[0,t]}\int_{(-1,-\varepsilon)}x(N(\diff s\times \diff x)-\nu(\diff x) \diff s)\right),
 $$
 where $\left\{B_t; t \geq 0 \right\}$ is a standard Brownian motion and $N$ is a Poisson random measure in   the measure space  $([0,\infty)\times(-\infty,0),\B [0,\infty)\times \B (-\infty,0), \diff t\times\nu( \diff x))$. The last term is a square integrable martingale, to which the limit converges uniformly on any compact $[0,T]$.

 Using this decomposition, we can write \eqref{decomposition_v_prime} as
 \begin{align*}
 \int_{[0,t]} v'(U_{s^-}^{\xi})\diff U_s^{\xi} &= \int_0^t \sigma v'(U_{s^-}^\xi) \diff B_s +\int_0^t c v'(U_{s^-}^\xi) \diff s+\int_{[0,t]} \int_{(-\infty,-1]} v'(U_{s^-}^\xi)y N(\diff s\times \diff y)\\
 &+ \lim_{\varepsilon\downarrow 0}\int_{[0,t]} \int_{(-1,-\varepsilon)} v'(U_{s^-}^\xi)y (N(\diff s\times \diff y)-\nu(\diff y)\diff s)\\
 &+\int_0^t v'(U_{s^-}^\xi) \diff \xi^c_s+\sum_{0\leq s\leq t}v'(U_{s^-}^\xi)\Delta \xi_s.
 \end{align*}

 Defining $A_t:=U_{t^-}^\xi+\Delta X_t$, $t\geq 0$, we have $A_t+\Delta \xi_t=U_{t}^\xi$, and hence we can write the last term in (\ref{intbyparts})  as
 \begin{align*}
 \sum_{0\leq s\leq t}[v(U_s^{\xi}) &- v(U_{s^-}^{\xi})-v'(U_{s^-}^{\xi})\Delta U_s^{\xi}]=\sum_{0\leq s\leq t}[v(A_{s}+\Delta \xi_s)-v(A_{s})-v'(U_{s^-}^\xi)\Delta \xi_s] 1_{\{\Delta\xi_s\neq 0\}}\\
 &+ \int_{[0,t]} \int_{(-\infty,0)}(v(U_{s^-}^\xi+y)-v(U_{s^-}^\xi)-v'(U_{s^-}^\xi)y )N(\diff s\times \diff y).
 \end{align*}
 Putting together the above expressions in (\ref{intbyparts}), we obtain
 \begin{align*}
 v(U_t^{\xi})-v(x) &= \int_0^t \sigma v'(U_{s^-}^\xi) \diff B_s +\int_0^t c v'(U_{s^-}^\xi)\diff s +\frac{\sigma^2}{2}\int_0^t  v''(U_{s^-}^\xi) \diff s\\
 &+ \lim_{\varepsilon\downarrow 0}\int_{[0,t]} \int_{(-1,-\varepsilon)} v'(U_{s^-}^\xi)y (N(\diff s\times \diff y)-\nu(\diff y)\diff s) + \int_0^t v'(U_{s^-}^\xi) \diff \xi^c_s\\
 &+ \sum_{0\leq s\leq t}[v(A_{s}+\Delta \xi_s)-v(A_{s})-v'(U_{s^-}^\xi)\Delta \xi_s] 1_{\{\Delta\xi_s\neq 0 \}}+\sum_{0<s\leq t}v'(U_{s^-}^\xi)\Delta \xi_s.\\
 &+ \int_{[0,t]} \int_{(-\infty,0)}(v(U_{s^-}^\xi+y)-v(U_{s^-}^\xi)-v'(U_{s^-}^\xi)y1_{\{y\in(-1,0)\}})N(\diff s\times \diff y).
 \end{align*}
Using integration by parts (see e.g.\ Corollary II.2 of \cite{MR2273672}) in $e^{-qt}v(U_t^{\xi})$ and the above expressions, it follows that
\begin{align}\label{Decomp}
 e^{-q t}v(U_t^{\xi})-v(x) &= \int_{[0,t]} e^{-qs} \diff v(U^\xi_s)-\int_0^t q e^{-qs} v(U_{s^-}^\xi) \diff s\\
 &= \int_0^t e^{-qs}(\mathcal{L}-q)v(U_{s^-}^\xi)\diff s+\int_0^t e^{-qs}v'(U_{s^-}^\xi)\diff \xi_s^c+J_t +M_t,\nonumber
 \end{align}
 with
 $$J_t:=\sum_{0\leq s\leq t}e^{-qs}[v(A_{s}+\Delta \xi_s)-v(A_{s})] 1_{\{\Delta\xi_s\neq 0\}}$$
 and
 \begin{align*}
 M_t &:= \int_0^t \sigma e^{-qs} v'(U_{s^-}^{\xi}) \diff B_s +\lim_{\varepsilon\downarrow 0}\int_{[0,t]} \int_{(-1,-\varepsilon)} e^{-qs}v'(U_{s^-}^{\xi})y (N(\diff s\times \diff y)-\nu(\diff y) \diff s)\\
 &+\int_{[0,t]} \int_{(-\infty,0)}e^{-qs}(v(U_{s^-}^\xi+y)-v(U_{s^-}^\xi)-v'(U_{s^-}^\xi)y1_{\{y\in(-1,0)\}})(N(\diff s\times \diff y)-\nu(\diff y) \diff s).
 \end{align*}
Since  $v'(x)\geq -1$ for $x\leq b^*$, observe that $v(A_{s}+\Delta \xi_s)-v(A_{s})\geq -\Delta\xi_s$ for $s\leq t\wedge\tau_{b^*}^\xi$.

For each $n\in \N$, define the stopping time $\tau_n$ as
$$
\tau_n :=\inf\{t>0\; : \;U_t^\xi<-n\}\wedge \tau_{b^*}^\xi,
$$
and the martingale $M=\{M_{t\wedge\tau_n};\; t\geq 0\}$, with $M_0=0$.  Then, since $(\mathcal{L}-q)v(x)\geq -h(x)$ for $x<b^*$ by Proposition \ref{verification_1} and since $v'(x)\geq -1$ for $x\leq b^*$ by Proposition \ref{convexity_SN},  taking expectation in
(\ref{Decomp}) gives by optional sampling that
$$
v(x)\leq  \E_x \left[\int_0^{t\wedge\tau_n}e^{-q s} h(U_{s^-}^\xi)\diff s+\int_{[0,t\wedge\tau_n]}e^{ -qs}\diff  \xi_s+e^{-q (t\wedge\tau_n)}v(U_{t\wedge\tau_n}^\xi)\right].
$$

We shall take limits first as $n \uparrow \infty$ and then as $t \uparrow \infty$ on the right-hand side.  Because $h$ is decreasing and $U_{s^-}^\xi \leq b^*$ for $s \in [0, t \wedge \tau_n]$, we have $\int_0^{t\wedge\tau_n}e^{-q s} (h(U_{s^-}^\xi))_- \diff s \leq \int_0^\infty e^{-qs}(h(b^*))_- \diff s < \infty$.  Hence we can apply the monotone convergence theorem and dominated convergence theorem, respectively, to the expectation of the positive and negative part of $\int_0^{t\wedge\tau_n}e^{-q s} h(U_{s^-}^\xi)\diff s$.  For $\int_{[0,t\wedge\tau_n]}e^{ -qs}\diff  \xi_s$, monotone convergence theorem applies because $\xi$ is nondecreasing.   Hence
\begin{align*}
\lim_{t \uparrow \infty} \lim_{n \uparrow \infty}\E_x \left[\int_0^{t\wedge\tau_n}e^{-q s} h(U_{s^-}^\xi)\diff s+\int_{[0,t\wedge\tau_n]}e^{ -qs}\diff  \xi_s\right] = \E_x \left[\int_0^{\tau_{b^*}^\xi}e^{-q s} h(U_{s^-}^\xi)\diff s+\int_{[0,{\tau_{b^*}^\xi}]}e^{ -qs}\diff  \xi_s\right].
\end{align*}
On the other hand,
\begin{align} \label{bound_v_stopped_tau_n}
\begin{split}
e^{-q (t\wedge\tau_n)}(v(U_{t\wedge\tau_n}^\xi))_+ &= e^{-q t}(v(U_{t}^\xi))_+ 1_{\{ t < \tau_n \}} + e^{-q \tau_n } (v(U_{\tau_n}^\xi))_+ 1_{\{ t \geq \tau_n, \tau_n < \tau_{b^*}^\xi \}} + e^{-q  \tau_{b^*}^\xi  } (g(U_{\tau_n}^\xi))_+ 1_{\{ t \geq \tau_n, \tau_n = \tau_{b^*}^\xi \}} \\
&\leq e^{-q (t \wedge \tau_n)}(v(U_{t \wedge \tau_n}^\xi))_+ 1_{\{ t \wedge \tau_n  < \tau_{b^*}^\xi \}}+ e^{-q  \tau_{b^*}^\xi  } (g(U_{\tau_{b^*}^\xi}^\xi))_+. 
\end{split}
\end{align}

Because $v'(x) = -1$ for $x < a^*$, we have a bound
$$
(v(x))_+ \leq \sup_{a^* \leq y \leq b^*} (v(y))_+  + (a^*-x)_+, \quad x \leq b^*.
$$
Hence, noting $U^\xi \geq X$,
\begin{align} \label{v_bounded_by_x}
\begin{split}
(v(U_{t \wedge \tau_n}^\xi))_+ 1_{\{ t \wedge \tau_n < \tau_{b^*}^\xi \}} \leq  \sup_{a^* \leq y \leq b^*} (v(y))_+  + (a^* - U_{t \wedge \tau_n}^\xi)_+ \leq \sup_{a^* \leq y \leq b^*} (v(y))_+  + (a^* - X_{t \wedge \tau_n})_+
\\ \leq \sup_{a^* \leq y \leq b^*} (v(y))_+  + |a^*| + (\underline{X}_t)_-,
\end{split}
\end{align}
which has a finite moment for any fixed $t \geq 0$.  This together with \eqref{finiteness_policy_xi} and  \eqref{bound_v_stopped_tau_n} shows by Fatou's lemma that
$$
\limsup_{n \uparrow \infty} \E_x \big[ e^{-q (t\wedge\tau_n)}v(U_{t\wedge\tau_n}^\xi) \big] \leq \E_x \big[ e^{-q (t\wedge \tau_{b^*}^\xi)}v(U_{t\wedge \tau_{b^*}^\xi}^\xi) \big].
$$

In order to take $t \uparrow \infty$, observe that
\begin{align} \label{bound_limsup_tau_n}
e^{-q (t\wedge \tau_{b^*}^\xi)}v(U_{t\wedge \tau_{b^*}^\xi}^\xi) =  e^{-q t}v(U_{t}^\xi) 1_{\{ t < \tau_{b^*}^\xi\}} + e^{-q \tau_{b^*}^\xi} g(U_{\tau_{b^*}^\xi}^\xi) 1_{\{ t \geq \tau_{b^*}^\xi\}}.
\end{align}
As in \eqref{v_bounded_by_x},
\begin{align*}
\E_x [(v(U_{t}^\xi))_+1_{\{ t < \tau_{b^*}^\xi\}}]  \leq \sup_{a^* \leq y \leq b^*} (v(y))_+  + |a^*| + \E_x \left[ (X_t)_-\right].
\end{align*}
Hence by Lemma \ref{transcondition} and  $\E_x [(v(U_{t}^\xi))_-1_{\{ t < \tau_{b^*}^\xi\}}] \leq \sup_{a^* \leq y \leq b^*} (v(y))_-$, the limit of $e^{-qt}\E_x [v(U_{t}^\xi) 1_{\{ t < \tau_{b^*}^\xi\}}]$ vanishes as $t \rightarrow \infty$;
 together with \eqref{finiteness_policy_xi} and \eqref{bound_limsup_tau_n},
\begin{align*}
\limsup_{t \rightarrow \infty}\E_x \big[e^{-q (t\wedge \tau_{b^*}^\xi)}v(U_{t\wedge \tau_{b^*}^\xi}^\xi) \big] \leq  \E_x \big[ e^{-q \tau_{b^*}^\xi} v(U_{\tau_{b^*}^\xi}^\xi) 1_{\{  \tau_{b^*}^\xi < \infty \}} \big] = \E_x \big[ e^{-q \tau_{b^*}^\xi} g(U_{\tau_{b^*}^\xi}^\xi) 1_{\{  \tau_{b^*}^\xi < \infty \}} \big].
\end{align*}
Putting altogether the results above, we conclude $v(x)\leq J(x;\xi,\tau_{b^*}^\xi)$.

For the second part of the theorem, let $\xi_t^{a^*}$  be as in the statement of the theorem. When $x<a^*$, the controller acts immediately, returning the process to the set $[a^*,\infty)$, i.e. $\xi_0^{a^*}=a^*-x$. Hence, we assume that the initial condition $x\geq a^*$. Fix an arbitrary stopping time $\tau\in \Upsilon$ such that $ J(x;\xi^{a^*},\tau) > - \infty$ and hence
\begin{align}
\E_x [e^{-q  \tau  } (g(U_{\tau}^{\xi^{a^*}}))_- 1_{\{\tau < \infty \}}] < \infty. \label{finiteness_policy_tau}
\end{align}
For each
$n\in \N$ define the stopping time $\tau_n$ as
$$
\tau_n :=\inf\{t>0\;: \;U_t^{\xi^{a^*}}>n\}\wedge \tau.
$$
From (\ref{Decomp}), Proposition \ref{verification_1}(1), Lemma \ref{lemma_generator_b_plus} and because $\diff \xi_s^{a^*,c} = 0$ (where $\xi_s^{a^*,c}$ is the continuous part of $\xi_s^{a^*}$) and $\Delta \xi_s^{a^*} = 0$ whenever $U_{s-}^{\xi^{a^*}} + \Delta X_s \in (a^*, \infty)$, the above localization argument gives
$$
 v(x)\geq  \E_x\left[\int_0^{t\wedge\tau_n}e^{-qs}h(U_{s^-}^{\xi^{a^*}})\diff s+\int_{[0,t\wedge\tau_n]}e^{ -qs}\diff  \xi_s^{a^*}  +e^{-q (t\wedge\tau_n)}v(U_{t\wedge\tau_n}^{\xi^{a^*}})\right].
$$
Similarly to the derivation above,
\begin{align*}
\lim_{t \uparrow \infty} \lim_{n \uparrow \infty}\E_x \left[\int_0^{t\wedge\tau_n}e^{-qs}h(U_{s^-}^{\xi^{a^*}})\diff s+\int_{[0,t\wedge\tau_n]}e^{ -qs}\diff  \xi_s^{a^*}  \right] = \E_x \left[\int_0^{\tau}e^{-q s} h(U_{s^-}^{\xi^{a^*}})\diff s+\int_{[0,\tau]}e^{ -qs}\diff  {\xi_s^{a^*}}\right].
\end{align*}
For the case $K \geq 0$, note that $v$ is bounded from below and hence Fatou's lemma implies
\begin{align*}
\liminf_{t \rightarrow \infty} \liminf_{n \rightarrow \infty}\E_x\left[e^{-q (t\wedge\tau_n)}v(U_{t\wedge\tau_n}^{\xi^{a^*}})\right] \geq \E_x\left[e^{-q \tau}v(U_{\tau}^{\xi^{a^*}}) 1_{\{\tau < \infty \}}\right]  + \liminf_{t \rightarrow \infty} \E_x\left[e^{-q t}v(U_{t}^{\xi^{a^*}}) 1_{\{\tau = \infty \}}\right] \\
\geq \E_x\left[e^{-q \tau}v(U_{\tau}^{\xi^{a^*}}) 1_{\{\tau < \infty \}}\right] \geq \E_x\left[e^{-q \tau}g(U_{\tau}^{\xi^{a^*}}) 1_{\{\tau < \infty \}}\right],
\end{align*}
where the second inequality holds because $v$ is bounded from below on $[a^*,\infty)$ for the case $K\geq 0$ and the third inequality holds since $v \geq g$.

  For the case $K < 0$,
\begin{align*}
e^{-q (t\wedge\tau_n)}(v(U_{t\wedge\tau_n}^{\xi^{a^*}}))_- &= e^{-q t}(v(U_{t}^{\xi^{a^*}}))_- 1_{\{ t < \tau_n \}} + e^{-q \tau_n } (v(U_{\tau_n}^{\xi^{a^*}}))_- 1_{\{ t \geq \tau_n, \tau_n < \tau \}} + e^{-q  \tau} (v(U_{\tau}^{\xi^{a^*}}))_- 1_{\{ t \geq \tau_n, \tau_n = \tau \}} \\
&\leq e^{-q (t \wedge \tau_n)}(v(U_{t \wedge \tau_n}^{\xi^{a^*}}))_- 1_{\{ t \wedge \tau_n < \tau \}}+ e^{-q  \tau  } (v(U_{\tau}^{\xi^{a^*}}))_-.
\end{align*}
By  \eqref{finiteness_policy_tau} and  $v \geq g$, $\E_x [e^{-q  \tau  } (v(U_{\tau}^{\xi^{a^*}}))_- ]< \infty$.

Recall $\xi^{a^*}_t :=   (a^* - \underline{X}_t)_+ = (a^* - \underline{X}_t) + (\underline{X}_t-a^*)_+ \leq (a^* - \underline{X}_t) + (x-a^*)_+$, and hence
\begin{align*}
U_{t}^{\xi^{a^*}} \leq a^* + (x-a^*)_+ + \overline{X}_t - \underline{X}_t.
\end{align*}
Now,
\begin{align} \label{v_bound_by_X_t}
\begin{split}
(v(U_{t \wedge \tau_n }^{\xi^{a^*}}))_-1_{\{t \wedge \tau_n  < \tau\}} &\leq  \inf_{a^* \leq y \leq b^*} (v(y))_-  + |K| ( U_{t \wedge \tau_n}^{\xi^{a^*}} - b^*)_+  \\ &\leq  \inf_{a^* \leq y \leq b^*} (v(y))_-  + |K| \left[( a^* + (x-a^*)_+ + \overline{X}_t - \underline{X}_t  - b^*)_+\right]
\\ &\leq \inf_{a^* \leq y \leq b^*} (v(y))_-  + |K|( a^* + (x-a^*)_+ - b^*)_+ + |K| \left( \overline{X}_t - \underline{X}_t \right).
\end{split}
\end{align}
Because the right-hand side has a finite moment, Fatou's lemma implies $\liminf_{n \rightarrow \infty}\E_x [e^{-q (t\wedge\tau_n)}v(U_{t\wedge\tau_n}^{\xi^{a^*}}) ] \geq \E_x [ e^{-q (t\wedge\tau)}v(U_{t\wedge\tau}^{\xi^{a^*}}) ]$.

Furthermore,
\begin{align*}
e^{-q (t\wedge\tau)}v(U_{t\wedge\tau}^{\xi^{a^*}}) = e^{-q t}v(U_{t}^{\xi^{a^*}}) 1_{\{ t  < \tau\}}+ e^{-q \tau}v(U_{\tau}^{\xi^{a^*}}) 1_{\{ t  \geq \tau\}}.
\end{align*}
 By Lemma \ref{transcondition} and \eqref{v_bound_by_X_t},
$e^{-qt}\E \big[ (v(U_{t }^{\xi^{a^*}}))_-1_{\{t   < \tau\}} \big] \xrightarrow{t \uparrow \infty} 0$.  This together with $(v(U_\tau^{\xi^{a^*}}))_+ \leq \sup_{a^* \leq y \leq b^*} (v(y))_+$ (thanks to $K < 0$)  shows $e^{-qt}\E \big[ v(U_{t }^{\xi^{a^*}})1_{\{t   < \tau\}} \big] \xrightarrow{t \uparrow \infty} 0$.

On the other hand, we have $e^{-q \tau}(v(U_{\tau}^{\xi^{a^*}}))_- 1_{\{ t  \geq \tau\}} \leq e^{-q \tau}(g(U_{\tau}^{\xi^{a^*}}))_- 1_{\{ \tau < \infty \}}$, which has a finite moment by \eqref{finiteness_policy_tau}.  Hence,
\begin{align*}
\liminf_{t \rightarrow \infty}\E_x \left[ e^{-q (t\wedge\tau)}v(U_{t\wedge\tau}^{\xi^{a^*}}) \right] \geq \E_x \left[  e^{-q \tau}v(U_{\tau}^{\xi^{a^*}}) 1_{\{\tau < \infty\}} \right] \geq \E_x \left[  e^{-q \tau}g(U_{\tau}^{\xi^{a^*}}) 1_{\{\tau < \infty\}} \right].
\end{align*}
Putting altogether, we have $v(x) \geq J(x;\xi^{a^*},\tau)$, as desired.

\end{proof}

\section{Spectrally positive \lev case} \label{section_spectrally_positive}
We now consider the spectrally positive \lev process for $X$. We assume throughout this section that its dual process $-X$ is a spectrally negative \lev process with a Laplace exponent \eqref{laplace_spectrally_positive}.  This is to say
\begin{align*}
\psi(s)  := \log \E \left[ e^{-s X_1} \right] =  c s +\frac{1}{2}\sigma^2 s^2 + \int_{(0,\infty)} (e^{-s z}-1 + s z 1_{\{0 < z < 1\}}) \nu (\diff z), \quad s \geq 0,
\end{align*}
where $\nu$ is a \lev measure with the support $(0,\infty)$ that satisfies the integrability condition $\int_{(0,\infty)} (1 \wedge z^2) \nu(\diff z) < \infty$.  For the case of bounded variation, we can define $\delta := c + \int_{(0,1)}z\, \nu(\diff z)$.
The first moment of $-X$ in \eqref{drift} is $\hat{\mu} = \psi'(0^+)$, which is assumed to be finite by Assumption \ref{assump_finiteness_mu}.

 The scale function $W^{(q)}$ used in this section is associated with the spectrally negative \lev process $-X$.
Throughout this section, we assume the differentiability of the scale function on $(0,\infty)$; this is satisfied if $X$ is of unbounded variation or the \lev measure has no atoms (see Remark \ref{remark_smoothness_zero}(1)).
\begin{assump}
We assume that $W^{(q)}$ is differentiable on $(0,\infty)$.
\end{assump}

\begin{remark}\label{Rem spec pos inf}
Lemma \ref{transcondition} also holds for the spectrally positive \lev process $X$.  Indeed, it is known that $- \underline{X}_{\eta(r)}$ is exponentially distributed with parameter $\Phi(r)$ and hence $- \E \underline{X}_{\eta(r)} = \Phi(r)^{-1}$.
\end{remark}

Recall Assumption \ref{assumption_K} again about our definition of the terminal cost $g$. Regarding the running cost function $h$, we focus on the case it is linear:
\begin{align}
h(x) = -\alpha x + \beta, \quad x \in \R, \label{h_g_SP}
\end{align}
for some constants $\alpha, \beta \in \R$ that satisfy the following.

\begin{assump}  \label{assump_h_g_SP}
The function
\begin{align*}
\hat{h}(x) := h(x) - Kqx- [Cq    + K \hat{\mu}], \quad x \in \R,
\end{align*}
is decreasing.  In other words,
\begin{align*}
\hat{\alpha} :=  \alpha + Kq > 0.
\end{align*}
\end{assump}
While we are assuming here that the running cost function $h$ is linear, Assumption \ref{assump_h_g_SP} is requiring a weaker condition on the slope in comparison to Assumption \ref{assump_h_g} that was assumed for the spectrally negative \lev case. Indeed, Assumption \ref{assump_h_g_SP} accommodates the case $q \geq \alpha > - K q$.   In fact, we will see in this case that the controller's optimal strategy becomes trivial (i.e.\ $\xi \equiv 0$).   This is intuitively clear because, for small $x$ sufficiently far from the stopper's exercise region, moving the process $\Delta_x > 0$ units upward costs the controller $\Delta_x$ while its reduction of the future running cost is approximately $\int_0^\infty e^{-qt} (\Delta_x \alpha) \diff t = \Delta_x \alpha/ q$, which is less than or equal to $\Delta_x$ when $q \geq \alpha$; hence  the controller has no motivation to get involved.

\begin{figure}[htbp]
\begin{center}
\begin{minipage}{1.0\textwidth}
\centering
\begin{tabular}{cc}
 \includegraphics[scale=0.58]{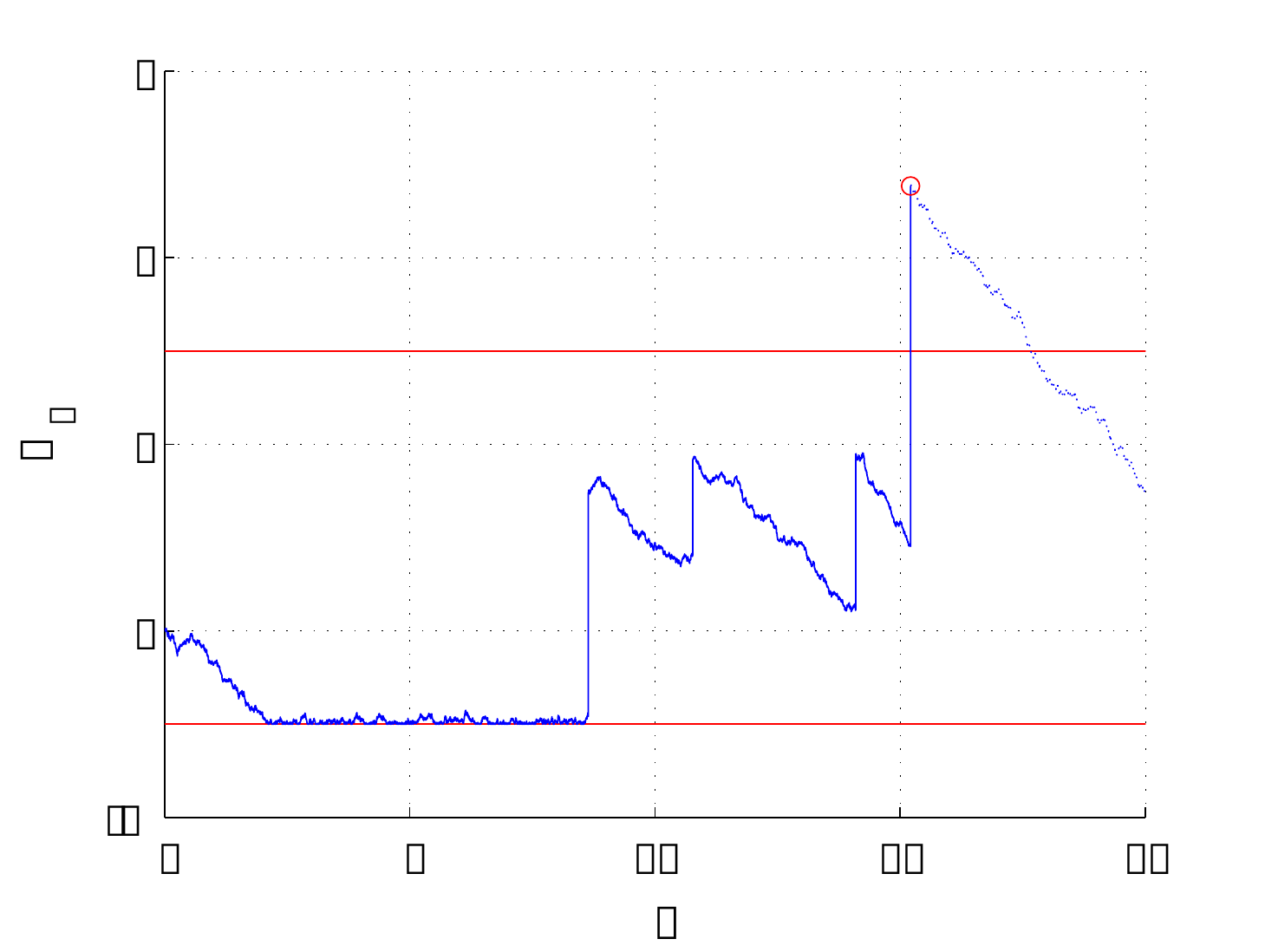} & \includegraphics[scale=0.58]{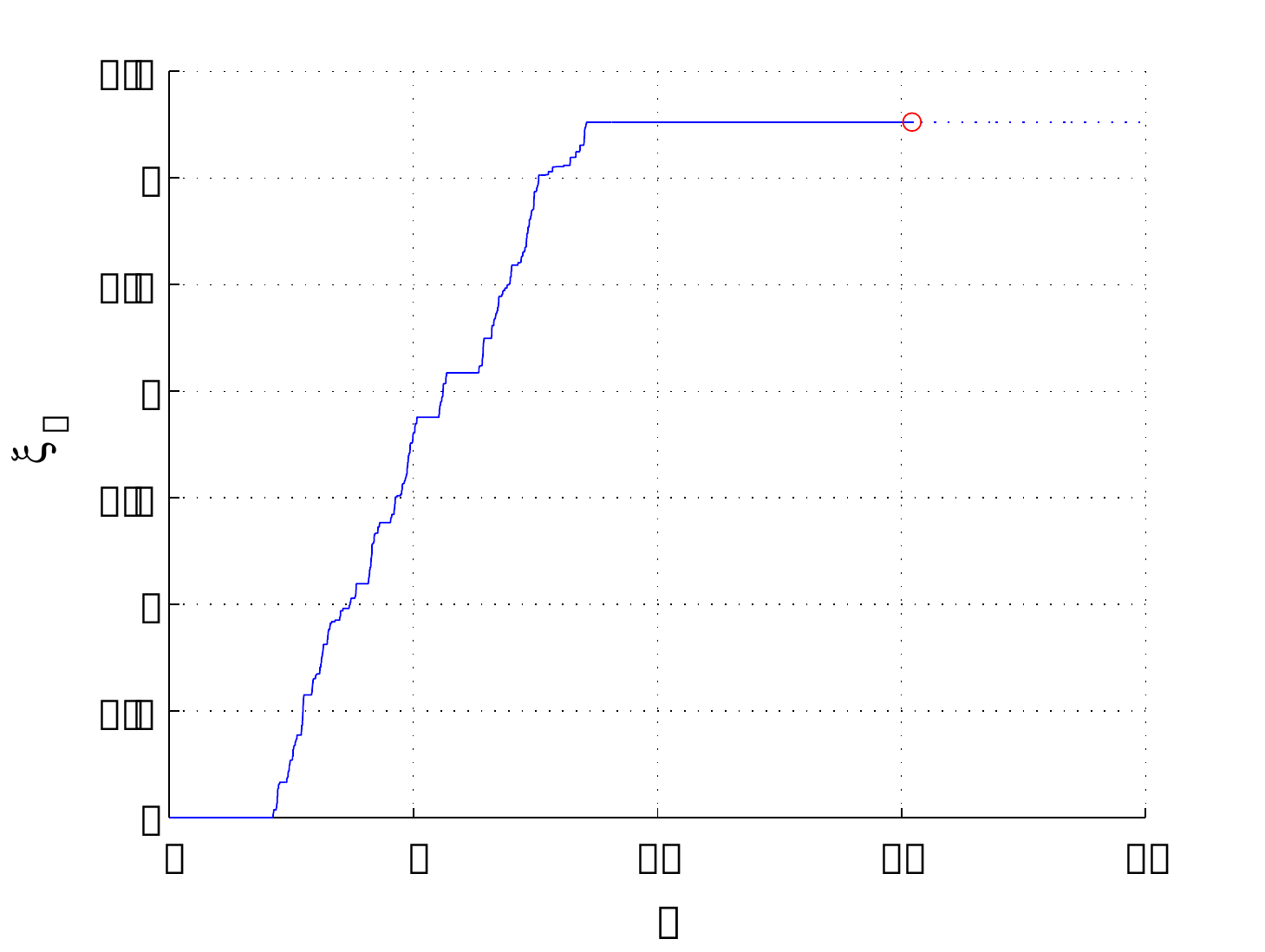}  \\
$U^{\xi^a}$ & $\xi^a$ \\
\end{tabular}
\end{minipage}
\caption{Sample realizations of $U^{\xi^a}$ and $\xi^a$ when $X$ is a spectrally positive \lev process. The circles indicate the points at the stopping time $\tau_{a,b}$.}  \label{sample_path_pos}
\end{center}
\end{figure}

\subsection{Write $J_{a,b}$ using scale functions}
Similarly to the spectrally negative \lev case, we shall show the optimality of a set of strategies $(\xi^a, \tau_{a,b})$ defined as in \eqref{the_strategies} for an appropriate choice of $a$ and $b$.

Figure \ref{sample_path_pos} shows sample paths of the controlled process $U^{\xi^a}$ and its corresponding control process $\xi^a$ when $X$ is a spectrally positive \lev process.  The controlled process is reflected at the lower boundary $a=-1$ and is stopped at the first time it \emph{reaches or exceeds} the upper boundary $b=3$.  Due to the lack of negative jumps, the process necessarily creeps downward at $a$ (assuming it starts at a higher point); hence $\xi^a$ increases only continuously.  On the other hand, due to the positive jumps, the process $U^{\xi^a}$ can jump over $b$ at the stopping time $\tau_{a,b}$.

As in the spectrally negative \lev case (see  Section \ref{subsection_J_scale_function_spectrally_negative}), $J_{a,b}$ as in \eqref{J_a_b_def} can be rewritten in terms of the scale function, using again the results by \cite{Pistorius_2004} and \cite{Avram_et_al_2007}.
Define
\begin{align} \label{gamma_def}
\begin{split}
\Gamma(x,b) &:= 1+ {W^{(q)}(0)}  h(x) + \int_x^b h(y) W^{(q)'}(y-x)  \diff y \\ &- [(C+bK) q  + K \hat{\mu} ]{W^{(q)}(b-x)} + K Z^{(q)}(b-x),  \quad x \leq b, \\
\overline{\Gamma}(x,b) &:= x - b - \int_x^b h(y) W^{(q)}(y-x)  \diff y \\ &+ (C+bK) Z^{(q)}(b-x)-K \overline{Z}^{(q)}(b-x) + K \hat{\mu} \overline{W}^{(q)}(b-x), \quad b,x \in \R.
\end{split}
\end{align}
In particular, if $x \geq b$, $\overline{\Gamma}(x,b) = x(1+K) + C -b$.
Because $\overline{\Gamma}(b,b) = C + bK$ and
\begin{align*}
\frac \partial {\partial x}\overline{\Gamma}(x,b) &=   \Gamma(x,b),
\end{align*}
we have
\begin{align*}
\overline{\Gamma}(x,b) = C+bK- \int_x^b \Gamma(y,b) \diff y, \quad x \leq b.
\end{align*}
\begin{lemma}
Fix any $a< b$.  For every $x \in \R$,
\begin{align}  \label{J_a_b_spec_pos}
\begin{split}
J_{a,b}(x)
&= \left\{ \begin{array}{ll}  \frac {W^{(q)}(b-a)} {W^{(q)'}(b-a)} \Gamma(a,b) + \overline{\Gamma}(a,b) - x + b, & x \leq a, \\ \frac {W^{(q)}(b-x)} {W^{(q)'}(b-a)} \Gamma(a,b) + \overline{\Gamma}(x,b) - x + b, & a < x.  \end{array} \right.
\end{split}
\end{align}
\end{lemma}
\begin{proof}
Fix $a \leq x \leq b$. First, by Theorem 1(ii) of \cite{Pistorius_2004},
\begin{align*}
&\E_x \Big[ \int_0^{\tau_{a,b}} e^{-qt}  h (U_{t}^{\xi^a}) \diff t   \Big] = \frac {W^{(q)}(b-x) W^{(q)}(0)} {W^{(q)'}(b-a)}   h(a) \\ &\qquad + \int_a^b h(y) \left[ W^{(q)}(b-x) \frac {W^{(q)'} (y-a)} {W^{(q)'} (b-a)} - W^{(q)}(y-x) \right] \diff y \\
&\qquad = \frac {W^{(q)}(b-x)} {W^{(q)'}(b-a)}  \left[  W^{(q)}(0) h(a) + \int_a^b h(y) W^{(q)'}(y-a)  \diff y \right] - \int_x^b h(y)  W^{(q)}(y-x) \diff y.
\end{align*}
Second, as in Proposition 1 of \cite{Avram_et_al_2007},
\begin{align*}
\E_x \Big[ \int_{[0,\tau_{a,b}]} e^{-qt} \diff \xi^a_{t} \Big] = \frac {W^{(q)}(b-x)} {W^{(q)'}(b-a)}.
\end{align*}
Finally, by Proposition 2(ii) of \cite{Pistorius_2004},
\begin{align*}
\E_x [ e^{- q \tau_{a,b}} ] =Z^{(q)}(b-x) - q W^{(q)}(b-a)  \frac {W^{(q)}(b-x)} {W^{(q)'}(b-a)},
\end{align*}
and, by Proposition 2 of \cite{Avram_et_al_2007},
\begin{align*}
\E_x [ e^{- q \tau_{a,b}} (U_{\tau_{a,b}}^{\xi^a}-b) ] =-\overline{Z}^{(q)}(b-x) + \hat{\mu} \overline{W}^{(q)}(b-x) + \frac{Z^{(q)}(b-a)-\hat{\mu} W^{(q)}(b-a)} {W^{(q)'}(b-a)} W^{(q)}(b-x).
\end{align*}
Summing up these terms and using the particular form of $g$,
\begin{align*}
\begin{split}
J_{a,b}(x) &=   \frac {W^{(q)}(b-x)} {W^{(q)'}(b-a)}  \left[  W^{(q)}(0) h(a) + \int_a^b h(y) W^{(q)'}(y-a)  \diff y \right] - \int_x^b h(y)  W^{(q)}(y-x) \diff y \\
&+ \frac {W^{(q)}(b-x)} {W^{(q)'}(b-a)} + (C+bK) Z^{(q)}(b-x) - (C+bK) q W^{(q)}(b-a)  \frac {W^{(q)}(b-x)} {W^{(q)'}(b-a)} \\
&-K \overline{Z}^{(q)}(b-x) + K \hat{\mu} \overline{W}^{(q)}(b-x) + K \frac{Z^{(q)}(b-a)-\hat{\mu} W^{(q)}(b-a)} {W^{(q)'}(b-a)} W^{(q)}(b-x) \\
&= \frac {W^{(q)}(b-x)} {W^{(q)'}(b-a)} \Gamma(a,b) + \overline{\Gamma}(x,b)  - x + b.
\end{split}
\end{align*}

For $x > b$, we have $J_{a,b}(x) = g(x) = \frac {W^{(q)}(b-x)} {W^{(q)'}(b-a)} \Gamma(a,b) + \overline{\Gamma}(x,b) - x + b$.  For $x < a$, $J_{a,b}(x) = (a-x) + J_{a,b}(a)$, as desired.
\end{proof}

\begin{remark}Integration by parts gives
\begin{align*}
\int_a^b h(y) W^{(q)'}(y-a)  \diff y= W^{(q)}(b-a) h(b) - W^{(q)}(0) h(a) - \int_a^b W^{(q)}(y-a) (-\alpha) \diff y.
\end{align*}
Hence
\begin{align} \label{Gamma_rewrite}
\begin{split}
\Gamma(a,b) &= 1 + W^{(q)}(b-a) [h(b) - ((C+bK) q    + K \hat{\mu})]\\ &- \int_a^b W^{(q)}(y-a) (-\alpha) \diff y + K Z^{(q)}(b-a) \\
&= 1 +K + W^{(q)}(b-a) \hat{h}(b)  + \hat{\alpha} \overline{W}^{(q)}(b-a).
\end{split}
\end{align}
\end{remark}

\subsection{Continuous/smooth fit}  Similarly to the spectrally negative \lev case, we choose $a$ and $b$ so that the function $J_{a,b}$ is continuous and smooth.  Define the derivative of $\Gamma(\cdot, \cdot)$ with respect to the first argument,
\begin{align} \label{gamma_rewrite}
\gamma(a,b) := \frac \partial {\partial a} \Gamma(a,b) = -W^{(q)'}(b-a) \hat{h}(b) - \hat{\alpha} W^{(q)}(b-a),  \quad a < b,
\end{align}
where the second equality holds by \eqref{Gamma_rewrite}.  This subsection shows the following.
\begin{proposition} \label{prop_smoothness_spec_pos}
\begin{enumerate}
\item If $b \geq a \geq -\infty$ are such that $\Gamma(a,b)/W^{(q)'}(b-a) = 0$, then $J_{a,b}$ is continuous (resp.\ differentiable) at $b$ when $X$ is of bounded (resp.\ unbounded) variation.
\item If in addition $a > -\infty$ and $\gamma(a,b) = 0$, then  $J_{a,b}$ is twice-differentiable at $a$.
\end{enumerate}
\end{proposition}
Here we write this condition as a division respect to $W^{(q)'}(b-a)$ (instead of simply writing $\Gamma(a,b) = 0$) to cope with the case $a = -\infty$; we will see below that with the choice $b =\bar{b}$ defined below, $\Gamma(a,\bar{b})$ converges as $a \downarrow -\infty$ to a finite constant and hence  $\lim_{a \downarrow -\infty}\Gamma(a,\bar{b})/W^{(q)'}(\bar{b}-a) = 0$.

The proof of Proposition \ref{prop_smoothness_spec_pos} is again achieved by straightforward differentiation and asymptotic behavior of the scale function near zero. Taking derivatives in  \eqref{J_a_b_spec_pos},
\begin{align} \label{J_a_b_derivative_spec_pos}
\begin{split}
J_{a,b}'(x)
&= \left\{ \begin{array}{ll} -1 & x < a, \\ -\frac {W^{(q)'}(b-x)} {W^{(q)'}(b-a)} \Gamma(a,b) + \Gamma(x,b)-1, & a < x < b, \\ K, & x > b,\end{array} \right. \\
J_{a,b}''(x^+)
&= \left\{ \begin{array}{ll} 0, & x < a, \\ \frac {W^{(q)''}((b-x)^-)} {W^{(q)'}(b-a)} \Gamma(a,b) + \gamma(x,b), & a < x < b, \\ 0, & x > b. \end{array} \right. 
\end{split}
\end{align}
Sending $x \uparrow b$ in \eqref{J_a_b_spec_pos} and the first identity of \eqref{J_a_b_derivative_spec_pos},
\begin{align}  \label{J_a_b_at_b}
\begin{split}
J_{a,b}(b^-)
&= \frac {W^{(q)}(0)} {W^{(q)'}(b-a)} \Gamma(a,b)  +C+bK, \\
J_{a,b}'(b^-)
&= -\frac {W^{(q)'}(0^+)} {W^{(q)'}(b-a)} \Gamma(a,b) + \Gamma(b,b)-1.
\end{split}
\end{align}
Here we have that by \eqref{Gamma_rewrite}
\begin{align}
\Gamma(b,b) &= 1+K+ {W^{(q)}(0)} \hat{h}(b), \quad b \in \R, \label{Gamma_at_b}
\end{align}
which is $1+K$ if and only if $X$ is of unbounded variation or $\hat{h}(b) = 0$ by Remark \ref{remark_smoothness_zero}(2).  Therefore, if $\Gamma(a,b)/W^{(q)'}(b-a) = 0$, then $J_{a,b}$ is continuous (resp.\ differentiable) at $b$ when $X$ is of bounded (resp.\ unbounded) variation.

Suppose  $a > -\infty$ and $\Gamma(a,b)/W^{(q)'}(b-a) = 0$. By sending $x \downarrow a$ in \eqref{J_a_b_derivative_spec_pos},
\begin{align}
J_{a,b}'(a^+)
= -1 = J_{a,b}'(a^-). \label{fit_at_a_SP}
\end{align}
We also have $J_{a,b}''(a^+) = \gamma(a,b)$,
 which equals $J_{a,b}''(a^-)=0$ if and only if $\gamma(a,b) = 0$.  This completes the proof of Proposition \ref{prop_smoothness_spec_pos}.

\subsection{Existence of $(a^*,b^*)$}  We shall pursue $(a^*, b^*)$ such that $\Gamma(a^*, b^*)/W^{(q)'}(b^*-a^*) = \gamma(a^*, b^*)=0$;  we obtain $(a^*, b^*)$ such that $\Gamma(a^*, b^*)/W^{(q)'}(b^*-a^*) = 0$ and $a^* := \arg \inf_{a \leq b^*} \Gamma(a, b^*)$.   Contrary to the spectrally negative \lev case, we shall see that $a^* = -\infty$ or $a^* = b^*$ can happen and in this case $\gamma(a^*, b^*) \neq 0$.

First, Assumption \ref{assump_h_g_SP} implies
 $\lim_{x \downarrow -\infty} \hat{h}(x) = \infty$ and $\lim_{x \uparrow \infty} \hat{h}(x) = -\infty$, and hence guarantees the existence and uniqueness of $-\infty < \underline{b} < \bar{b} < \infty$ such that
\begin{align}
\hat{h}(b)  >(\leq) 0, \quad b <(\geq) \underline{b}, \label{def_b_bar} \\
\Phi(q) \hat{h}(b) + \hat{\alpha} >(\leq)  0, \quad b <(\geq) \bar{b}. \label{def_b_bar_over}
\end{align}
With $\underline{b}$ and $\bar{b}$ as critical points, we shall prove the following.
\begin{proposition} \label{proposition_existence_SP}
\begin{enumerate}
\item  When $q \geq \alpha$, by setting $b^* = \bar{b}$ and  $a^* = -\infty$, we have $\Gamma(a^*, b^*)/W^{(q)'}(b^*-a^*) := \lim_{a \downarrow -\infty}\Gamma(a, b^*)/W^{(q)'}(b^*-a)  =0$ and $\Gamma(\cdot, b^*) \geq 0$ and $\gamma(\cdot, b^*) \geq 0$  on $(-\infty, b^*)$.
\item Suppose $q < \alpha$ and $X$ is of unbounded variation. Then, there exist $-\infty < a^* <b^* < \infty$ such that $\overline{b} > b^* > \underline{b}$ and $\Gamma(a^*,b^*) = \gamma(a^*,b^*) = 0$.
\item Suppose $q < \alpha$ and   $X$ is of bounded variation. Let $B$ be the unique value such that \begin{align}
\hat{h}(B) = -\delta {(1+K)},\label{about_B}
\end{align}
which exists by Assumption \ref{assump_h_g_SP}.

(i) If $(q + \nu(0, \infty)) (1+K) >\hat{\alpha}$ (including the case $\nu(0,\infty) = \infty$), there exist $-\infty < a^* <b^* < \infty$ such that $\underline{b} < b^* < B$ and $\Gamma(a^*,b^*) = \gamma(a^*,b^*) = 0$.

(ii) Otherwise, if we set $a^* = b^* = B$, then $\Gamma(a^*,b^*)=0$ and $\Gamma(\cdot,b^*) \geq 0$  and $\gamma(\cdot, b^*) \leq 0$ on $(-\infty, b^*)$.
\end{enumerate}
\end{proposition}

\begin{remark} \label{remark_a_equals_b} Recall that it was necessarily that $a^* < b^*$ when $X$ is spectrally negative.  In view of this and the spectrally positive \lev case as obtained in Proposition \ref{proposition_existence_SP}(3), we see that the criterion for whether $a^* < b^*$ or $a^* = b^*$ is closely linked to how much the process moves (or fluctuates) upward in an infinitesimal time.
 Indeed, for the case of spectrally negative \lev process and spectrally positive \lev process of unbounded variation, the process at any point $x \in \R$ immediately enters the set $(x,\infty)$.  Otherwise,  $a^* < b^*$ holds if the frequency of jumps $\nu(0,\infty)$ is sufficiently large.

Intuitively speaking, $a^*$ is close to $b^*$ if the controller wants to terminate the game immediately. However,  when the process fluctuates upward in an infinitesimal time by itself, the controller would not have to push it completely to the stopping region (recall that the controller also wants to minimize the controlling cost).   On the other hand, for other cases as in (ii) of Proposition \ref{proposition_existence_SP}(3), the process can tend to move away from the stopping region without entering it, no matter how close it is to the stopping region; to avoid this, $a^*=b^*$ can happen for these cases.
\end{remark}

We now prove Proposition \ref{proposition_existence_SP}.  With $\underline{b}$ satisfying \eqref{def_b_bar}, we first consider the function  $\Gamma(\cdot,\underline{b}) : a \mapsto \Gamma(a, \underline{b})$ defined on $(-\infty, \underline{b}]$.
By \eqref{Gamma_at_b} and Assumption \ref{assumption_K},
\begin{align*}
\Gamma(\underline{b},\underline{b})  = 1+K > 0.
\end{align*}
By \eqref{gamma_rewrite} and Assumption \ref{assump_h_g_SP},
\begin{align*}
\gamma(a,\underline{b}) =   - \hat{\alpha} W^{(q)}(\underline{b}-a) < 0, \quad a < \underline{b}.
\end{align*}
This implies that $\Gamma(\cdot,\underline{b}) $ is uniformly bounded from below by $1+K > 0$ and hence never touches nor goes below zero.

Now consider increasing the value of $b$ starting from $\underline{b}$.   Dividing both sides of \eqref{gamma_rewrite} by $W^{(q)}(b-a)$,
\begin{align} \label{gamma_divided}
\frac {\gamma(a,b)} {W^{(q)}(b-a)}
&=  - \frac{W^{(q)'}(b-a)} {W^{(q)}(b-a)} \hat{h}(b) - \hat{\alpha}.
\end{align}
\begin{lemma} \label{assump_decreasing_increasing}For any fixed $b > \underline{b}$, there exists $\bar{a}(b) \in [-\infty, b]$ such that $\Gamma(a,b)$ is decreasing (resp.\ increasing) in $a$ on $(-\infty, \bar{a}(b))$ (resp.\ on $(\bar{a}(b), b)$).  Equivalently,  $\gamma(a,b)$ is negative (resp.\ positive) for $a < \bar{a}(b)$  (resp.\ $a > \bar{a}(b)$).
\end{lemma}
\begin{proof} Recall, as in Remark \ref{remark_smoothness_zero}(3), that  $W^{(q)'}(y)/W^{(q)}(y)$ is decreasing.  In view of  \eqref{gamma_divided},
for $b > \underline{b}$ where $\hat{h}(b) < 0$,  there exists a unique value $\bar{a}(b)$ such that $\gamma(a,b)$ is negative (resp.\ positive)  for $a < \bar{a}(b)$ (resp.\ $a > \bar{a}(b)$).
\end{proof}
Because \eqref{gamma_divided} is monotonically increasing in $a$ (given $b > \underline{b}$), we can define by Remark \ref{remark_smoothness_zero}(2,3),
\begin{align*}
\lim_{a \downarrow -\infty}\frac {\gamma(a,b)} {W^{(q)}(b-a)}
&=  - \lim_{a \downarrow -\infty} \frac{W^{(q)'}(b-a)} {W^{(q)}(b-a)} \hat{h}(b) - \hat{\alpha} = -(\Phi(q) \hat{h}(b) + \hat{\alpha}).
\end{align*}
Recall the definition of $\bar{b}$ that satisfies \eqref{def_b_bar_over}. By \eqref{W_infty}, if $b > \bar{b}$ or equivalently $\Phi(q) \hat{h}(b) + \hat{\alpha} < 0$, $\Gamma(a, b) \xrightarrow{a \downarrow -\infty} -\infty$.  On the other hand, if  $b < \bar{b}$ or equivalently $\Phi(q) \hat{h}(b) + \hat{\alpha} > 0$, $\Gamma(a, b) \xrightarrow{a \downarrow -\infty} \infty$.  At the critical point $\bar{b}$ (which is larger than $\underline{b}$) at which $\Phi(q) \hat{h}(\bar{b}) + \hat{\alpha} = 0$, it is clear from the monotonicity of \eqref{gamma_divided} in $a$ that $\gamma(a,\bar{b}) \geq 0$ for any $a < \bar{b}$ and hence $\Gamma(a, \bar{b})$ decreases monotonically as $a \downarrow -\infty$.  In fact, the limit $\Gamma(-\infty, \bar{b}) := \lim_{a \downarrow -\infty}\Gamma(a, \bar{b})$ can be obtained explicitly as follows.
\begin{lemma} \label{lemma_equivalence_a_minus_infty}
We have $\Gamma(-\infty,\bar{b})
 =  1 -  \alpha / q$. Hence, $\Gamma(-\infty,\bar{b}) \geq 0$ if and only if $q \geq \alpha$.
\end{lemma}
\begin{proof}
 By substituting the definition of $\bar{b}$ in \eqref{Gamma_rewrite},
\begin{align*}
\Gamma(a,\bar{b})
&= 1 +K    + \hat{\alpha} \Big( \overline{W}^{(q)}(b-a) -  \frac{W^{(q)}(b-a)} {\Phi(q)} \Big).
\end{align*}
Recall as in  \eqref{laplace_in_terms_of_z} that $\E_x \big[ e^{-q T_0^-} \big] = 1 + q \big( \overline{W}^{(q)}(x) -  W^{(q)}(x) / {\Phi(q)} \big) \xrightarrow{x \uparrow \infty} 0$
and hence $\overline{W}^{(q)}(x) -  W^{(q)}(x) / \Phi(q) \xrightarrow{x \uparrow \infty} - q^{-1}$. Therefore,
\begin{align*}
\Gamma(-\infty,\bar{b})
&= 1 +K    -  \frac {\hat{\alpha}} q =  1 - \frac \alpha q,
\end{align*}
as desired.
\end{proof}

If $q \geq \alpha$, we already have a desired pair of $(a^*, b^*)$ defined by $(-\infty, \bar{b})$; this concludes the proof of Proposition \ref{proposition_existence_SP}(1). We shall see in our later discussion that this corresponds to the case the controller never exercises.

Now suppose for the rest of this subsection that $q < \alpha$ and hence $\Gamma (-\infty, \bar{b}) < 0$.  By the continuity of $\Gamma(\cdot, \cdot)$, we can choose $\check{b} < \bar{b}$ such that $\Gamma(\bar{a}(\check{b}), \check{b}) < 0$.  This together with the fact that $\Gamma(\bar{a}(\underline{b}), \underline{b}) > 0$ shows that we can define the smallest value of $\underline{b} < b^* < \check{b}$ such that  $\Gamma(\cdot, b^*)$ first touches the $x$-axis, i.e.,
\begin{align*}
b^* := \min \{ b \geq \underline{b} : \Gamma(\bar{a}(b),b) = 0 \},
\end{align*}
and let $a^* = \bar{a}(b^*)$.  For any $\underline{b} < b < \bar{b}$, because  $\Gamma (-\infty, b) = \infty$ as discussed above, we have $\bar{a}(b) \in (-\infty, b]$, and hence we must have that $-\infty < a^* \leq b^*$.  Now there are two scenarios: $a^* < b^*$ or $a^* = b^*$.

We first consider the case $X$ is of unbounded variation. In this case, we must have $\Gamma(b^*,b^*) = 1+K > 0$ and hence $a^* < b^*$.  This necessarily means that $\Gamma(\cdot, b^*)$ attains a local minimum at $a^*$ and hence $\gamma(a^*, b^*) = 0$. These calculations conclude the proof of Proposition \ref{proposition_existence_SP}(2).

Now we assume that $X$ is of  bounded variation.  Unlike the unbounded variation case,
$\Gamma(b,b)$ now depends on $b$ and can attain negative values for large $b$.  Recall our definition of $B$ that satisfies \eqref{about_B}. Note that by Remark \ref{remark_smoothness_zero}(2), we can also write equivalently $\hat{h}(B) = -(1+K) /{W^{(q)}(0)}$.
Because $\Gamma(b,b) $ as in \eqref{Gamma_at_b} decreases to $-\infty$, we have $B > \underline{b}$ and
\begin{align}
\Gamma(b,b) > 0 \Longleftrightarrow b < B. \label{zero_large_B}
\end{align}


First suppose $\nu(0,\infty) < \infty$. For fixed $b \geq \underline{b}$, again \eqref{gamma_divided} is increasing in $a$ by Remark \ref{remark_smoothness_zero}(3) and hence there exists, by  Remark \ref{remark_smoothness_zero}(2),
\begin{align}
\Upsilon(b) := \lim_{a \uparrow b}\frac {\gamma(a,b)} {W^{(q)}(b-a)} = - \frac{W^{(q)'}(0^+)} {W^{(q)}(0)} \hat{h}(b) - \hat{\alpha} = -\frac {q +\nu(0, \infty)} {\delta}  \hat{h}(b) - \hat{\alpha}, \label{def_upsilon}
\end{align}
which is monotonically increasing in $b$ on $(\underline{b}, \infty)$.
Note that $\Upsilon(b) \leq 0$ if and only if $\gamma(\cdot,b) \leq 0$ uniformly on $(-\infty,b)$ if and only if $\Gamma(\cdot,b) $ is monotonically decreasing on $(-\infty, b)$.

\textbf{Case 1}: Suppose $\Upsilon(B) \leq 0$.   In this case, the monotonicity of $\Upsilon$ implies  $\Upsilon(b) \leq \Upsilon(B) \leq 0$ for $\underline{b} < b < B$; in other words, $\Gamma(\cdot,b) $ is monotonically decreasing on $(-\infty, b)$ for any $\underline{b} < b < B$.  Moreover, \eqref{zero_large_B}  implies that $\Gamma(b,b) > 0$ for $\underline{b} \leq b < B$.  Hence, $B$ is the smallest value of $b$ such that $\Gamma$ touches the $x$-axis; we must have $a^* = b^* = B$.  Notice here that $\Gamma(a^*,b^*)=0$ and $\Gamma(\cdot,b^*) \geq 0$  and $\gamma(\cdot, b^*) \leq 0$ on $(-\infty, b^*)$.  See Figure \ref{figure_Gammas}(iii) below for the shape of $\Gamma(\cdot, \cdot)$ and $\gamma(\cdot, \cdot)$ for this particular case.

\textbf{Case 2}:  Suppose $\Upsilon(B) > 0$.  Because $\Upsilon$ is monotone and continuous, there exists $\underline{B} \in (\underline{b},B)$ such that $\Upsilon(\underline{B}) = 0$.  For $\underline{b} \leq b \leq \underline{B}$, $\Gamma(\cdot, b)$ is monotonically decreasing.  Moreover, because $\underline{B} < B$,  $\Gamma(b, b) > 0$.  Hence we must have $b^* > \underline{B}$ and hence $\Upsilon(b^*) > 0$ or $\gamma(b^{*-}, b^*) > 0$, which shows $a^* \neq b^*$.
We also see that $b^* < B$ by $\Gamma(B,B) = 0$ and the definition of $b^*$.
%

We now show that \textbf{Cases 1} and \textbf{2} correspond to (ii) and (i), respectively, in Proposition \ref{proposition_existence_SP}(3). By substituting \eqref{about_B} in \eqref{def_upsilon},
\begin{align*}
\Upsilon(B) = (q + \nu(0, \infty)) (1+K) - \hat{\alpha}.\end{align*}
Hence for the case of bounded variation with $\nu(0,\infty) < \infty$,
\begin{align*}
a^* = b^*  \Longleftrightarrow \Upsilon (B) \leq 0 \Longleftrightarrow (q + \nu(0, \infty)) (1+K) \leq \hat{\alpha},
\end{align*}
which is confirmed to be mutually exclusive with $q \geq \alpha$, which is a necessary and sufficient condition for $a^* = -\infty$ (here $\nu(0,\infty) > 0$ because otherwise the process becomes a subordinator).

This result can be extended to the case $\nu(0,\infty) = \infty$.  Indeed, we have $\gamma(b^-, b) = \infty$ for any $b > \underline{b}$, and hence $a^* \neq b^*$ due to the same reason as in \textbf{Case 2} above.
Summarizing the results, Proposition \ref{proposition_existence_SP}(3) holds.
\subsection{Verification of optimality}  With $(a^*, b^*)$ as in Proposition \ref{proposition_existence_SP}(2,3), substituting   $\Gamma(a^*,b^*)  = 0$ in \eqref{J_a_b_spec_pos} gives
\begin{multline} \label{value_function_spec_pos}
J_{a^*,b^*}(x)   = \overline{\Gamma}(x; b^*) - x + b^*  =  - \int_x^{b^*} h(y) W^{(q)}(y-x)  \diff y + (C+b^*K) Z^{(q)}(b^*-x) \\ -K \overline{Z}^{(q)}(b^*-x) + K \hat{\mu} \overline{W}^{(q)}(b^*-x), \quad x \geq a^*.
\end{multline}
For $x < a^*$, $J_{a^*,b^*}(x)  = J_{a^*,b^*}(a^*) - x + a^*$.  Also, for $a^* < x < b^*$,
\begin{align} \label{J_star_derivatives_spec_pos}
\begin{split}
J_{a^*,b^*}'(x)
&=  \Gamma(x,b^*)-1, \\
J_{a^*,b^*}''(x)  &= \gamma(x,b^*).
\end{split}
\end{align}
As in Proposition \ref{proposition_existence_SP}(1), the same results hold for the case $a^* = -\infty$ with $J_{-\infty, b^*}(x) = \lim_{a \downarrow -\infty} J_{a, b^*}(x)$.  It can be easily shown that this corresponds to the value under the controller's strategy $\xi^{-\infty}_t := 0$ for all $t \geq 0$ and $\tau_{-\infty, b^*} := T^+_{b^*}$.

\begin{proposition} \label{convexisty_SP}
The function $J_{a^*,b^*}$ is convex. Hence,  $-1 \leq J_{a^*,b^*}'(x) \leq g'(x) = K$ and $J_{a^*,b^*}(x) \geq g(x)$ for all $x \in \R$.
\end{proposition}
\begin{proof}
(i) The proof is immediate for $a^* = b^*$ because in this case it is continuous by Proposition \ref{prop_smoothness_spec_pos} and $J_{a^*,b^*}'(x) = -1$ on $(-\infty, b^*)$  while $J_{a^*,b^*}'(x) = K > -1$ on $(b^*, \infty)$.

(ii) When $a^* = -\infty$, by Proposition \ref{proposition_existence_SP}(1), $J_{-\infty, b^*}''(x) = \gamma(x, b^*) \geq 0$ on $(-\infty, b^*)$. Moreover, $J_{-\infty, b^*}'(b^{*-}) = \Gamma(b^*, b^*) - 1 = K+ {W^{(q)}(0)} \hat{h}(b^*) \leq K = J_{a^*, b^*}'(b^{*+})$ by $b^* > \underline{b}$, and it is continuous at $b^*$ by Proposition \ref{prop_smoothness_spec_pos}. Finally, for any $x \in (-\infty, b^*)$,
\begin{align*}
J_{-\infty, b^*}'(x) \geq  \lim_{x \downarrow -\infty} J_{-\infty, b^*}'(x)  = \Gamma(-\infty,b^*)-1 \geq -1,
\end{align*}
where the last bounds holds by Lemma \ref{lemma_equivalence_a_minus_infty}.

(iii) Suppose $-\infty < a^* < b^*$.  First, by Lemma \ref{assump_decreasing_increasing} noticing that $\bar{a}(b^*)=a^*$ in this case,
 $J_{a^*, b^*}''(x)  = \gamma(x, b^*) \geq 0$ on $(a^*, b^*)$.
 Second, the same reasoning as in (ii) shows $J_{a^*, b^*}'(b^{*-}) \leq J_{a^*, b^*}'(b^{*+})$ by $b^* > \underline{b}$. Finally, the smooth fit at $a^*$ and continuous fit at $b^*$ as in Proposition \ref{prop_smoothness_spec_pos} show the claim.
\end{proof}

We now state the main theorem for the spectrally positive \lev case.
\begin{theorem}\label{verification_SP}
Let $a^*$ and $b^*$ be as in Proposition \ref{proposition_existence_SP}, and define $v(x) := J_{a^*,b^*}(x)$ as the value of the  functional associated with the strategies
$\xi^{a^*}_t := \sup_{0 \leq t' \leq t} (a^*-X_{t'}) \vee 0$, $t \geq 0$, and $\tau_{a^*,b^*}:= \inf \{ t \geq 0: U_t^{\xi^{a^*}} > b^* \}$.
Then,  the pair $(\xi^{a^*}, \tau_{a^*,b^*})$ is a \emph{saddle point} and $v(\cdot)$ is the value function of the game.
\end{theorem}

%
We shall show this by verifying the variational inequalities for the rest of this section.
First, the generator of $X$ is given by  \begin{align*}
\mathcal{L} f(x) &:= -c f'(x) + \frac 1 2 \sigma^2 f''(x) + \int_{(0,\infty)} \left[ f(x+z) - f(x) -  f'(x) z 1_{\{0 < z < 1\}} \right] \nu(\diff z).
\end{align*}
By Proposition \ref{convexisty_SP} and Assumption \ref{assump_finiteness_mu}, the integral part of $\mathcal{L}J_{a^*,b^*}$ is finite.  Hence, by the smoothness  by Proposition \ref{prop_smoothness_spec_pos} and \eqref{fit_at_a_SP}, $\mathcal{L} J_{a^*,b^*}$ is well defined everywhere on $\R \backslash \{b^*\}$.
\begin{proposition} \label{proof_harmonic_spec_pos}
(1) Suppose $a^* < b^*$.  For $a^* < x < b^*$, $(\mathcal{L}-q) J_{a^*,b^*}(x)  + h(x) = 0$.\\
(2) We have $(\mathcal{L}-q) J_{a^*,b^*}(x) + h(x)\leq 0$ for $x > b^*$.
\end{proposition}
\begin{proof}
 (1) The result holds immediately in view of \eqref{value_function_spec_pos} as in the spectrally negative \lev case.\\
 (2) Because $X$ is spectrally positive,
\begin{align*}
\mathcal{L} J_{a^*,b^*}(x) &= K \E X_1 = - K \hat{\mu},
\end{align*}
and hence $(\mathcal{L}-q) J_{a^*,b^*}(x) + h(x) =  h(x) - [(C+xK) q    + K \hat{\mu}] = \hat{h}(x)$.  The result holds immediately by $x > b^* \geq \underline{b}$ and because $\hat{h}$ is decreasing.
\end{proof}


%

By differentiating \eqref{Gamma_rewrite},
\begin{align} \label{derivative_Gamma_b}
\frac \partial {\partial b} \Gamma(a,b)
&= W^{(q)'}(b-a) \hat{h}(b), \quad a < b,
\end{align}
which is positive (resp.\ negative) whenever $b < \underline{b}$ (resp.\ $b > \underline{b}$).
Namely, for fixed $a$, $\Gamma(a,b)$ increases (resp.\ decreases) monotonically in $b$ on $(a, \underline{b})$ (resp.\ $(\underline{b}, \infty)$).
Hence there exists $\lim_{b \uparrow \infty}\Gamma(a,b)$.

\begin{lemma} \label{Gamma_b_limit_spec_pos}
For all $a \in \R$, $\lim_{b \uparrow \infty}\Gamma(a,b) =-\infty$.
\end{lemma}
\begin{proof}
By  \eqref{Gamma_rewrite},  because  $\overline{W}^{(q)}(b-a)/W^{(q)}(b-a) \xrightarrow{b \uparrow \infty} \Phi(q)^{-1} \in (0,\infty)$ as in Exercise 8.5 of \cite{Kyprianou_2006} and $\hat{h}(b) \xrightarrow{b \uparrow \infty} -\infty$,
\begin{align*}
\frac {\Gamma(a,b)} {W^{(q)}(b-a) \hat{h}(b)}
&= 1 + \frac {1+K}  {W^{(q)}(b-a) \hat{h}(b)} + \hat{\alpha} \frac { \overline{W}^{(q)}(b-a)} {W^{(q)}(b-a) \hat{h}(b) } \xrightarrow{b \uparrow \infty} 1.
\end{align*}
Now because $W^{(q)}(b-a) \hat{h}(b) \xrightarrow{b \uparrow \infty} - \infty$, this shows the claim.
\end{proof}

By \eqref{derivative_Gamma_b} and Lemma \ref{Gamma_b_limit_spec_pos},  for each $a \leq a^*$ (which ensures $\Gamma(a,a) \geq 0$ even for the bounded variation case because $a^* \leq B$), there exists a unique $b(a) > \underline{b}$ such that  $\Gamma(a,b(a)) = 0$.

\begin{lemma} \label{lemma_monotonicity_b_a}
Suppose $a^* > -\infty$. (i) If $a < a' < a^*$, then $b(a) > b (a') > b^*$ and (ii)  if $a < a^*$, $\gamma(x,b(a)) \leq 0$ for all $x \leq a$.
\end{lemma}
\begin{proof}
Suppose $a < a^*$. We first prove that $b(a) > b^*$. Assume for contradiction that $a < a^*$ and $b(a) \leq b^*$ hold simultaneously.  By  \eqref{derivative_Gamma_b} and  $\underline{b} < b(a)$, we have $0 = \Gamma(a, b(a)) \geq  \Gamma(a, b^*)$, which is a contradiction because $\Gamma(\cdot, b^*)$ is decreasing on $(a, a^*)$ (by Lemma \ref{assump_decreasing_increasing} and because $a^* = \bar{a}(b^*)$ for the case $a^* < b^*$ and by Proposition \ref{proposition_existence_SP}(3) for the case $a^* = b^*$) and $\Gamma(a^*, b^*) = 0$.  Hence whenever $a < a^*$ we must have $b^* < b(a)$.

This also shows $\gamma(a,b(a)) \leq 0$ (and hence (ii)) by Lemma \ref{assump_decreasing_increasing}.  Indeed, if $\gamma(a,b(a)) > 0$, this means by Lemma \ref{assump_decreasing_increasing} that $\gamma(\cdot,b(a)) > 0$ on $(a, a^*)$ and hence $0 = \Gamma(a, b(a)) < \Gamma(a^*, b(a))$.  However, this contradicts with $0 = \Gamma(a^*, b^*)$, which is larger than $\Gamma(a^*, b(a))$ by $b^* < b(a)$ and \eqref{derivative_Gamma_b}.

Now suppose $a < a' < a^*$ and assume for contradiction that $b(a) \leq b(a')$ to complete the proof for (i).  By  \eqref{derivative_Gamma_b} we have $0 = \Gamma(a, b(a)) \geq  \Gamma(a, b(a'))$, which is a contradiction because $\Gamma(\cdot, b(a'))$ is decreasing on $(a, a')$  (due to (ii)) and $\Gamma(a', b(a')) = 0$.
\end{proof}

\begin{lemma} \label{lemma_gamma_b_a_shape} Suppose $a^* > -\infty$.  Fix $a < a^*$.  We have $\Gamma(y, b(a)) \leq 0$ for $a < y < a^*$.
\end{lemma}
\begin{proof}
By definition, $\Gamma(a,b(a))=0$.   Because $\Gamma(a^*,b)$ is decreasing in $b$ on $(\underline{b}, \infty)$ by \eqref{derivative_Gamma_b}, and $b(a) > b^* = b(a^*) \geq \underline{b}$ by Lemma \ref{lemma_monotonicity_b_a}(i),
\begin{align*}
\Gamma(a^*,b(a)) \leq \Gamma(a^*,b(a^*))=0.
\end{align*}
Because $\{ \Gamma(y, b(a)); a \leq y \leq a^* \}$ is decreasing and then increasing (or simply increasing or decreasing for $a \leq y \leq a^*$), we must have that
$\Gamma(y, b(a)) \leq \max (\Gamma(a,b(a)), \Gamma(a^*,b(a))) \leq 0$ for $a \leq y \leq a^*$, which proves the claim.
%
\end{proof}


\begin{proposition}  \label{proof_subharmonic_spec_pos}
Suppose $a^* > -\infty$.  For $x < a^*$, $(\mathcal{L}-q) J_{a^*,b^*}(x) + h(x) \geq 0$.
\end{proposition}
\begin{proof}
It is sufficient to prove
\begin{align}
(\mathcal{L}-q) (J_{a^*,b^*}-J_{x,b(x)})(x^+) := \lim_{y \downarrow x}(\mathcal{L}-q) (J_{a^*,b^*}-J_{x,b(x)})(y) \geq 0. \label{generator_positive_hypothesis}
\end{align}
Indeed if both \eqref{generator_positive_hypothesis} and $(\mathcal{L}-q) J_{a^*,b^*}(x) + h(x)< 0$ hold simultaneously,
\begin{align*}
0 > (\mathcal{L}-q) J_{a^*,b^*}(x) + h(x) \geq (\mathcal{L}-q) J_{x,b(x)}(x^+) + h(x),
\end{align*}
which leads to a contradiction because $(\mathcal{L}-q) J_{x,b(x)}(y) + h(y) = 0$ for $x < y < b(x)$ that holds similarly to Proposition \ref{proof_harmonic_spec_pos}(1).  Notice that the function $J_{x,b(x)}$ admits the same form as \eqref{value_function_spec_pos} because $\Gamma(x, b(x)) = 0$.

By \eqref{J_a_b_derivative_spec_pos}, for $x < y < b(x)$, \label{generator_parts}
\begin{align}
J_{x,b(x)}'(y)
=  \Gamma(y,b(x))-1 \quad \textrm{and} \quad
J_{x,b(x)}''(y)  =  \gamma(y,b(x)). \label{derivatives_x_b_x}
\end{align}
The dominated convergence theorem gives
\begin{align*}
&(\mathcal{L}-q) (J_{a^*,b^*}-J_{x,b(x)})(x^+) = - c (J_{a^*,b^*}' - J_{x,b(x)}')(x^+) + \frac 1 2 \sigma^2 (J_{a^*,b^*}'' - J_{x,b(x)}'')(x^+) \\ &+ \int_{(0,\infty)} \left[(J_{a^*,b^*}-J_{x,b(x)})(x+z) - (J_{a^*,b^*}-J_{x,b(x)})(x) -  (J_{a^*,b^*}'- J_{x,b(x)}')(x^+) z 1_{\{0 < z < 1\}} \right] \nu(\diff z) \\ &- q(J_{a^*,b^*}-J_{x,b(x)})(x^+).
\end{align*}
By the smooth fit of $J_{x,b(x)}$ at $x$ as in \eqref{fit_at_a_SP} and $J_{a^*,b^*}'(x) = -1$ and $J_{a^*,b^*}''(x) = 0$ as $x < a^*$, this is simplified to
\begin{align} \label{generator_parts}
\begin{split}
(\mathcal{L}-q) (J_{a^*,b^*}-J_{x,b(x)})(x^+) &= - \frac 1 2 \sigma^2  J_{x,b(x)}''(x^+) \\ &+ \int_{(0,\infty)} \left[(J_{a^*,b^*}-J_{x,b(x)})(x+z) - (J_{a^*,b^*}-J_{x,b(x)})(x)  \right] \nu(\diff z)
\\ &- q(J_{a^*,b^*}-J_{x,b(x)})(x^+).
\end{split}
\end{align}
By taking limits in \eqref{derivatives_x_b_x} and by Lemma \ref{lemma_monotonicity_b_a}(ii),
\begin{align*}
J_{x,b(x)}''(x^+)
&=  \gamma(x,b(x)) \leq 0.
\end{align*}

In order to prove the positivity of the integral part of \eqref{generator_parts}, we shall prove that
\begin{align}
J_{x,b(x)}'(y) \leq  J_{a^*,b^*}'(y), \quad y \in (x,\infty) \backslash \{b(x), b^*\}. \label{derivative_dominant}
\end{align}
Notice from \eqref{J_a_b_at_b} that both $J_{a^*,b^*}$ and $J_{x,b(x)}$ are continuous on $(x,\infty)$.  Recall also that $b(x) > b^*$ by Lemma \ref{lemma_monotonicity_b_a}(i).

(i) For $x < y < a^*$, by Lemma \ref{lemma_gamma_b_a_shape}, $J_{x,b(x)}'(y)
= \Gamma(y,b(x))-1  \leq -1 = J_{a^*,b^*}'(y)$.

(ii) For $a^* < y < b^*$
\begin{align*}
\begin{split}
J_{x,b(x)}'(y)
&= \Gamma(y,b(x))-1 \leq  \Gamma(y,b^*)-1  = J_{a^*,b^*}'(y).
\end{split}
\end{align*}
Here the inequality holds because we have $b(x) >  b^* > \underline{b}$ and because $\Gamma(y, \cdot)$ is decreasing by \eqref{derivative_Gamma_b}.

(iii) For $b^* < y < b(x)$, because, by Lemma \ref{assump_decreasing_increasing} and \eqref{derivative_Gamma_b}, $\Gamma(y,b(x)) \leq \Gamma(b^*,b(x)) \vee \Gamma(b(x),b(x)) \leq \Gamma(b^*,b^*) \vee \Gamma(b(x),b(x)) =  \Gamma(b^*,b^*) $, and hence
\begin{align*}
\begin{split}
J_{x,b(x)}'(y)
= \Gamma(y,b(x))-1  \leq K + W^{(q)}(0) \hat{h}(b^*)  \leq K = J_{a^*,b^*}'(y),
\end{split}
\end{align*}
where the last inequality holds by $b^* > \underline{b}$.

(vi) For $y > b(x)$, we have that $J_{x,b(x)}'(y) = J_{a^*,b^*}'(y)=K$.  Hence \eqref{derivative_dominant} holds, and consequently the integral of \eqref{generator_parts} is positive.

Finally, thanks to the continuous fit for $J_{x,b(x)}$ at $b(x)$ and because $b(x) > b^*$, $J_{x,b(x)}(b(x)) = J_{a^*,b^*}(b(x)) = C + K b(x)$.  This together with \eqref{derivative_dominant} shows
\begin{align*}
J_{x,b(x)}(x^+) \geq  J_{a^*,b^*}(x^+).
\end{align*}
Putting altogether, \eqref{generator_positive_hypothesis} indeed holds.  This completes the proof.
\end{proof}


Propositions \ref{convexisty_SP}, \ref{proof_harmonic_spec_pos} and \ref{proof_subharmonic_spec_pos} show the variational inequality, and we can modify slightly the proof of Theorem \ref{verification} to show that this is a sufficient condition for optimality.  The proof is very similar, and we do not include it so as to avoid unnecessary repetitions.  The only differences worthy of remark are (1) the smoothness of the value function in view of the difference between Propositions \ref{proposition_smoothness} and  \ref{prop_smoothness_spec_pos} and (2) the  estimation  of the running minimum process $\underbar{X}$ (obtained in Lemma \ref{transcondition} for the spectrally negative \lev case), which is used to interchange limits over expectations.   However, these do not cause any problems.  For (1), the candidate value function fails to be differentiable at $b^*$ (although continuous) for the case of bounded variation, but the Meyer-It\^o formula (Theorem 4.71 of \cite{MR2273672}) applies and the local-time term vanishes due to paths of bounded variation.   For (2), as it is pointed out in Remark \ref{Rem spec pos inf}, it also holds for the spectrally positive \lev case.  Hence Theorem \ref{verification_SP} holds.

\section{Numerical Examples} \label{section_numerics}
In this section, we confirm the results numerically using the case $h$ is given by  \eqref{h_g_SP}.  For $X$, we use the spectrally negative process with i.i.d.\ phase-type distributed jumps \cite{Asmussen_2004} of the form
\begin{equation*}
  X_t  - X_0= \delta t+\sigma B_t - \sum_{n=1}^{N_t} Z_n, \quad 0\le t <\infty,
\end{equation*}
for some $\delta \in \R$ and $\sigma \geq 0$, and also the spectrally positive \lev process defined as its dual.  Here $B=\{B_t; t\ge 0\}$ is a standard Brownian motion, $N=\{N_t; t\ge 0\}$ is a Poisson process with arrival rate $\kappa$, and  $Z = \left\{ Z_n; n = 1,2,\ldots \right\}$ is an i.i.d.\ sequence of phase-type-distributed random variables with representation $(m,{\bm \alpha},{\bm T})$; see \cite{Asmussen_2004}.
These processes are assumed mutually independent. The Laplace exponent \eqref{laplace_spectrally_positive} is then
\begin{align*}
 \psi(s)   = \delta s + \frac 1 2 \sigma^2 s^2 + \kappa \left( {\bm \alpha} (s {\bm I} - {\bm{T}})^{-1} {\bm t} -1 \right),
 \end{align*}
which can be extended to $s \in \mathbb{C}$ except at the negative of eigenvalues of ${\bm T}$.  Suppose $\{ -\xi_{i,q}; i \in \mathcal{I}_q \}$ is the set of the roots of the equality $\psi(s) = q$ with negative real parts, and if these are assumed distinct, then
the scale function can be written
\begin{align*}
W^{(q)}(x) = \frac {e^{\Phi(q) x}} {\psi'(\Phi(q))} - \sum_{i \in \mathcal{I}_q} C_{i,q} e^{-\xi_{i,q}x},
\end{align*}
where
\begin{align*}
C_{i,q} &:= \left. \frac { s+\xi_{i,q}} {q-\psi(s)} \right|_{s = -\xi_{i,q}} = - \frac 1 {\psi'(-\xi_{i,q})};
\end{align*}
see \cite{Egami_Yamazaki_2010_2}.  Here $\{ \xi_{i,q}; i \in \mathcal{I}_q \}$ and  $\{ C_{i,q}; i \in \mathcal{I}_q \}$ are possibly complex-valued.

In our example, we assume $m = 6$ and
\begin{align*}
&{\bm T} = \left[ \begin{array}{rrrrrr}   -5.7714 &   0.0881 &   0.0080 &   0.0031  &  0.0002  &  0.0036 \\
    0.0000  & -6.3823 &   0.0031 &   0.0000   & 0.0000  &  6.3793 \\
         0.0000   & 5.8540  & -5.8540    &     0.0000  &  0.0000 &   0.0000 \\
    6.6551  &  0.0353  &  1.4502  & -8.1408  &  0.0000 &   0.0003 \\
    0.0000  &  0.0007 &   0.0054  &  6.0959 &  -6.1029 &   0.0009 \\
    0.0000   & 0.0092  &  0.0049 &   0.0000 &   6.4107 &  -6.4249 \end{array} \right], \quad {\bm \alpha} =    \left[ \begin{array}{l}            0.0000 \\
    0.0000 \\
    1.0000 \\
         0.0000 \\
         0.0000 \\
    0.0000  \end{array} \right],
\end{align*}
which give an approximation of the log-normal distribution with density function
\begin{align*}
f(x) = \frac 2 {x \sqrt{2 \pi}} \exp \left\{ -  {2(\log x)^2}  \right\}, \quad x > 0,
\end{align*}
obtained using the EM-algorithm; see \cite{Egami_Yamazaki_2010_2} regarding the approximation performance of the corresponding scale function.
Throughout this section, we let $q = 0.05$ and $\delta = \kappa = 2.5$.  We consider both the bounded and unbounded variation cases with $\sigma = 0$ and $\sigma = 1$.

\begin{figure}[htbp]
\begin{center}
\begin{minipage}{1.0\textwidth}
\centering
\begin{tabular}{cc}
 \includegraphics[scale=0.58]{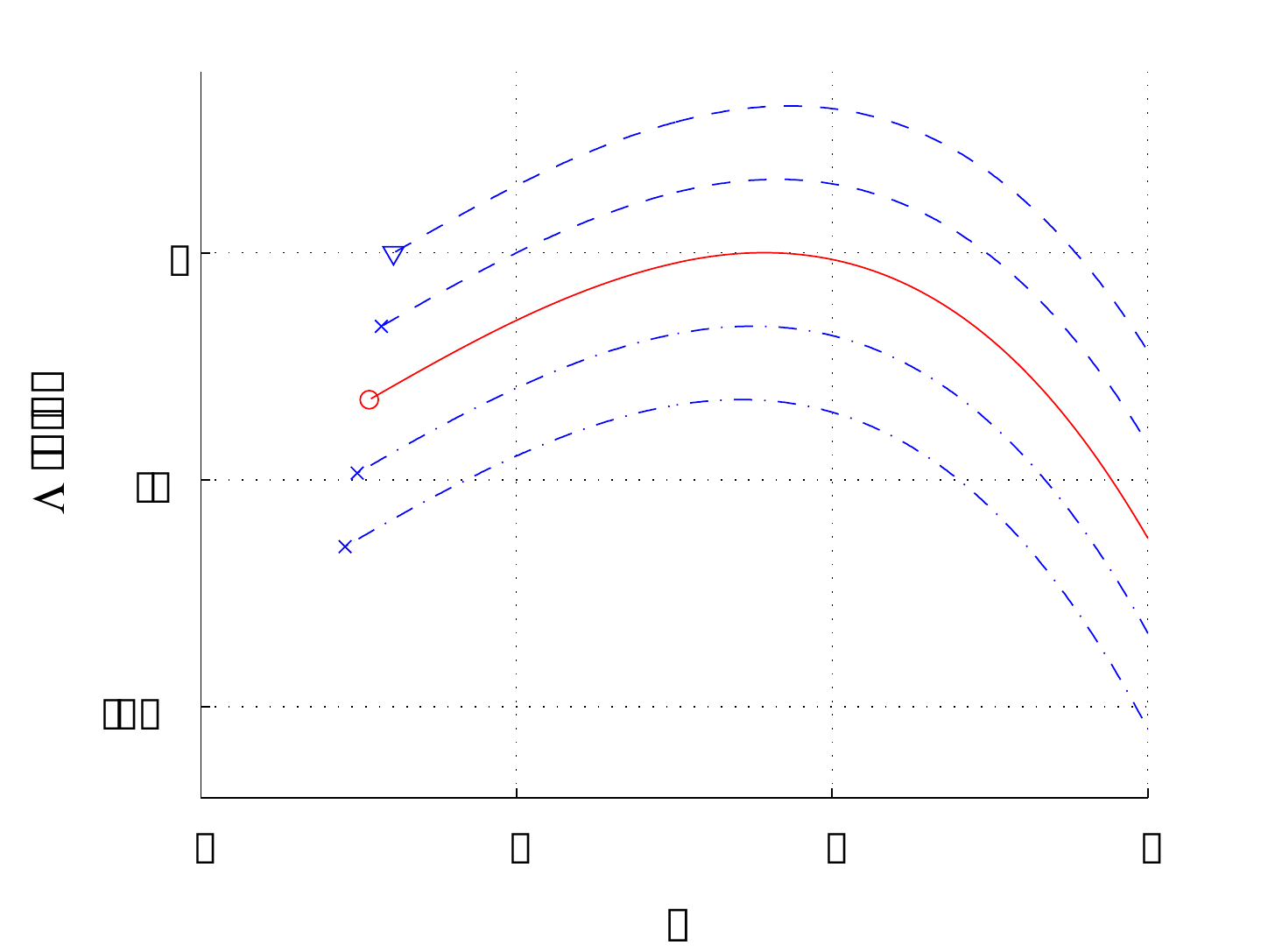} & \includegraphics[scale=0.58]{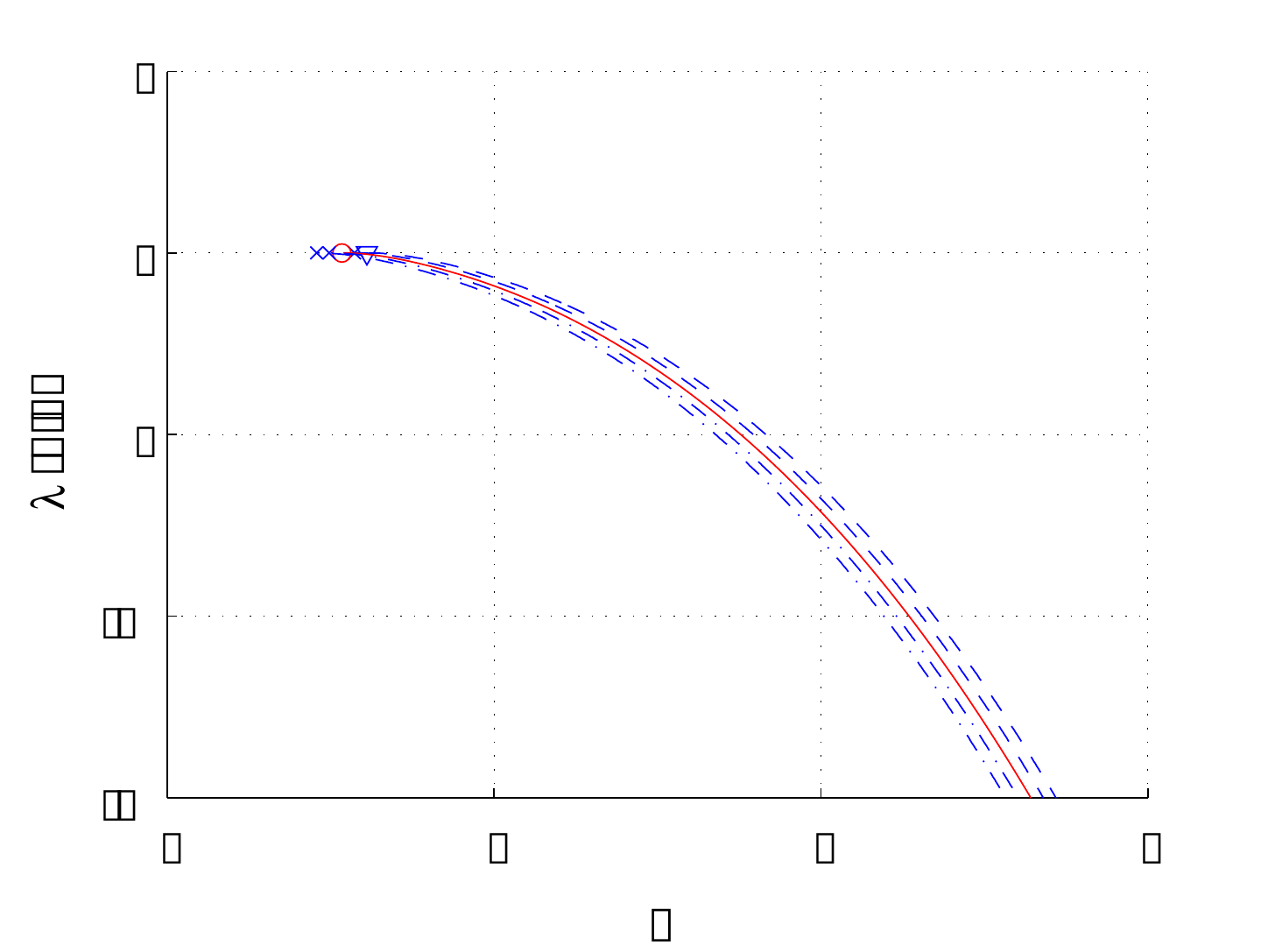}  \\
\multicolumn{2}{c}{Unbounded variation case with $\sigma = 1$}\vspace{0.3cm} \\
 \includegraphics[scale=0.58]{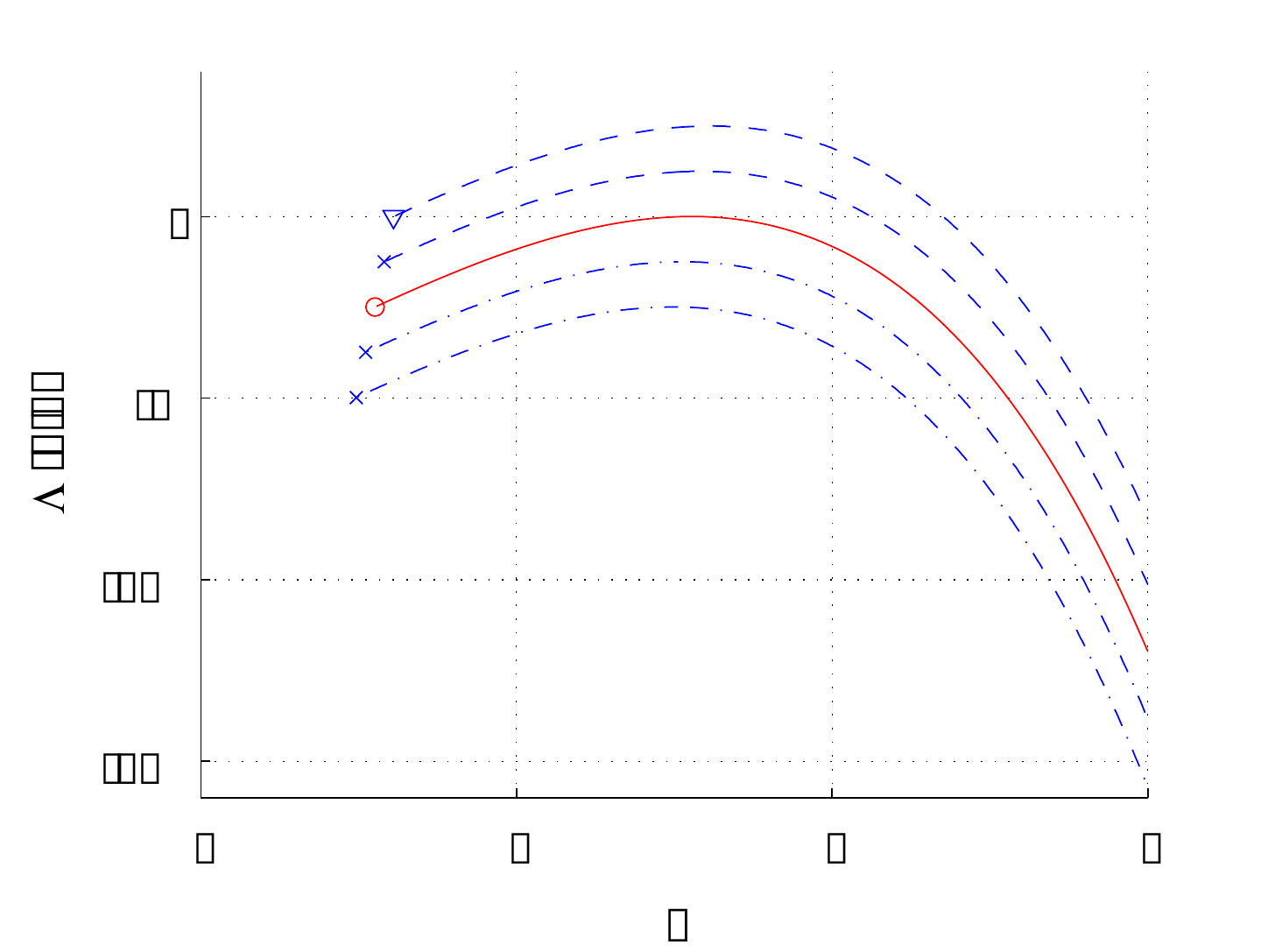} & \includegraphics[scale=0.58]{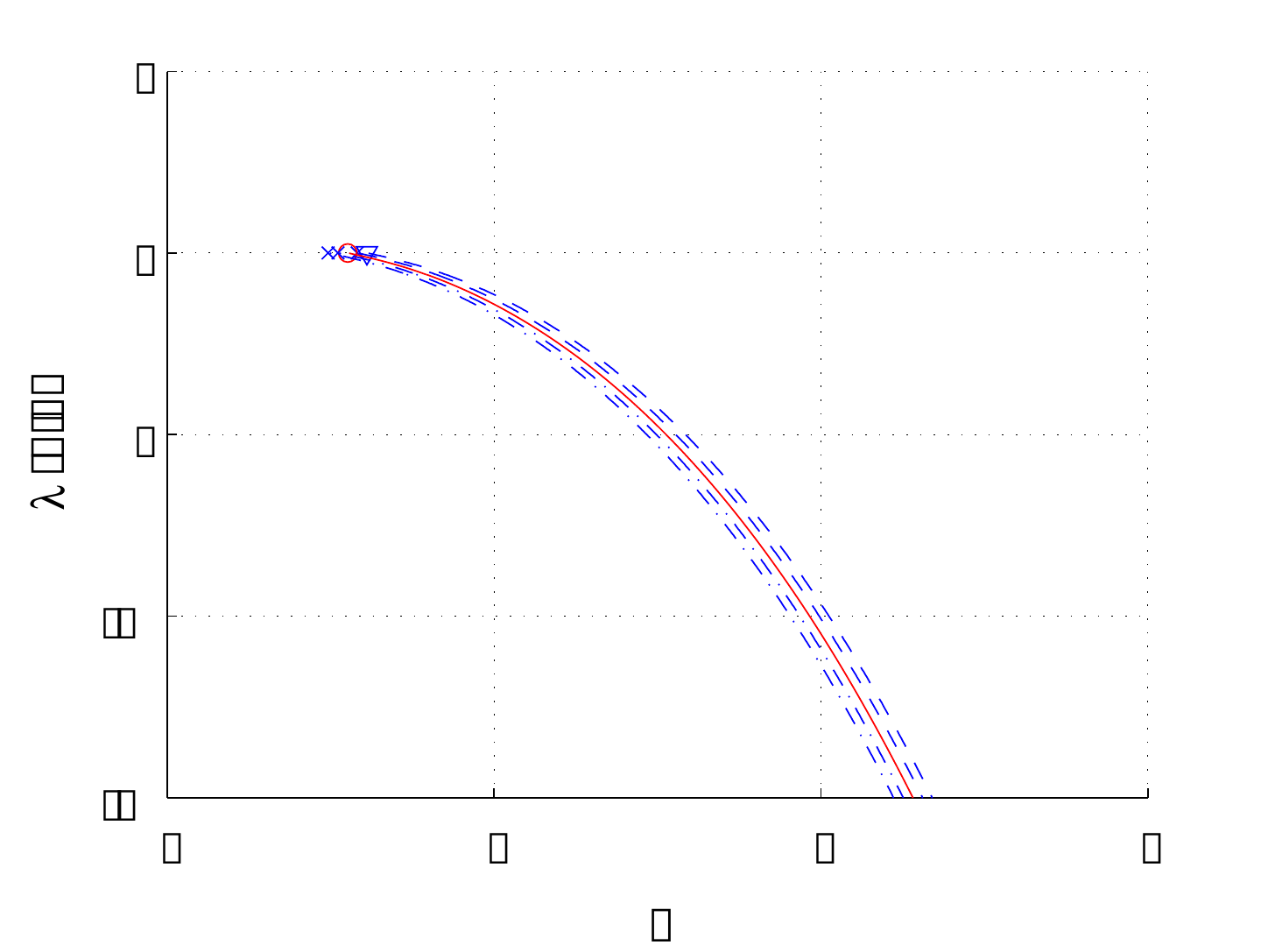}  \\
\multicolumn{2}{c}{Bounded variation case with $\sigma = 0$}\vspace{0.3cm} \\
\end{tabular}
\end{minipage}
\caption{$\Lambda(a,b)$ (left) and $\lambda(a,b)$ (right) as a function of $b$ for $a =2 a^* - \bar{a}, a^* + (a^*-\bar{a})/2, a^*, (\bar{a}+a^*)/2,  \bar{a}$. Both $\Lambda(a,b)$ and $\lambda(a,b)$ are monotonically increasing in $a$; the ones starting with inverted triangles are for $\bar{a}$ and the ones with circles are for $a^*$.} \label{figure_Lambda}
\end{center}
\end{figure}

\begin{figure}[htbp]
\begin{center}
\begin{minipage}{1.0\textwidth}
\centering
\begin{tabular}{cc}
 \includegraphics[scale=0.58]{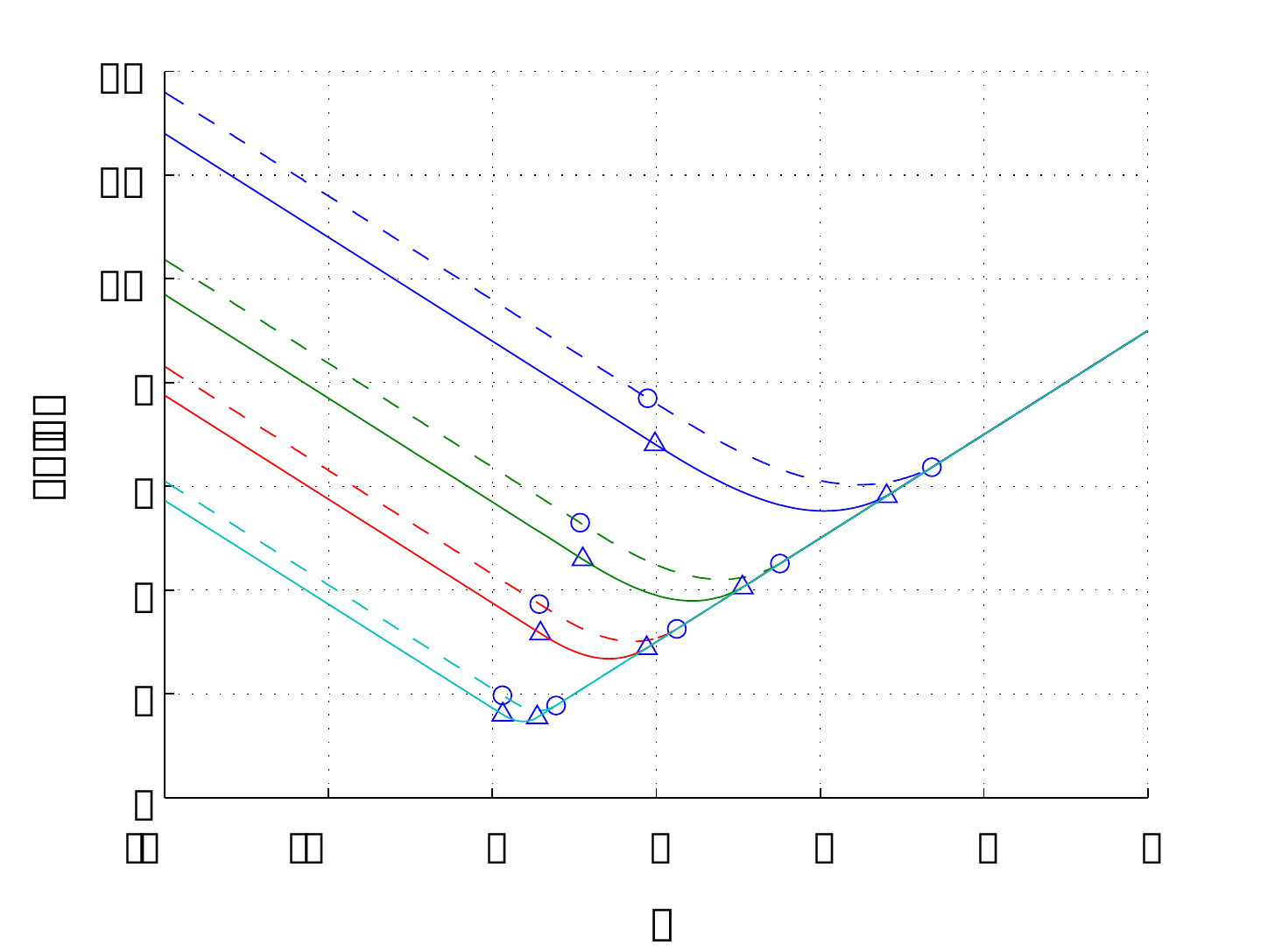} & \includegraphics[scale=0.58]{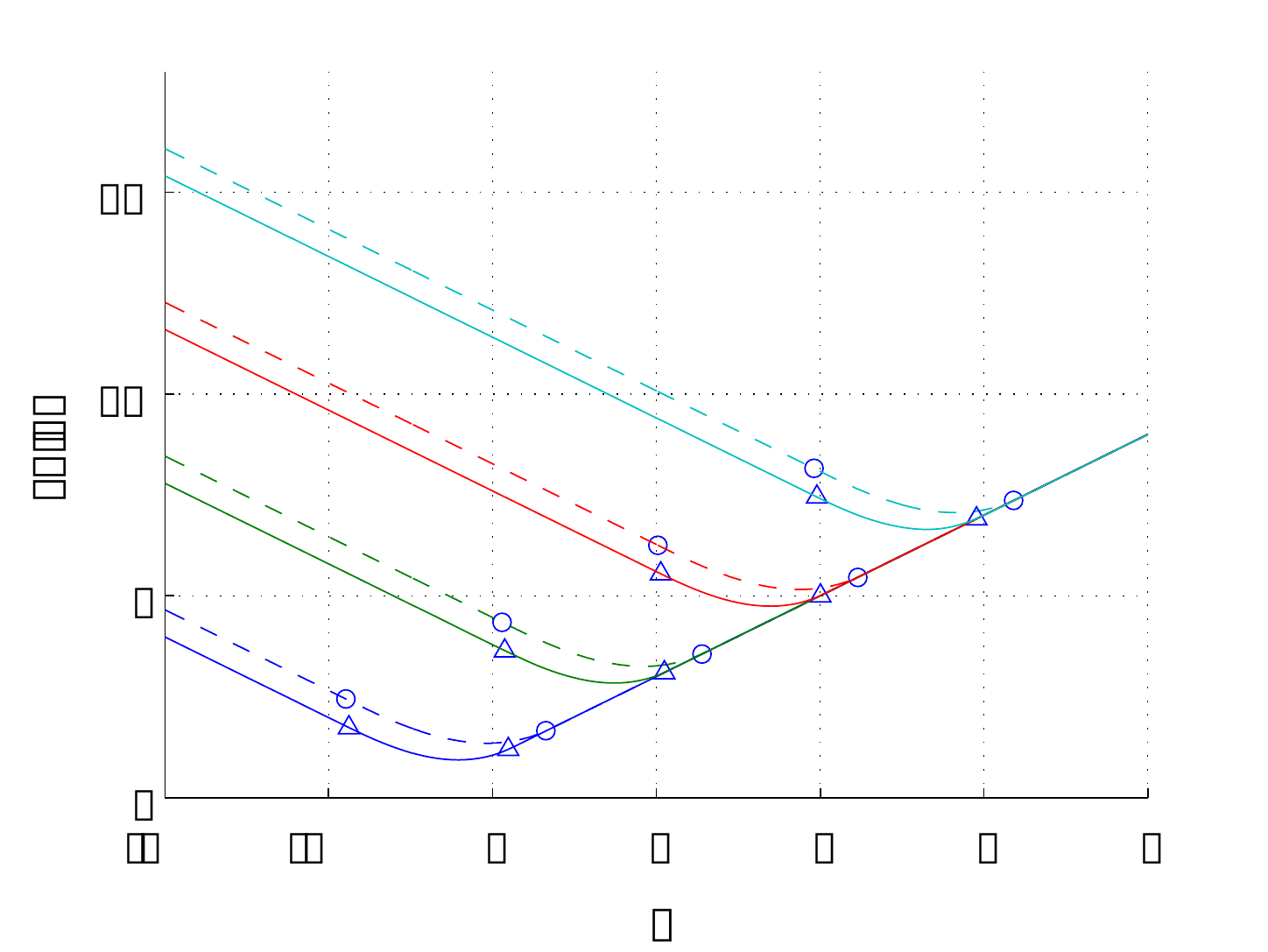}  \\
$\alpha =0.5,1,2,10$ &$\beta = -2,0,2,4$ \vspace{0.3cm} \\
 \includegraphics[scale=0.58]{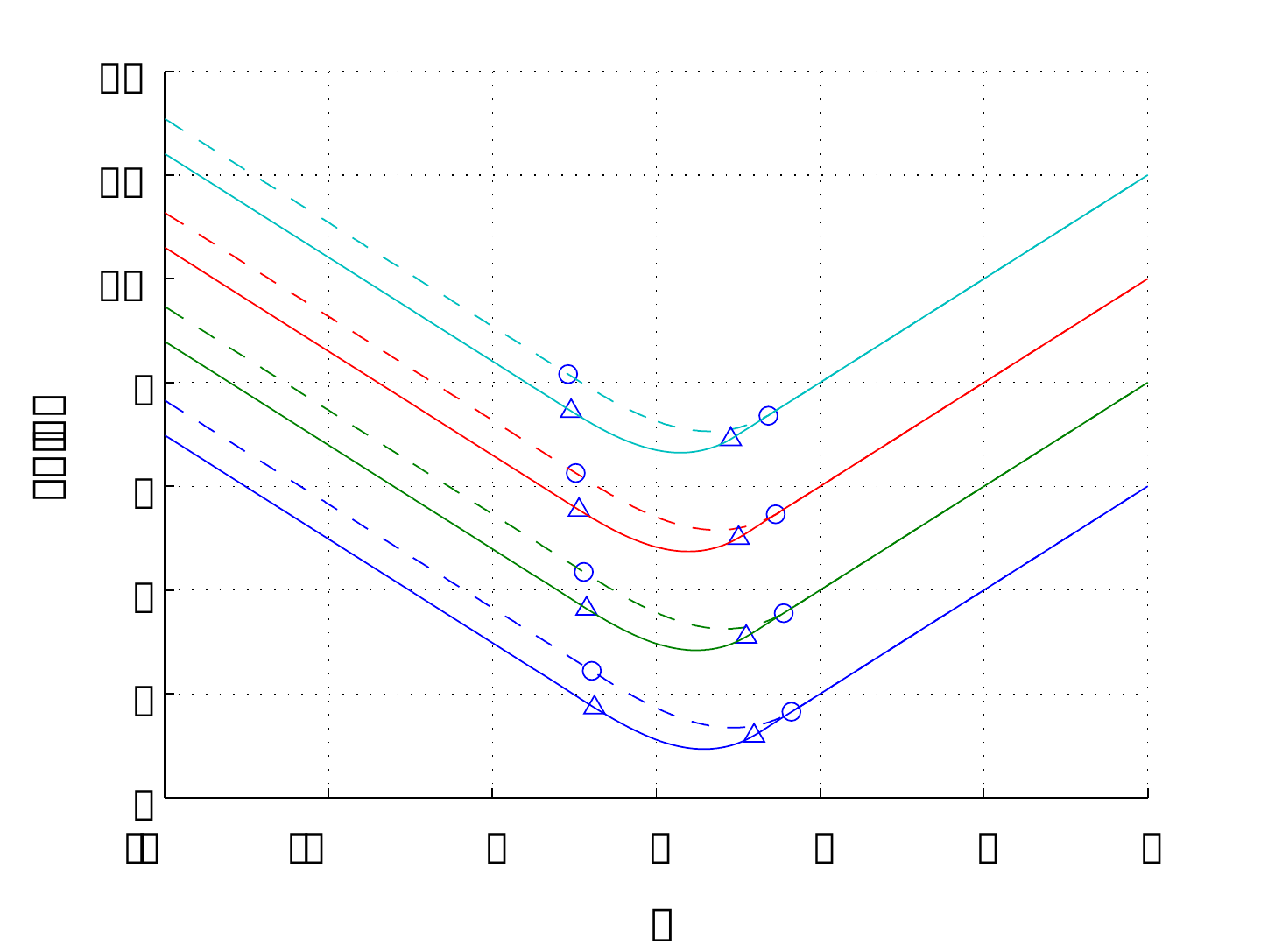} & \includegraphics[scale=0.58]{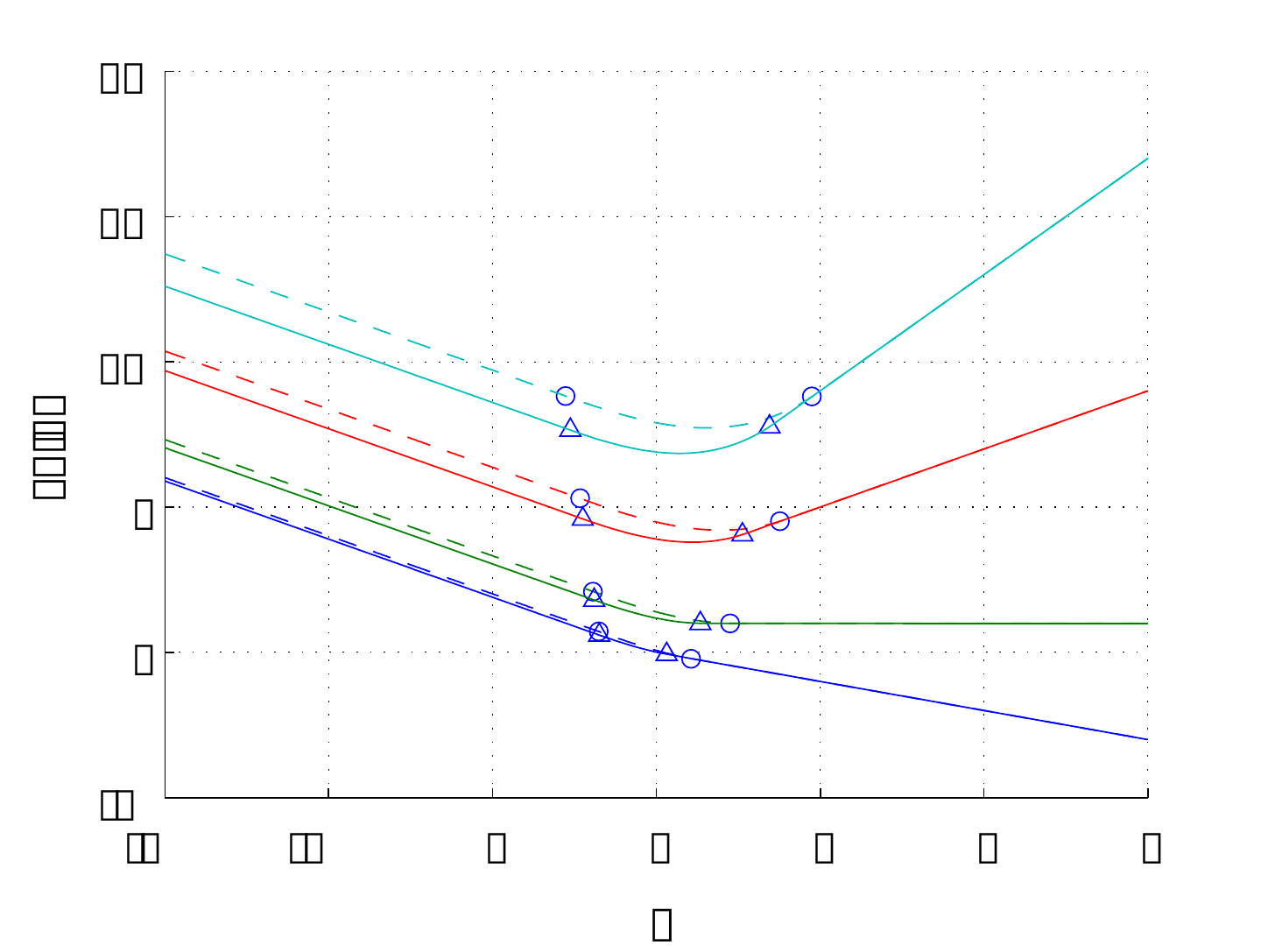}  \\
 $C = -2,0,2,4$&  $K =-0.5,0,1,2$ \vspace{0.3cm} \\
\end{tabular}
\end{minipage}
\caption{The value functions for the spectrally negative \lev case. } \label{figure_value_negative}
\end{center}
\end{figure}

\subsection{Spectrally negative \lev case}  We first consider the spectrally negative \lev case as studied in Section  \ref{section_spectrally_negative}.  Recall that the values of $a^*$ and $b^*$ are such that $\Lambda(a^*,b^*) = \lambda(a^*, b^*) = 0$. Figure \ref{figure_Lambda} shows, for both the unbounded and bounded variation cases,  $\Lambda(a,\cdot): b\mapsto \Lambda(a,b)$ and  $\lambda(a,\cdot): b\mapsto \lambda(a,b)$ for fixed  $a =2 a^* - \bar{a}, a^* + (a^*-\bar{a})/2, a^*, (\bar{a}+a^*)/2,  \bar{a}$.  Regarding the computation of the pair $(a^*,b^*)$, we observe, by \eqref{lambda_derivative_b} and \eqref{lambda_small_starting}, that the function $\lambda(a,\cdot)$ for each fixed $a$ starts at a positive value and monotonically decreasing in $b$ to $-\infty$ and hence its zero $\tilde{b}(a)$ (see \eqref{def_b_of_a_SN}) can be computed via bisection method. Recalling that $\Lambda(a,\tilde{b}(a))$ is increasing in $a$, we then apply another bisection method to $\{\Lambda(a,\tilde{b}(a)); a \geq 0\}$ to compute $(a^*,b^*)$.  Once these are obtained, the value function is computed by \eqref{def_value_function}.

In Figure \ref{figure_value_negative}, we show how the value functions change in each of the parameters $\alpha, \beta, C$ and $K$.  The parameters we use are $\alpha=\beta=C=K=1$ unless otherwise specified.  In each figure, the bounded variation case  ($\sigma =0$) is plotted in solid while the unbounded variation case ($\sigma = 1$) is in dotted.  The circles (resp.\ triangles) indicate the values at $a^*$ and $b^*$ for the case of unbounded (resp.\ bounded) variation.   We can confirm the results in Proposition \ref{proposition_smoothness}  that the value functions are smooth in all cases.  In addition, the value of the game is monotone in each parameter.  Finally, the value function is indeed convex for both the bounded and unbounded variation cases; this is consistent with Proposition \ref{convexity_SN}.

\subsection{Spectrally positive \lev case}  We now move onto the spectrally positive \lev case as studied in Section  \ref{section_spectrally_positive}. Recall that the values of $a^*$ and $b^*$
such that $\Gamma(a^*, b^*)/W^{(q)'}(b^*-a^*) = 0$ and $a^* := \arg \inf_{a \leq b^*} \Gamma(a, b^*)$ are first attained when we increase the value of $b$ starting from $\underline{b}$.  When $X$ is of unbounded variation or $\Upsilon(B) > 0$, we always have $a^* < b^*$ with $\Gamma(a^*,b^*) = \gamma(a^*,b^*) = 0$; otherwise $a^*=b^*=B$ with $\Gamma(a^*,b^*) = 0$.  In both cases, because $\Gamma(\cdot, b)$ has at most one local minimum (in other words, $\gamma(\cdot, b)$ has at most one zero), the values of $(a^*, b^*)$ can be obtained by bisection methods similarly to the spectrally negative \lev case as obtained above.

In Figure \ref{figure_Gammas}, we show  $\Gamma(b, \cdot)$ and  $\gamma(b, \cdot)$ for fixed $b = \underline{b}, (\underline{b}+b^*)/2,  b^*, b^* + (b^*-\underline{b})/2, 2 b^* - \underline{b}$ for (i) $\sigma = 1$ and $\alpha = 1$, (ii) $\sigma = 0$ and $\alpha = 1$ and (iii) $\sigma = 0$ and $\alpha = 10$.  We choose these values of $\alpha$ so that $a^* < b^*$ in (i) and (ii) while $a^*=b^* = B$ in (iii).
In Figure \ref{figure_with_diffusion}, we show the value functions in the same fashion as in Figure \ref{figure_value_negative}.  From our choice of $\alpha$ such that $\alpha > q$, in all the results $a^* > -\infty$. Here, for the bounded variation case with $\alpha = 10$ in the top left plot, $a^* = b^* = B$; in other cases we have $a^* < b^*$.   The results are consistent with Proposition \ref{prop_smoothness_spec_pos} that the value function is twice differentiable at $a^*$ (unless $a^* = b^*$), and at $b^*$ it is continuous (resp.\ differentiable) when $X$ is of bounded (resp.\ unbounded) variation.  We also confirm the convexity of the value function as in Proposition \ref{convexisty_SP}.

\begin{figure}[htbp]
\begin{center}
\begin{minipage}{1.0\textwidth}
\centering
\begin{tabular}{cc}
 \includegraphics[scale=0.58]{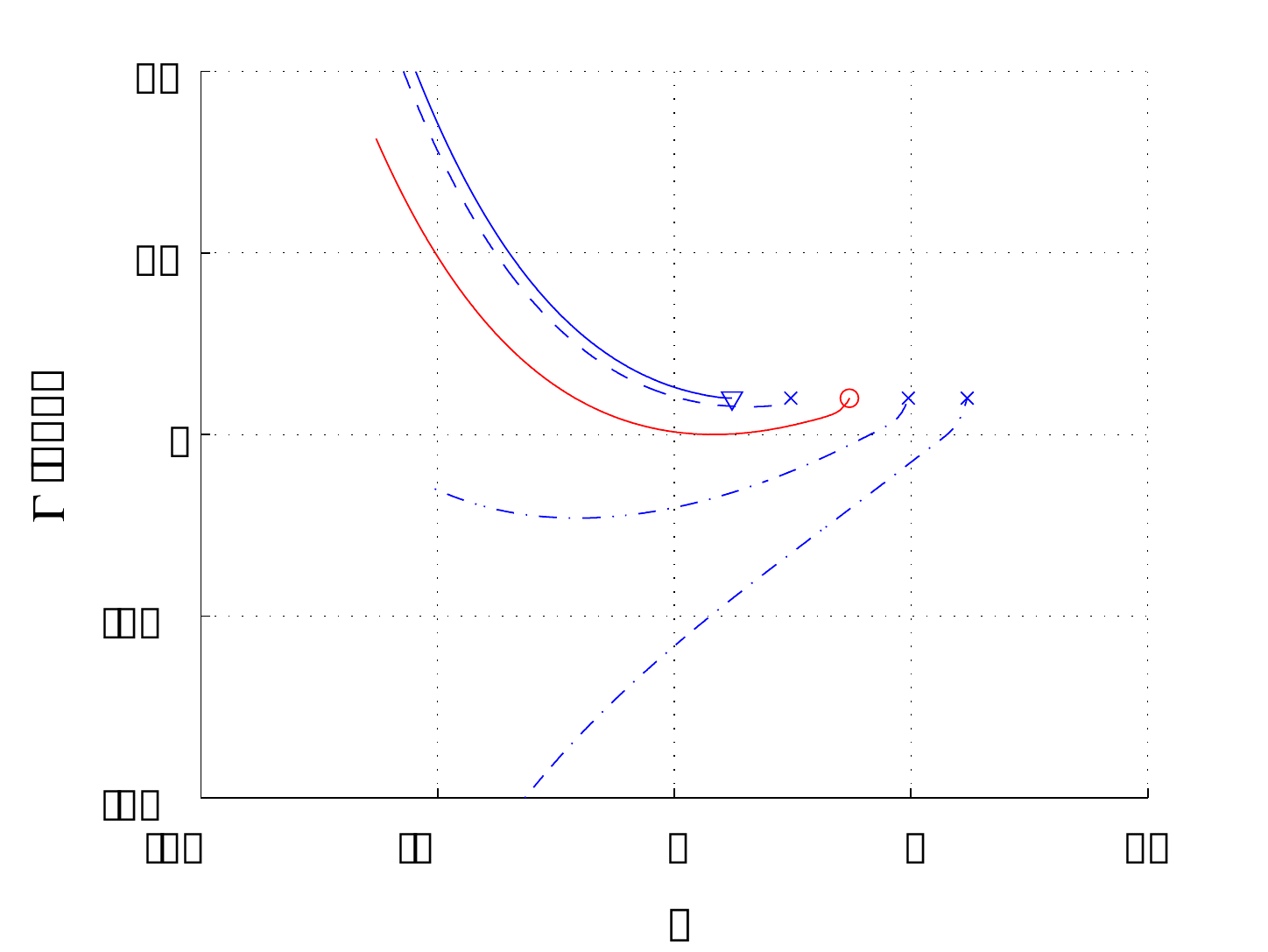} & \includegraphics[scale=0.58]{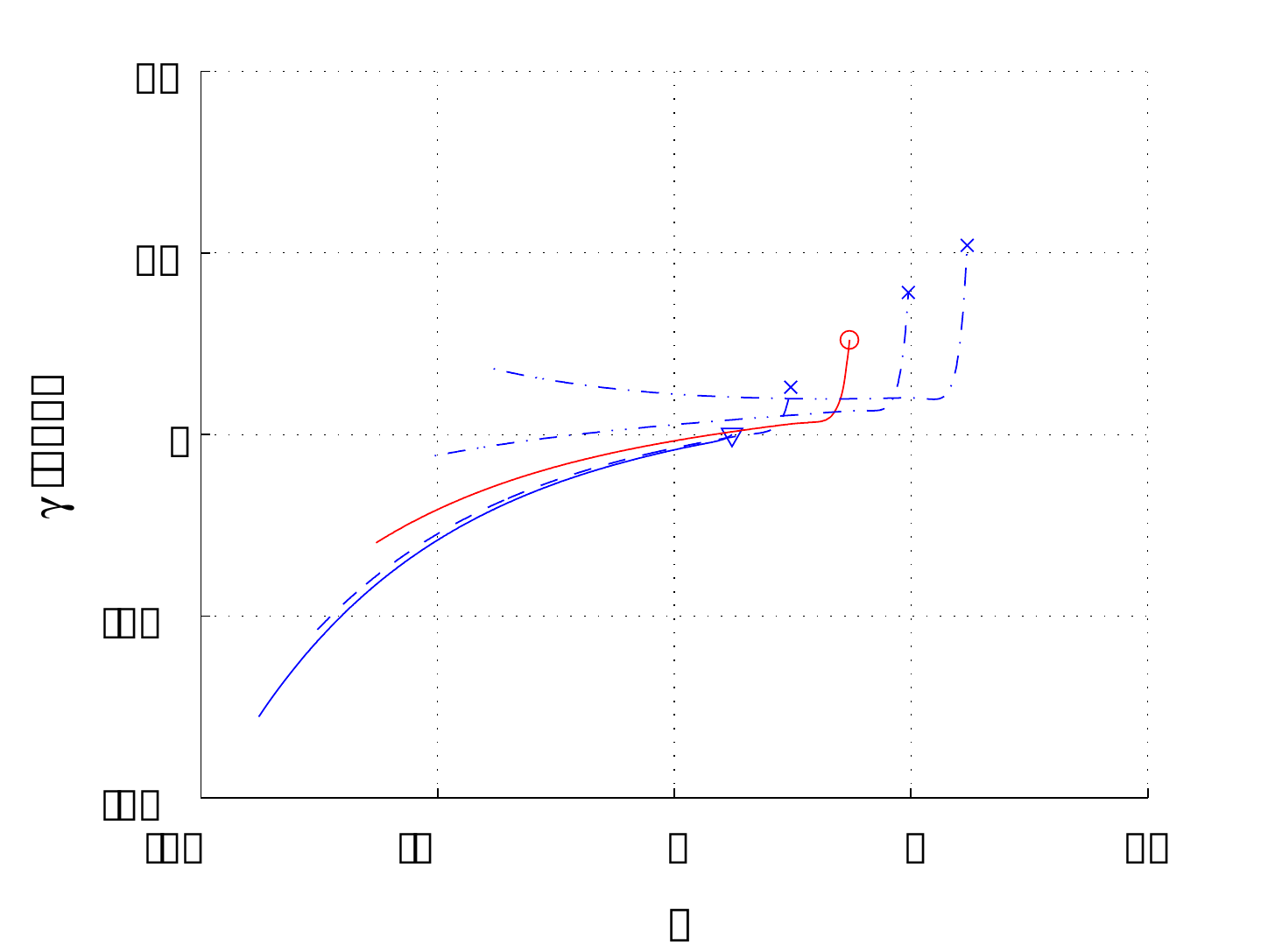}  \\
\multicolumn{2}{c}{(i) $\sigma = 1$ and $\alpha = 1$}\vspace{0.3cm} \\
 \includegraphics[scale=0.58]{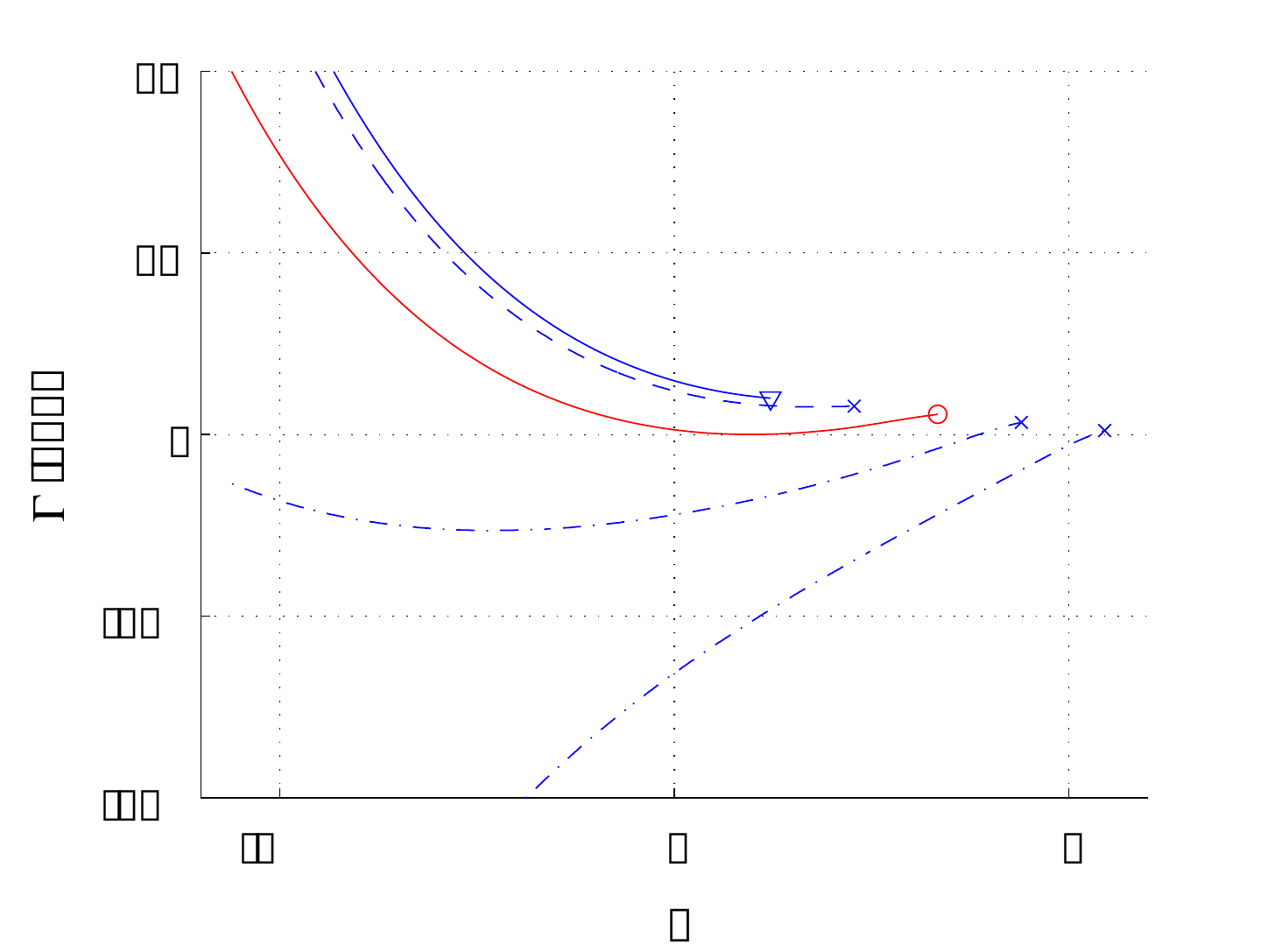} & \includegraphics[scale=0.58]{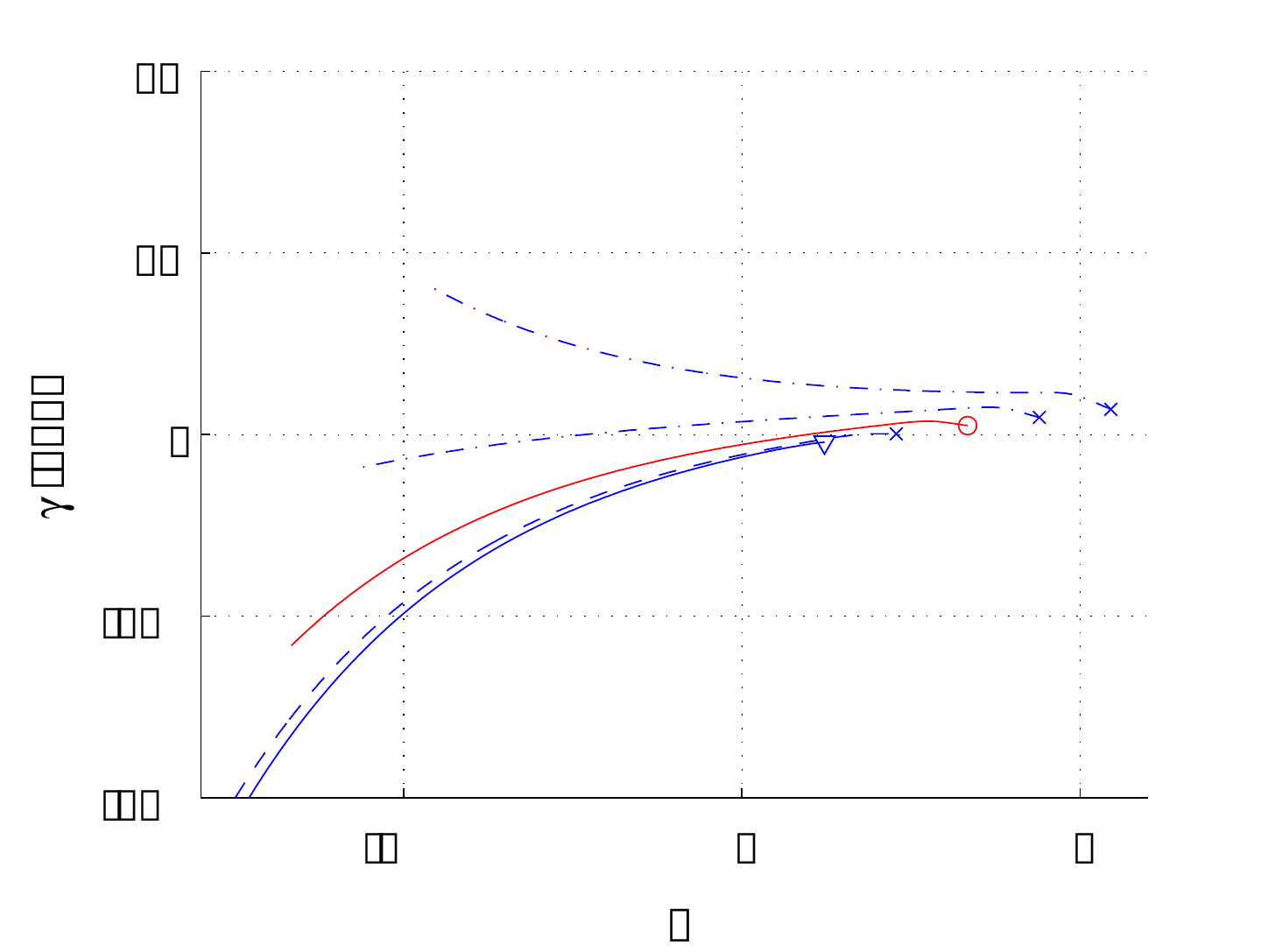}  \\
\multicolumn{2}{c}{(ii) $\sigma = 0$ and $\alpha = 1$}\vspace{0.3cm} \\
 \includegraphics[scale=0.58]{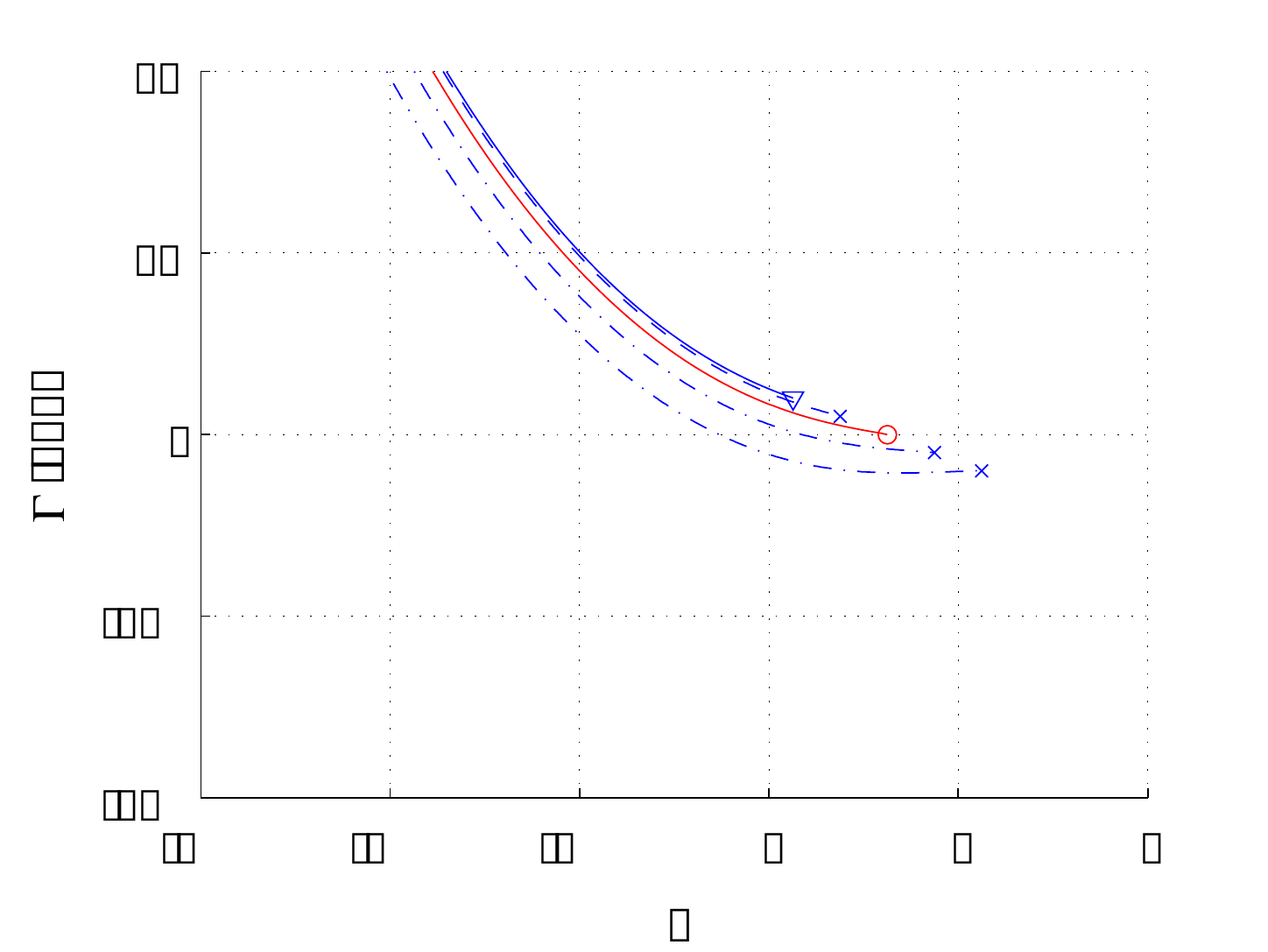} & \includegraphics[scale=0.58]{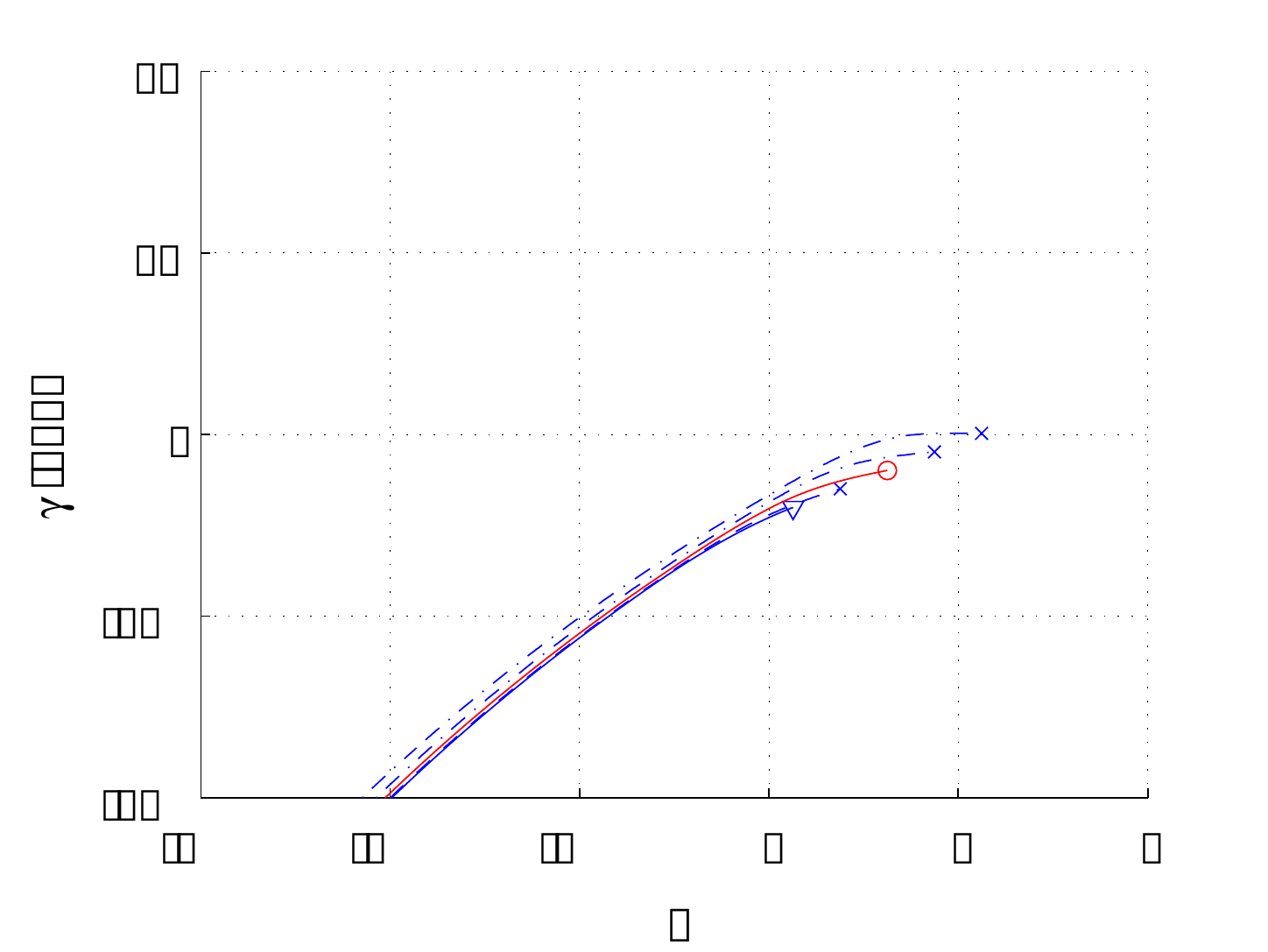}  \\
\multicolumn{2}{c}{(iii) $\sigma = 0$ and $\alpha = 10$} \vspace{0.3cm} \\
\end{tabular}
\end{minipage}
\caption{$\Gamma(a,b)$ (left) and $\gamma(a,b)$ (right) as a function of $a$ for $b = \underline{b}, (\underline{b}+b^*)/2,  b^*, b^* + (b^*-\underline{b})/2, 2 b^* - \underline{b}$. The ones starting with inverted triangles are for $\underline{b}$ and the ones with circles are for $b^*$.} \label{figure_Gammas}
\end{center}
\end{figure}

\begin{figure}[htbp]
\begin{center}
\begin{minipage}{1.0\textwidth}
\centering
\begin{tabular}{cc}
 \includegraphics[scale=0.58]{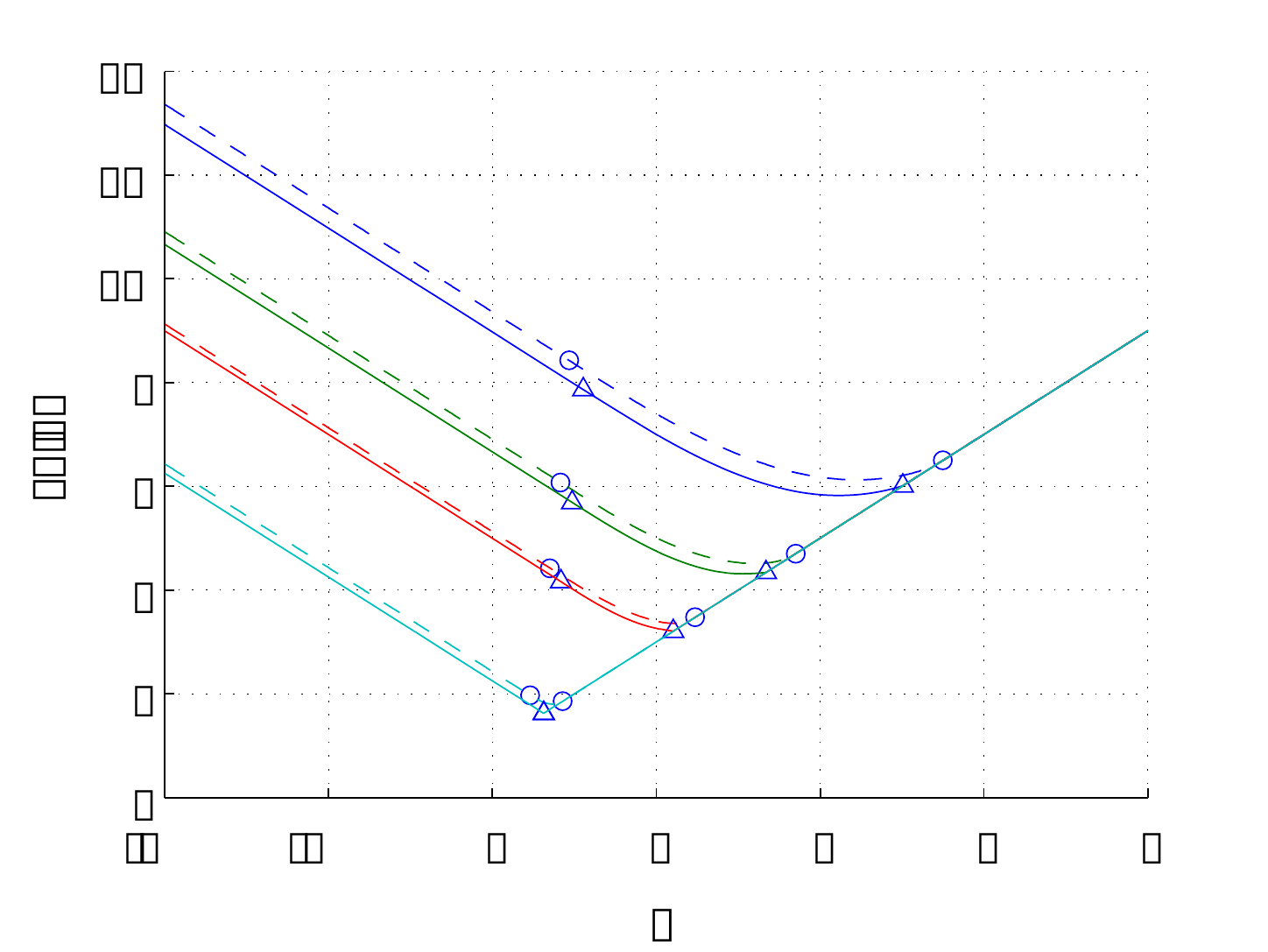} & \includegraphics[scale=0.58]{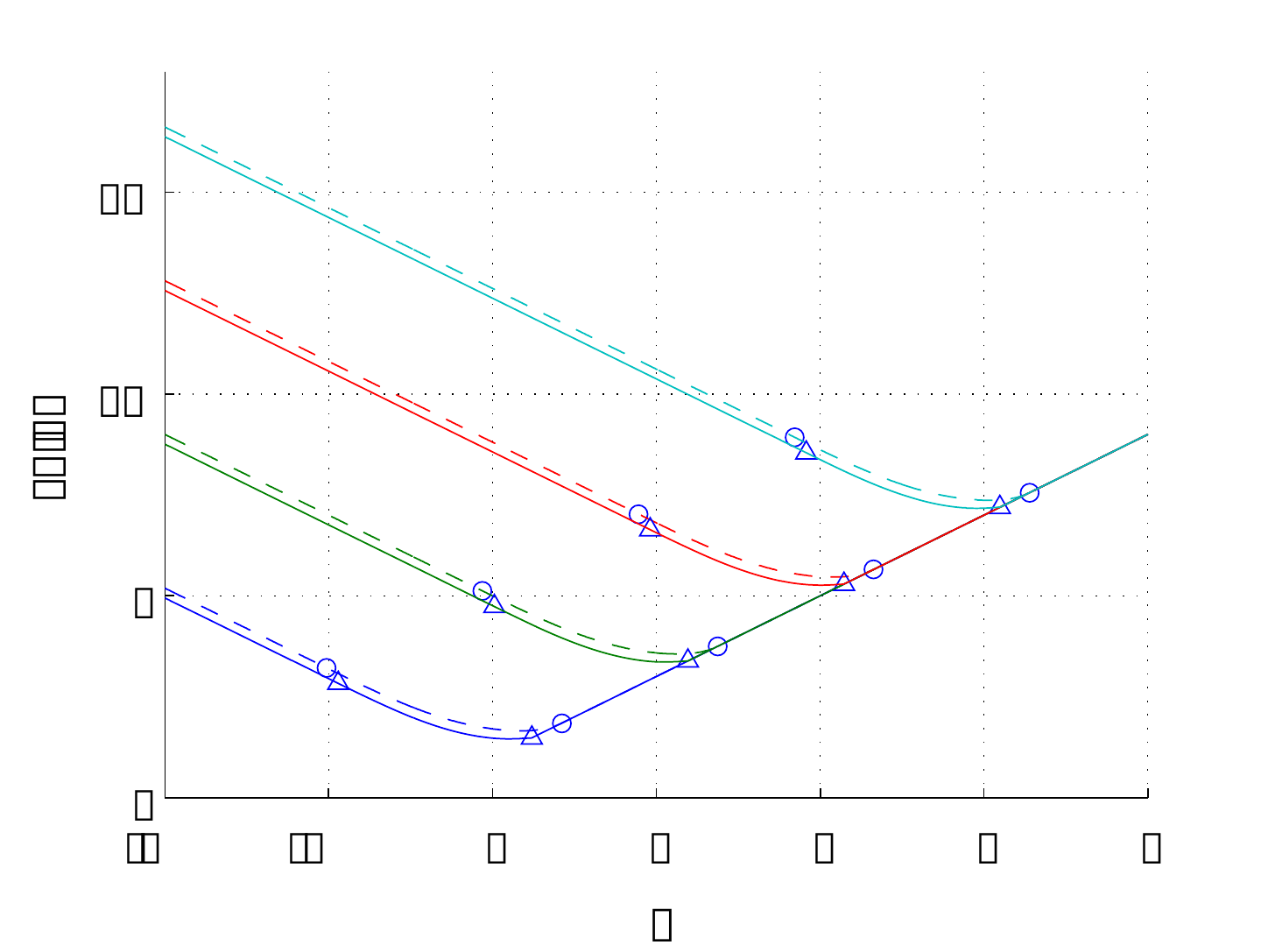}  \\
$\alpha =0.5,1,2,10$ &$\beta = -2,0,2,4$ \vspace{0.3cm} \\
 \includegraphics[scale=0.58]{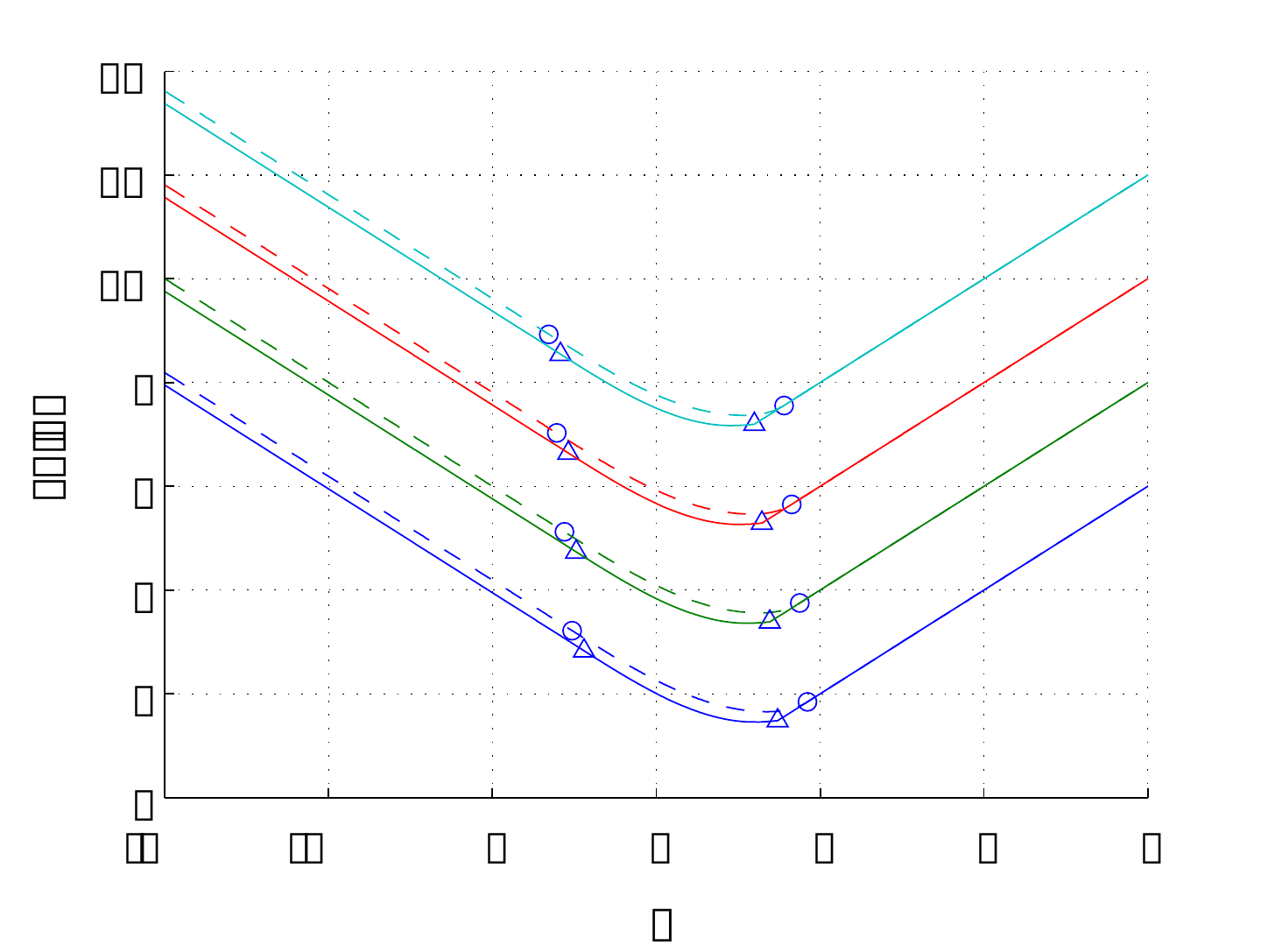} & \includegraphics[scale=0.58]{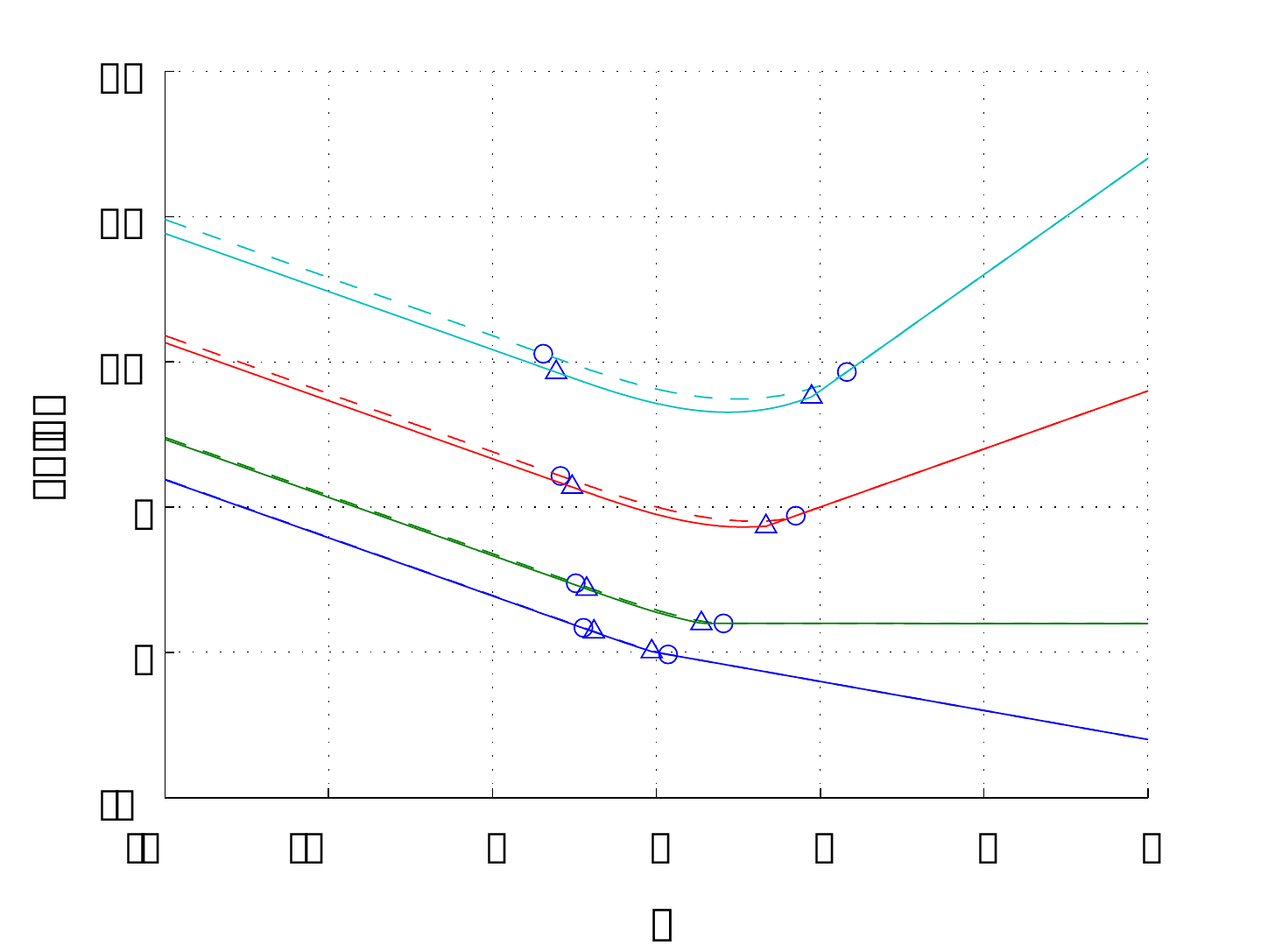}  \\ $C = -2,0,2,4$&  $K =-0.5,0,1,2$  \vspace{0.3cm} \\
\end{tabular}
\end{minipage}
\caption{The value functions for the spectrally positive \lev case. }  \label{figure_with_diffusion}
\end{center}
\end{figure}

\section*{Acknowledgements} The authors thank the two anonymous referees for their thorough reviews and insightful comments that help improve the presentation of the paper. The first author is indebted  to V.\ Rivero for  helpful discussions; the support received from Conacyt through the Laboratory LEMME is also greatly appreciated. The second author thanks M.\ Fukushima and H.\ Nagai for valuable comments and is supported by MEXT KAKENHI grant numbers  22710143 and 26800092, JSPS KAKENHI grant number 23310103, the Inamori foundation research grant, and the Kansai University subsidy for supporting young scholars 2014.

\bibliographystyle{abbrv}
\bibliographystyle{apalike}

\bibliographystyle{agsm}
\bibliography{dual_model_bib}

\end{document}